\documentclass[11pt]{amsart}
\usepackage{amsmath, amsfonts, amssymb,color}
\usepackage{stmaryrd,epsfig}
\usepackage{xypic}
\usepackage{amssymb}
\usepackage{graphics}
\usepackage{graphicx}
\usepackage{latexsym}
\usepackage{amsmath}
\usepackage{amssymb}
\usepackage{amscd}
\usepackage{multicol}
\usepackage{leftidx}
\usepackage{tikz}

\numberwithin{equation}{section}

\newcommand{\e}{\mathrm{e}}

\newcommand{\ep}{\varepsilon}

\newcommand{\cT}{\mathcal{T}}

\newcommand{\bZ}{\mathbb{Z}}
\newcommand{\bR}{\mathbb{R}}
\newcommand{\bC}{\mathbb{C}}

\newcommand{\g}{\mathfrak{g}}
\newcommand{\gaf}{\mathfrak{g}_{\rm aff}}

\newcommand{\mfu}{\mathfrak{u}}
\newcommand{\mfb}{\mathfrak{b}}
\newcommand{\h}{\mathfrak{h}}

\newcommand{\he}{{\rm ht}}

\newcommand{\Fla}{{\mathcal F\ell}_{G}}

\newcommand{\af}{{\rm aff}}

\newcommand{\Lca}{\mathcal{L}}

\newcommand{\eps}{\varepsilon}
\newcommand{\epsri}{\sigma}

\newcommand{\hc}{{\mathfrak{h}}}
\newcommand{\cB}{\mathcal{B}}
\newcommand{\cG}{\mathcal{G}}

\DeclareMathOperator{\coh}{H}
\DeclareMathOperator{\quantum}{QH}

\newtheorem{thm}{Theorem}[section] 
\newtheorem{lemma}[thm]{Lemma}
\newtheorem{prop}[thm]{Proposition}
\newtheorem{cor}[thm]{Corollary}

\theoremstyle{definition}
\newtheorem{remark}[thm]{Remark}
\newtheorem{example}[thm]{Example}
\newtheorem{defn}[thm]{Definition}

\DeclareMathOperator{\Fl}{Fl}

\DeclareMathOperator{\GL}{GL}
\DeclareMathOperator{\SL}{SL}

\DeclareMathOperator{\Span}{Span}

\DeclareMathOperator{\QH}{QH}

\renewcommand{\P}{{\mathbb P}}
\newcommand{\C}{{\mathbb C}}
\newcommand{\Z}{{\mathbb Z}}
\newcommand{\Q}{{\mathbb Q}}

\newcommand{\cO}{{\mathcal O}}

\newcommand{\gw}[2]{\langle #1 \rangle^{\mbox{}}_{#2}}

\newcommand{\ev}{\operatorname{ev}}

\newcommand{\wb}{\overline}

\newcommand{\ignore}[1]{}

\newcommand{\Mb}{\wb{\mathcal M}}

\begin{document}

\setcounter{tocdepth}{1}
\setcounter{topnumber}{1}
\setcounter{bottomnumber}{1}

\title[An affine  quantum cohomology ring for flag manifolds]{An affine quantum cohomology ring for flag manifolds and the periodic Toda lattice}
\author[A.-L.~Mare and L.~C.~Mihalcea]{Augustin-Liviu Mare and Leonardo C.~Mihalcea}
\address{ Department of Mathematics and Statistics, University of Regina,
Regina, SK Canada S4S 0A2}
\email{mareal@math.uregina.ca}
\address{Department of  Mathematics, 460 McBryde Hall, Virginia Tech University, Blacksburg, VA 24061 USA } 
\email{ lmihalce@math.vt.edu}
\date{May 5, 2016}
\subjclass[2010]{Primary 14N35; Secondary 14M15, 17B67, 37K10, 37N20}

\thanks{L.~C.~Mihalcea was supported in part by NSA Young
  Investigator Award H98230-13-1-0208 and a Simons Collaboration Grant}

\begin{abstract} Consider the generalized flag manifold $G/B$ and the corresponding affine flag manifold $\Fla$. In this paper we use curve neighborhoods for Schubert varieties in $\Fla$ to construct certain affine Gromov-Witten invariants of $\Fla$, and to obtain a family of ``affine quantum Chevalley" operators $\Lambda_0, \ldots, \Lambda_n$ indexed by the simple roots in the affine root system of $G$. These operators act on the cohomology ring $\coh^*(\Fla)$ with coefficients in $\Z[q_0, \ldots,q_n]$. By analyzing commutativity and invariance properties of these operators we deduce the existence of two quantum cohomology rings, which satisfy properties conjectured earlier by Guest and Otofuji for $G= \SL_n(\C)$. The first quantum ring is a deformation of the subalgebra 
of $\coh^*(\Fla)$ generated by divisors. The second ring, denoted $\quantum^*_\af(G/B)$, deforms the ordinary quantum cohomology ring $\quantum^*(G/B)$ by adding an affine quantum parameter $q_0$. We prove that $\quantum^*_\af(G/B)$ is a Frobenius algebra, and that the new quantum product determines a flat Dubrovin connection. Further, we develop an analogue of Givental and Kim formalism for this ring and we deduce a presentation of $\quantum^*_\af(G/B)$ by generators and relations. The ideal of relations is generated by the integrals of motion for the periodic Toda lattice associated to the dual of the extended Dynkin diagram of $G$.
\end{abstract}
\maketitle

\tableofcontents
\section{Introduction}

Let $G$ be a simply connected, simple, complex Lie group and $B \subset G$ a Borel subgroup. Let $\cG:=\C^* \ltimes G(\C[t,t^{-1}])$ where $\C^*$ acts by loop rotation, and $\cB \subset \cG$ be the standard Iwahori subgroup of $\cG$ determined by $B$. Associated to this data there is the finite (generalized) flag manifold $G/B$ and the affine flag manifold $\Fla:=\cG/\cB$. The first is a finite dimensional complex projective manifold, while the second is an infinite-dimensional complex projective ind-variety. An influential result of Givental and Kim \cite{givental.kim} (for $G= \SL_{n+1}(\C)$) and Kim \cite{Kim} (for all $G$) states that the ideal of relations in the quantum cohomology ring $\quantum^*(G/B)$ of the generalized flag manifold $G/B$ is generated by the integrals of motion of the Toda lattice associated to the dual root system of $G$. Soon after that, Guest and Otofuji \cite{go,otofuji} (for $G=\SL_{n+1}(\C)$) and Mare \cite{Ma3} (for $G$ of Lie types $A_n, B_n, C_n$) assumed that there exists a (still undefined) quantum cohomology algebra $\quantum^*(\Fla)$ for the affine flag manifold $\Fla$, which satisfies the analogues of certain natural properties enjoyed by quantum cohomology, such as associativity, commutativity, the divisor axiom etc. The list of conjectured properties includes the fact that the quantum multiplication of Schubert divisor classes satisfies a quadratic relation determined by the Hamiltonian of the {\em dual periodic} Toda lattice. With these assumptions they proved that the ideal of relations in $\quantum^*(\Fla)$ is determined by the integrals of motion for the periodic Toda lattice associated to the dual of the extended Dynkin diagram of $G$.

In fact, Guest and Otofuji considered in \cite{go} two other rings related to the cohomology ring $\coh^*(\Fla)$. One is $\coh^*_\af(G/B)$, which is an isomorphic copy of the ordinary cohomology ring $\coh^*(G/B)$ inside $\coh^*(\Fla)$, induced by an evaluation morphism (we will explain this below). The other is the subring $\coh^\#(\Fla) \subset \coh^*(\Fla)$ generated by the Schubert divisor classes. It is well known that $\coh^*(G/B)$ is generated by the divisor classes (over $\Q$), but for affine flag manifolds $\coh^\#(\Fla)$ is a proper subring of  $\coh^*(\Fla)$. It was also conjectured
in \cite{go} that these two rings are closed under the hypothetical quantum product on $\coh^*(\Fla)$, and as a consequence the authors deduced that the integrals of motion of the periodic Toda lattice generate the ideal of relations in these rings.

The main goal of this paper is to rigorously define quantum products on the rings $\coh^*_\af(G/B)$ and $\coh^\#(\Fla)$, for $G$ of all Lie types, which will satisfy the analogues of the properties predicted in \cite{go,otofuji,Ma3}. We will then identify the ideal of relations in the quantum ring $\quantum^*_\af(G/B)$ determined by $\coh^*_\af(G/B)$ with the conserved quantities of the periodic Toda lattice associated to the dual of the extended Dynkin diagram for $G$ (these diagrams correspond to twisted affine Lie algebras). The definition of the quantum product involves the geometry of spaces of rational curves inside Schubert varieties of $\Fla$, as encoded in the ``curve neighborhoods" of (finite-dimensional) Schubert varieties $\Omega \subset \Fla$. The curve neighborhood $\Theta_d(\Omega)$ for a degree $d \in \coh_2(\Fla)$ is the subvariety of $\Fla$ containing the points on all rational curves $C \subset \Fla$ such that $C \cap \Omega \neq \emptyset$ and the degree of $C$ is $[C] = d$. Because $\Omega$ is finite dimensional so is $\Theta_d(\Omega)$, an observation which goes back to Atiyah \cite{atiyah}. Curve neighborhoods appeared in several recent studies of quantum cohomology and quantum K-theory rings of homogeneous spaces \cite{bcmp:qkfinite,buch.mihalcea:nbhds}.

\subsection{Statement of results} In what follows we give a more precise version of our results. We first fix some notations. Let $q= (q_0, \ldots, q_n)$ denote the sequence of quantum parameters indexed by the simple roots $\{ \alpha_0, \ldots, \alpha_n \}$ in the affine root system associated to $G$. Let $\deg q_i =2$ (see Remark \ref{rmk:deggq} for geometry behind this). The cohomology ring $\coh^*(G/B)$ is a graded $\Z$-algebra with a basis given by Schubert classes $\sigma_w \in \coh^{2 \ell(w)}(G/B)$, where $w \in W$ varies in the Weyl group of $G$. Similarly, the cohomology ring $\coh^*(\Fla)$ is a graded $\Z$-algebra with a basis given by affine Schubert classes $\eps_w \in \coh^{2 \ell(w)}(\Fla)$, where $w \in W_\af$ varies in the affine Weyl group of $G$. Here $\ell(w)$ denotes the length function for the Coxeter groups $W$ and $W_\af$. In what follows we consider complex degrees, which are half the topological degrees; e.g. $\deg \sigma_w = \ell(w)$. By $s_i$ we denote the simple reflection corresponding to the simple root $\alpha_i$. Let $\eps_i:= \eps_{s_i}$ and $\sigma_i:= \sigma_{s_i}$. 

Fix $d \in \coh_2(\Fla)$ an effective degree and write $d = d_0 \alpha_0^\vee + \cdots  + d_n \alpha_n^\vee$ where $d_i \in \Z_{\ge 0}$ and $\alpha_i^\vee$ are the simple affine coroots of $G$. Let $\{ \lambda_0, \ldots, \lambda_n \}$ be a set of fundamental weights dual to the coroot basis under the evaluation pairing $\langle \cdot, \cdot \rangle$, i.e. $\langle \lambda_i, \alpha_j^\vee \rangle = \delta_{ij}$ (the Kronecker symbol). By $q^d $ we will mean $q_0^{d_0} \cdots  q_n^{d_n}$; note that $\deg q^d = 2(d_0 + \cdots  + d_n)$. Fix also $s_i, u,w \in W_\af$. Consider the Schubert variety $X(w):= \overline{\cB w \cB/ \cB} \subset \Fla$. This is a projective variety of complex dimension $\ell(w)$. A key definition in this paper is that of the ``Chevalley" Gromov-Witten invariants $\gw{ \eps_{i}, \eps_u, [X(w)] }{d}$ where $[X(w)] \in \coh_{2 \ell(w)}(\Fla)$ is the fundamental class. Let $\Theta_d(w):= \Theta_d(X(w))$ be the curve neighborhood of $X(w)$, defined rigorously in Theorem \ref{T:existence} below. By definition $\gw{ \eps_{i}, \eps_u, [X(w)] }{d} = 0$ unless $1 + \ell(u) = \ell(w) + \deg q^d$. This is a familiar condition from quantum cohomology which is equivalent to the fact that the quantum multiplication is homogenous (or that the subspace of the moduli space of stable maps consisting of maps passing through classes represented by $\eps_i, \eps_u$ and $[X(w)]$ has expected dimension $0$). If $1 + \ell(u) = \ell(w) + \deg q^d$ then we define \[ \gw{ \eps_{i}, \eps_u, [X(w)] }{d} = \langle \lambda_i, d \rangle \int_{\Fla} \eps_u \cap [\Theta_d(w)] \/, \] where $\cap$ is the cap product between cohomology and homology and the integral means taking the degree $0$ homology component. This definition is the natural generalization of the analogous formula for the corresponding Gromov-Witten invariants of $G/B$, recently proved by Buch and Mihalcea in \cite{buch.mihalcea:nbhds}. Note  that the integral is nonzero only when $\dim \Theta_d(X(w)) = \ell(u)$, which is a very strong constraint on the possible degrees $d$ that may appear. We prove in \S \ref{section:rule} that $d= \alpha^\vee$, the dual of a positive affine real root $\alpha$ which satisfies $\ell(s_\alpha) = 2 \he (\alpha^\vee) -1$, where $\he$ denotes the height. { There are finitely many such coroots $\alpha^\vee$, because they} satisfy $\alpha^\vee < c:= \alpha_0^\vee + \theta^\vee$ where $c$ is the imaginary coroot, and $\theta$ is the highest root of the finite root system; see Proposition \ref{letalpha}.

Define the free $\Z[q]$-module $\quantum^*(\Fla):=\coh^*(\Fla) \otimes_\Z \Z[q]$ graded in the obvious way. The definition of the Gromov-Witten invariants above allows us to define the family of $\Z[q]$-linear, degree $1$ operators of graded $\Z[q]$-modules $\Lambda_i: \quantum^*(\Fla) \to \quantum^*(\Fla)$ given by \[ \Lambda_i(\eps_u) = \eps_i \cdot \eps_u + \sum_{d \in \coh_2(\Fla), d \neq 0} \gw{\eps_u, \eps_i,[X(w)]}{d} q^d \eps_w \/, \]
where $\eps_i \cdot \eps_u$ is the ordinary multiplication in $\coh^*(\Fla)$. These operators can be interpreted as affine quantum Chevalley operators on $\coh^*(\Fla)$. In Theorem \ref{thm:afqops} we find an explicit combinatorial formula for these operators, which generalizes to the affine case the quantum Chevalley formula of Peterson, proved by Fulton and Woodward \cite{fulton.woodward}.

Let $\cT \subset \cB$ and $T \subset B$ be maximal tori. Let $\e_1: \cG/\cT \to G/T$ be the map obtained by evaluation of loops at $t=1$.
Since $\cG/\cT$ and $G/T$ are affine bundles over the corresponding flag manifolds $\Fla$ and $G/B$, this induces a map $\e_1^*: \coh^*(G/B) \to \coh^*(\Fla)$. We analyze this map in  \S \ref{s:p} and prove that it is injective and that \[ \e_1^*(\sigma_i) = \eps_i - m_i \eps_0, \quad (1 \le i \le n), \textrm{ where } \theta^\vee = m_1 \alpha_1^\vee + \cdots  + m_n \alpha_n^\vee \/. \] In the topological category this map was studied by Mare \cite{Ma3} { and it played a key role in Peterson's ``quantum=affine" phenomenon \cite{Pe} proved by Lam and Shimozono \cite{lam.shimozono:affine}}. Using this map define $\coh^*_\af(G/B) := \e_1^*(\coh^*(G/B))$, which is a graded subalgebra of $\coh^*(\Fla)$ isomorphic to $\coh^*(G/B)$. Recall also the subalgebra $\coh^\#(\Fla)$ of $\coh^*(\Fla)$ generated by Schubert classes $\eps_0, \ldots, \eps_n$.  The main result of this paper, proved in \S \ref{sec:affqcohrings} below, is the following: 

\begin{thm}\label{T:mainintro} The affine quantum Chevalley operators $\Lambda_i$ satisfy the following properties:
\begin{itemize}

\item[(a)] The operators $\Lambda_i$ commute up to the imaginary coroot, i.e.~for any $w \in W_\af$ and any $0 \le i, j \le n$ we have \[ \Lambda_i \Lambda_j (\eps_w) = \Lambda_j \Lambda_i(\eps_w)  \mod q^c = q_0 q^{\theta^\vee} \/; \] 

\item[(b)] The modified operators $\Lambda_i - m_i \Lambda_0$ commute (without any additional constraint), i.e.~for any $w \in W_\af$ and any $1 \le i, j \le n$ we have \[(\Lambda_i - m_i \Lambda_0)(\Lambda_j - m_j \Lambda_0)(\eps_w) = (\Lambda_j - m_j \Lambda_0)(\Lambda_i - m_i \Lambda_0)(\eps_w)  \/; \] 
 
\item[(c)] Let $0 \le i \le n$. Then the Chevalley operator $\Lambda_i$ preserves the submodule $\quantum^\#(\Fla):= \coh^\#(\Fla) \otimes \Z[q]$. 

\item[(d)] Let $1 \le i \le n$. Then the modified Chevalley operator $\Lambda_i - m_i \Lambda_0$ preserves the submodule $\coh^*_\af (G/B) \otimes \Z[q]$. 
\end{itemize}
\end{thm}

As we show in Remark \ref{rmk:ex}, the restriction on commutativity up to $q^c$ from part (a) cannot be removed, even for $G= \SL_2(\C)$. 
In \S \ref{sec:affqcohrings} we use the algebra generated by the { commuting} operators $\Lambda_i - m_i \Lambda_0$, together with the fact that Schubert divisors $\sigma_i$ generate $\coh^*(G/B)$ over $\Q$, to give the definition of a product on $\coh^*_\af (G/B) \otimes \Q[q]$.  Using the injective algebra homomorphism $\e_1^*$ one can transfer this product by $\Q[q]$-linearity and define a product $\star_\af$ on $\quantum^*_\af(G/B):= \coh^*(G/B) \otimes \Q[q]$. 

\begin{cor}\label{cor:main intro} The pair $(\quantum^*_\af(G/B), \star_\af)$ is a graded, commutative, associative $\Q[q]$-algebra with a $\Q[q]$-basis given by Schubert classes $\sigma_w$, where $w$ varies in $W$. Further, the $\Q[q_1, \ldots, q_n]$-algebra $\quantum^*_\af(G/B) / \langle q_0 \rangle$ is naturally isomorphic to the ordinary quantum cohomology algebra $\quantum^*(G/B)$. 
\end{cor}

A similar product can be defined on the quotient $\quantum^\#(\Fla)/ \langle q^c \rangle $ which makes it a graded, commutative $\Z[q]$-algebra. The multiplication in both rings is determined by the Chevalley operators, thus one can algorithmically calculate any quantum products. An interesting fact is that the structure constants in $\quantum^*_\af(G/B)$ with respect to the Schubert basis are in general {\em not} positive, although the affine Chevalley operator $\Lambda_i$, and thus the quantum products $\eps_i \star_\af$ on $\quantum^\#(\Fla)/ \langle q^c \rangle $, are positive. See \S \ref{ss:table} below for some examples. Note that for $G={\rm SL}_n(\C)$ the
ring $\quantum^*_{\af}(G/B)/\langle q^c\rangle$ was constructed with different methods in 
\cite{Ma5}.

The proof of Theorem \ref{T:mainintro} is based on several different techniques. On one side one needs the precise combinatorial formula for $\Lambda_i$, which requires a study of the geometry and combinatorics of the relevant curve neighborhoods. Once this is done, we perform in \S \ref{combin} a thorough investigation of the  ``Chevalley" roots $\alpha$ which may appear in the formula for $\Lambda_i$. We mentioned above that these are the affine positive real roots $\alpha$ such that $\ell(s_\alpha) = 2 \he (\alpha^\vee)  -  1$, and satisfy $\alpha^\vee < c= \alpha_0^\vee + \theta^\vee$. We use this investigation to establish an involution of the set of certain chains of length $2$ in the affine analogue of the quantum Bruhat graph defined by Brenti, Fomin, and Postnilov \cite{brenti.fomin.post:mixed}. These chains help calculate the quantities $\Lambda_i \Lambda_j(\eps_w)$, and the involution corresponds to proving the identity in (a). The identity in (b) is an easy extra calculation, although we find it surprising that the $q^c$ constraint can be removed; it would be interesting to give an independent explanation of this fact.

The proof of properties (c) and (d), which are equivalent to the closure of the corresponding quantum products, turn out to be equivalent to certain properties of the divided difference operators acting on $\coh^*(\Fla)$ and $\coh^*(G/B)$. Part (c) is a consequence of the Leibniz formula for these operators. Part (d) is much more subtle and requires certain facts which might be of independent interest. Let $R_\af$ be the Coxeter ring generated by the divided difference operators $D_i$ ($0 \le i \le n$) associated to simple roots, acting on $\coh^*(\Fla)$. These operators satisfy $D_i^2 = 0$ and the braid relations; see (\ref{eq:braid}) below and \cite[Ch. 11]{kumar:kacmoody} for details. Similarly let $R$ be the Coxeter ring of Bernstein-Gelfand-Gelfand (BGG) divided difference operators from \cite{BGG} acting on $\coh^*(G/B)$, generated by $\partial_i$ for $1 \le i \le n$. We prove in Theorem \ref{T:ev} that there is a {\em ring} homomorphism $\pi: R_\af \to R$ such that $\pi(D_i) = \partial_i$ for $1 \le i \le n$ and $\pi(D_0 ) = \partial_{- \theta}$ where $\partial_{- \theta}$ is the BGG operator associated to the (negative) root $- \theta$. Since the braid relations are satisfied, one can define an operator $D_w \in R_\af$ for any $w \in W_\af$ by composition $D_w:= D_{i_1} \cdots D_{i_k}$ where $w$ has a reduced word $w = s_{\alpha_{i_1}} \dots s_{\alpha_{i_k}}$. Similarly one defines $\partial_w \in R$ for $w \in W$. The key formula needed to prove (d) is that $D_w$ and $\e_1^*$ commute with each other via $\pi$, i.e. for any $a \in \coh^*(G/B)$, \[ D_w (\e_1^* (a)) = \e_1^*(\pi(D_w)(a)) \in \coh^*(\Fla) \/; \] cf. Theorem \ref{thm:crucial}. In fact, with these notations we show in \S \ref{sec:affqcohrings} that the Chevalley multiplication formula in $\quantum^*_\af(G/B)$ is given by \[ \sigma_i \star_\af \sigma_w = \sigma_i \cdot \sigma_w + \sum_{\alpha} \langle \lambda_i -m_i \lambda_0, \alpha^\vee \rangle q^{\alpha^\vee} \pi(D_{s_\alpha})(\sigma_w), \quad 1 \le i \le n, \quad w \in W \] where the sum is over {\em affine} positive real roots $\alpha$ such that $\ell(s_\alpha) = 2 \he(\alpha^\vee) -1$ and $D_{s_\alpha}$ is the affine BGG operator. This formula is similar to the quantum Chevalley formula from the finite case \cite{fulton.woodward}. 

The second part of the paper is devoted to the study of the Dubrovin and Givental-Kim formalisms for the quantum product $\star_\af$ on $\coh^*(G/B)$, in analogy to the study from \cite{Du} and \cite{givental:egw, Kim} of the ordinary quantum product on $\coh^*(G/B)$; our treatment is inspired from \cite{cox.katz:mirror}. More precisely, let $\langle \cdot , \cdot \rangle$ be the Poincar{\'e} pairing on $\coh^*(G/B)$ extended linearly over $\C[q]$. Then $\quantum^*_\af(G/B)$ is a Frobenius algebra, i.e. it satisfies $\langle a \star_\af
b, c \rangle = \langle a , b \star_\af c \rangle$ for any $a,b,c \in \coh^*(G/B)$. In \S \ref{frodu} we construct the analogue of Dubrovin connection $\nabla^\hbar$ on the trivial bundle $\coh^2(G/B) \times \coh^*(G/B) \to \coh^2(G/B)$, and we prove it is flat (Theorem \ref{thm:flat}); here $\hbar \in \C^*$ is a parameter. Following \cite{givental:egw} we define the Givental connection to be $\nabla:= \hbar \nabla^{- 1/ \hbar}$. Let $\nabla_{\partial/ \partial z_i}$ be the derivation corresponding to the vector field $\partial/\partial z_i$ where $z_i$ is the coordinate on $\coh^2(G/B)$ corresponding to the Schubert class $\sigma_i$. The flatness of the Dubrovin connection $\nabla^\hbar$ together with an argument of Mare \cite{Ma3} imply that the system of {\em quantum differential equations}
\[ \hbar \partial/\partial z_i (s) = \sigma_i \star_{\af} s \Longleftrightarrow \nabla_{\partial/\partial z_i} (s) = 0;  \quad i \in \{ 1, \ldots, n \}  \] has nontrivial solutions $s$ in the ring of formal power series $\C[[q_0, \ldots , q_n]][z_1, \ldots , z_n][\hbar^{-1}]$ where $q_i = e^{z_i}$ for $1 \le i \le n$ and where there is a relation $q^c=q_0 q_1^{m_1} \cdots q_n^{m_n} = 1$. See Remark \ref{rmk:subs} for an interpretation of this relation. Further, for each $w \in W$ one can find solutions $s_w$ which have the leading term $\sigma_w$. These are the main ingredients needed to adapt the Givental-Kim formalism from \cite{givental:egw,Kim} to the affine case, and relate the ring $\quantum^*_\af (G/B)$ to the periodic Toda lattice. The key fact which makes this possible is the quadratic relation in $\quantum^*_\af(G/B)$: \[ \sum_{i, j =1}^n (\alpha_i^\vee | \alpha_j^\vee) \sigma_i \star_\af \sigma_j   = (\theta^\vee | \theta^\vee ) q_0 + \sum_{i=1}^n  ( \alpha_i^\vee | \alpha_i^\vee ) q_i \/, \] where $( \cdot | \cdot )$ is the Killing form on the Lie algebra of $G$ normalized so that $(\theta | \theta) =2$. The quantization of this relation corresponds to the differential operator
 \[ \mathcal{H} := \sum_{i, j =1}^n (\alpha_i^\vee | \alpha_j^\vee) \frac{\hbar \partial}{\partial z_i}\frac{\hbar \partial}{ \partial z_j }- (\theta^\vee | \theta^\vee ) e^{z_0} - \sum_{i=1}^n  ( \alpha_i^\vee | \alpha_i^\vee ) e^{z_i} \/\] acting on the ring of formal power series above. Here $\partial/\partial z_i$ is the partial derivative and $\hbar$ and $e^{z_i}$ act by multiplication; see \S \ref{s:givform} for details. The operator $\mathcal{H}$ is the Hamiltonian of the quantum periodic Toda lattice associated to the {\em dual} of the extended Dynkin diagram for $G$. (Sometimes this is referred to as the quantum periodic Toda lattice for the {\em short} dominant root $\theta^\vee$, and it corresponds to a {\em twisted} affine Lie algebra.) The complete integrability of this system has been established by  Goodman and Wallach \cite{Go-Wa1}  for $G$ of Lie types $A_n, B_n, C_n, D_n$ and $E_6$,
 and by Etingof \cite{etingof} in the simply laced Lie types (in these cases, there is no difference between the dual and the non-dual versions of the periodic quantum Toda lattice). { For the remaining types $F_4$ and $G_2$ the integrability has been proved by Mare \cite{M}.}
 

The relevance of complete integrability comes from the fact that by the Givental-Kim formalism any differential operator $\mathcal{H}_k$ which is polynomial in $e^{z_i},\partial/\partial z_i$ and $\hbar$, and which commutes with $\mathcal{H}$ gives a relation in $\quantum^*_\af(G/B)$. This leads to the second main result of this paper, proved in \S \ref{s:Toda}.
 
\begin{thm}\label{thm:mainintro2} The algebra $\quantum^*_\af(G/B)$ has a presentation of the form $\Q[q][x_1, \ldots, x_n]/ I$ where $I$ is the ideal generated by the integrals of motion of the dual periodic Toda lattice associated to $G$, and where the indeterminates $x_i$ correspond to the Schubert divisors $\sigma_i$. \end{thm} 

 Observe that the ideal of relations is determined by the integrals of motion for the {\it ordinary} (non-quantum) dual periodic Toda lattice:
 recall that they are obtained from the differential operators in the quantum version after taking top degree terms and making certain substitutions.  

This theorem naturally generalizes the results of Givental and of Kim \cite{givental.kim,Kim} for the ordinary quantum cohomology ring $\quantum^*(G/B)$. In fact, there has been recently a flurry of activity relating quantum cohomology and quantum K-theory of the (cotangent bundle of) flag manifolds to integrable systems, and eventually to quantum groups; see e.g. \cite{givental.lee:qk,nekrasov.shatashvili, braverman.maulik.okounkov,maulik.okounkov,grtv:quantum,gorbounov.korff}. A connection between the periodic Toda lattice and certain $1$-point Gromov-Witten invariants for affine flag manifolds appears in the work of Braverman \cite{braverman:instantons}. He studies a particular compactification of the space of rational curves in $\Fla$ with one marked point, called based quasi-map spaces. Using this he constructs an equivariant J-function which is an eigenvector for a differential operator related to the Hamiltonian $\mathcal{H}$. Since in the finite case the J-function is determined by flat sections of the Givental connection, it is natural to expect that an equivariant version of the aforementioned flat sections $g_w$ will be related to Braverman's J-function. 


{\em Acknowledgements.} This paper benefited from interactions with many people. L.~Mare would like to thank Martin Guest and Takashi Otofuji for useful discussions. L.~C.~Mihalcea  would like to thank Prakash Belkale, Dan Orr, Alexey Sevastyanov, Chris Woodward for stimulating discussions; Shrawan Kumar and Mark Shimozono for patiently answering many questions about the geometry and combinatorics of affine flag manifolds; Pierre-Emmanuel Chaput, Nicolas Perrin and Anders Buch for inspiring collaborations to various projects where the curve neighborhoods played a key role; special thanks are due to Anders Buch for insightful comments and discussions.

\section{Preliminaries}
\label{prel}

The goal of this section is to set up notations and recall some basic facts about Lie algebras and their affine versions, and about the cohomology of the associated flag varieties. Throughout this article $G$ will denote a complex simple, simply connected Lie group. We fix $T \subset B \subset G$ a maximal torus included in a Borel subgroup of $G$. Let $\g$ and $\h$ be the Lie algebras for $G$ and $T$. 
The corresponding sets of roots and coroots are  $\Pi\subset \h^*$ and  $\Pi^\vee:=\{\alpha^\vee \mid \alpha \in \Pi\}\subset \h$. Pick a simple root system along with the corresponding coroots:
$$\Delta=\{\alpha_1,\ldots, \alpha_{n}\}, \quad \Delta^\vee=\{\alpha_1^\vee,\ldots, \alpha_{n}^\vee\}.$$
This determines the partition of $\Pi$ into positive and negative roots: $\Pi = \Pi^+\cup \Pi^-$.

Denote by $( \cdot | \cdot )$ the Killing form of $\g$,  which we normalize in such a way that
$(\alpha | \alpha)=2$, for any long root $\alpha$ (recall that the restriction of $( \cdot | \cdot )$ to $\hc$ is non-degenerate,
see e.g.~\cite[Section 14.2]{fulton.harris:repthry}). Let $\langle \ ,\ \rangle : \h^* \times \h \to \C$ be the evaluation pairing, and let $\{ \omega_1, \ldots, \omega_n \} \subset \h^*$ be the fundamental weights satisfying $\langle \omega_i, \alpha_j^\vee \rangle = \delta_{ij}$. Set $W:= N_G(T)/T$ the Weyl group. This is generated by the simple reflections $s_i:=s_{\alpha_i}$ corresponding to the simple roots $\alpha_i$. For $w \in W$, the length $\ell(w)$ equals the number of simple reflections in any reduced word decomposition of $w$; $w_0$ denotes the longest element of $W$. 

The (finite) {\em flag variety} $G/B$ is a complex projective manifold of dimension $\ell(w_0)$. The group $G$ acts transitively on the left. The {\em Schubert varieties} $X(w):= \overline{B w B/B}$ and $Y(w) = \overline{B^- w B/B}$ are irreducible subvarieties of $G/B$ so that $\dim X(w) = \mathrm{codim}~ Y(w) = \ell(w)$, where $B^- = w_0 B w_0$ is the opposite Borel subgroup. For now we consider the homology $\coh_*(G/B)$ and cohomology $\coh^*(G/B)$ with integral coefficients, but occasionally we will need to work over $\Q$, and we will specify when this is the case. The homology is a free $\Z-$module with a basis given by fundamental classes $[X(w)] \in \coh_{2 \ell(w)}(G/B)$, where $w$ varies in $W$. The Poincar{\'e} pairing $\coh^*(G/B) \otimes \coh_*(G/B) \to \Z$ sending $a \otimes b$ to $\int_{G/B} a \cap b$ is nondegenerate; here the integral denotes the push forward to a point. Let $\sigma_w \in \coh^{2 \ell(w)}(G/B)$ denote the dual class of $[X(w)]$ with respect to this pairing. Since the intersection $X(w) \cap Y(w)$ is transversal and it consists of a single $T$-fixed point $e_w$, $\sigma_w$ is naturally identified with $[Y(w)]$. 

For each {\em integral} weight $\lambda \in \h^*$ we denote by $\C_\lambda$ the $1$-dimensional $T$-module of weight $\lambda$,  defined by $z.u = \lambda(z)u$. Recall that the Borel group $B$ can be written as the product $B=UT$, where $U$ is the unipotent subgroup. Then we regard $\C_\lambda$ as a $B$-module by letting elements of $U$ act trivially. Let $\mathcal{L}_\lambda$ be the $G$-equivariant line bundle over $G/B$ \[ \mathcal{L}_\lambda:= G \times^B \C_{-\lambda} = (G \times \C)/B \] where $B$ acts on $G \times \C$ by $b. (g,u) = (g b^{-1}, \lambda(b)^{-1} u)$. With these definitions we have that $c_1(\mathcal{L}_{\omega_i}) = [Y(s_i)]$; see e.g. \cite[\S 8]{buch.mihalcea:nbhds}. This identifies $\coh^2(G/B)= \oplus_{i=1}^n [Y(s_i)]$ with the (integral) weight lattice $\oplus_{i=1}^n \Z \omega_i$. We can further identify $\coh_2(G/B)$ with the coroot lattice $\oplus_{i=1}^n \Z \alpha_i^\vee$ by letting $[X(s_i)]$ correspond to $\alpha_i^\vee$. Then the restriction of the Poincar{\'e} pairing to $\coh^2(G/B) \otimes \coh_2(G/B)$ is identified to the evaluation pairing $\langle \ ,\ \rangle$.

\subsection{Affine Kac-Moody algebras}\label{akm} Next we establish the main notation for the affine root systems and the coresponding Lie algebras, following the references \cite{Kac, kumar:kacmoody}. 
Let $\gaf$ be the affine (non-twisted) Kac-Moody Lie algebra associated to $\g$ (cf.~e.g. \cite[Ch.~7]{Kac} or \cite[Section 13.1]{kumar:kacmoody}). By definition,
 $$\gaf = \Lca(\g)\oplus \C c \oplus \C d,$$
where $\Lca(\g):=\g \otimes \bC[t, t^{-1}]$ ($t \in \C^*$) is the space of all Laurent polynomials in $\g$ and where 
$c$ is a central element with respect to the Lie bracket in $\gaf$. 
The Cartan subalgebra of $\g_\af$ is $\h_\af:=\h \oplus \bC c \oplus \bC d$ and let $\langle \ , \ \rangle : \h_\af^* \times \h_\af \to \bC$ be the evaluation pairing.  The {\em affine root system} $\Pi_\af$ associated to $\gaf$ consists of
\begin{itemize} 
\item $m\delta + \alpha$, where $m\in \bZ$ and $\alpha \in \Pi$
(these are called the {\it (affine) real roots}).
\item $m\delta$, $m\in \bZ\setminus \{0\}$
(these are the {\it imaginary roots}).
\end{itemize}
Here the embedding $\Pi \subset \Pi_\af$ identifies a root $\alpha \in \Pi$ with the linear function on $\h_\af$  whose restriction
to $\h$ equals to $\alpha$ and  which satisfies $\langle \alpha,c\rangle =\langle \alpha,d\rangle =0$.  The {\em imaginary root} $\delta$ of $\h_\af^*$ is defined by
\[\delta|_{\h \oplus \bC c} =0 \quad {\rm and}  \quad \langle \delta, d\rangle=1 \/.\]
The {\em (affine) simple root system} $\Delta_\af \subset \Pi_\af$ is 
$$\Delta_\af :=\{ \alpha_0:=\delta-\theta, \alpha_1, \ldots, \alpha_{n}\}$$ where
$\theta\in \Pi$ is the highest root. As in the finite case, this determines a partition $\Pi_\af = \Pi^+_\af\cup \Pi^-_\af$ into positive and negative roots.
Denote by $\Pi^{re}_\af$ and $\Pi^{re,+}_{\af}$ the set of affine roots which are real, respectively the subset of positive real roots. Then $\Pi_\af^{re,+}$ contains the finite positive roots along with $m\delta+\alpha$, where 
$m>0$ and $\alpha \in\Pi$. For $a,b$ in the root lattice $\bigoplus_{0\le i \le n} \bZ \alpha_i$, we say that $a \le b$ if  
$b -a$ is a linear combination with non-negative coefficients of $\alpha_0,\ldots,\alpha_{n}$.
By $a<b$ we mean $a\le b$ and $a\neq b$. The {\it (affine) simple coroot system} is the subset $\Delta_\af^\vee \subset \h_\af$:
$$\Delta_\af^\vee :=\{\alpha_0^\vee :=c-\theta^\vee, \alpha_1^\vee,\ldots, \alpha_{n}^\vee\}.$$ 
The {\it affine Weyl group} ${W}_\af$ of $\g_\af$ relative to $\h_\af$ is by definition the subgroup of $\GL(\h_\af^*)$ generated by the {\em simple reflections}
$s_0, \ldots, s_{n}$, where $s_i(\lambda):=\lambda -\langle \lambda, \alpha_i^\vee\rangle \alpha_i$ for $\lambda \in \h_\af^*$ and $0\le i \le n$. It turns out that $W_\af$ leaves $\Pi_\af$ invariant. Further, $\alpha \in \Pi_\af$ is a real root if and only if $\alpha =w\alpha_i$
for some $w\in W_\af$ and $0\le i\le n$. Then the resulting reflection $s_\alpha \in W_\af$ is independent of choices of $w$ and $\alpha_i$ and it is the linear automorphism of $\h_\af^*$ 
\[s_\alpha(\lambda) = \lambda - \langle \lambda, \alpha^\vee\rangle \alpha; \quad (\lambda \in \h_\af^*) \/. \] For each $w \in W_\af$ one defines the length $\ell(w)$ to be the minimal length of a reduced word of $w$ in terms of the generating system $s_0, \ldots, s_{n}$.
Recall that if $w\in W_\af$ and $\alpha \in \Pi_\af^{re,+}$, then:
\begin{itemize}
\item $\ell(ws_\alpha) <\ell(w)$ (resp. $\ell(w s_\alpha) > \ell(w)$) if and only if $w\alpha <0$ (resp. $w \alpha >0$).
\item (Strong Exchange Condition) If $\ell(ws_\alpha)<\ell(w)$ and $w=s_{i_1} \cdots s_{i_k}$ is a (possibly not reduced) expression
then $ws_\alpha = s_{i_1} \cdots \hat{s}_{i_j} \cdots s_{i_k}$ for some $1\le j \le k$.
 \end{itemize}
See e.g. \cite[\S 1.3] {kumar:kacmoody} for details. These facts also follow because $(W_\af, \{s_0,\ldots, s_{n}\})$ is a Coxeter group; see e.g.~\cite{Hu2} or \cite{Bj-Br}.

If $\alpha = w\alpha_i$ then the {\it coroot} $\alpha^\vee :=w\alpha_i^\vee$ is independent of choices of $w$ and $\alpha_i$. 
The set $\{\alpha^\vee\mid \alpha \in \Pi_\af^{re}\}$ consists of the real roots of the Kac-Moody affine Lie algebra
$\g_\af^\vee$ which is associated to the Langlands dual simple Lie algebra $\g^\vee$ of $\g$, and whose root system is
$\Pi^\vee$. The simple root system of $\g_\af^\vee$ is $\Delta_\af^\vee$ and the roots of $\g_\af^\vee$ are:
\begin{itemize}
\item $mc + \alpha^\vee$, where $m\in \bZ$ and $\alpha \in \Pi$
(the {\it real coroots}).
\item $mc$, $m\in \bZ\setminus \{0\}$
(the {\it imaginary coroots}).
\end{itemize}
We still denote by $\le$ the ordering on the coroot lattice determined by the positive coroots.
An alternative description of the coroots can be obtained in terms of the invariant bilinear form 
$( \cdot | \cdot )$ on $\g_\af$ given by
\begin{equation}\label{ur1}(u+ r_1 c + s_1 d |  v+r_2c+s_2d):=r_1s_2+s_1r_2 +{\rm Res}(t^{-1}(u , v)),\end{equation}
for all $u,v\in \Lca(\g)$ and $r_1,r_2, s_1, s_2 \in \bC$. Here $(u,v)$ is
the function $\C^*\to \bC$ given by $(u,v)(t)=(u(t) | v(t))$ and Res stands for residue. This inner product is 
invariant, in the sense that $([a_1, a_2] | a_3)=(a_1| [a_2, a_3])$, for all $a_1,a_2,a_3\in \gaf$. Its restriction to $\g$ 
coincides with the  biinvariant metric we have considered initially. 
The restriction of $( \cdot  | \cdot )$  to $\h_\af \times \h_\af$ is nondegenerate and invariant under $W_\af$, hence it induces a linear isomorphism
$\nu: \h_\af\to \h_\af^*$. We have
\begin{equation}\label{alphv}\alpha^\vee =\frac{2\nu^{-1}(\alpha)}{(\alpha | \alpha)},\end{equation}
for any root $\alpha \in \Pi_\af^{re}$, therefore the duals of finite roots
are exactly the finite coroots. The identity (\ref{alphv}) implies that $s_\alpha =s_{\alpha^\vee}$, hence the affine Weyl groups of $\g_\af$ and $\g_\af^\vee$ coincide.

The restriction of the inner product (\ref{ur1}) to $\h_\af$ is nondegenerate.  
Thus it induces an inner product on $\h_\af^*$, in particular on its subspace ${\h}^*\oplus \bC \delta$.
 Denote by $\h^*_\bR$ the subspace of $\h^*$ spanned by the
roots, i.e., the elements of $\Pi$.  Observe that $\Pi_\af$ is contained in ${\h}_\bR^*\oplus \bR \delta$.

\begin{remark} \label{rth} 
For $x \in {\h}_\bR^*\oplus \bR \delta$, we have
\begin{itemize}
\item $(x | x) \ge 0$
\item $(x | x)=0$ if and only if $x\in \bR\delta$. 
\end{itemize} 
Indeed, one can see from (\ref{ur1})  that if $x = x_0 +r\delta$, with $x_0 \in {\h}^*$ and $r\in \bR$ then
$(x|x) = (x_0 | x_0)$. The claim follows because the inner product $(x_0 | x_0)$ is the
Killing form and its restriction to $\h_\bR^*$ is strictly positive definite \cite[Section 14.2]{fulton.harris:repthry}.
\end{remark}
 
 The following property of the root system $\Pi_\af$ will be useful; see \cite[Proposition 5.1]{Kac}:
  \begin{prop}\label{string} Let $\alpha \in \Pi_\af^{re}$ and $\beta \in \Pi_\af$. Then there is an equality:
  $$\{ \beta + k \alpha \mid k\in \bZ\}\cap (\Pi_\af \cup \{0\}) = \{\beta - p \alpha , \beta -(p-1)\alpha, \ldots, \beta + (q-1)\alpha,
  \beta + q\alpha\},$$
  where $p,q \in \bZ_+$  are such that $p-q = \langle \beta, \alpha^\vee\rangle$. In particular, $\langle \beta, \alpha^\vee \rangle \in \Z$. 
    \end{prop} 

\subsection{Affine flag varieties and their cohomology}\label{sec:afm} In this section we recall some basic facts about affine flag varieties and (co)homology. Our main reference is \cite{kumar:kacmoody}, especially Chapters 7, 11 and 13. Let $G(\C[t,t^{-1}])$ be the group of Laurent polynomials loops $\C^* \to G$ (cf.~\cite[Def. 13.2.1]{kumar:kacmoody}) and let $\cG$ be the semidirect product $\C^* \ltimes G(\C[t,t^{-1}])$ where $\C^*$ acts by loop rotation, i.e. $(z \cdot \gamma)(t) = \gamma(tz)$ for $\gamma=\gamma(t) \in G(\C[t,t^{-1}])$. As explained in \cite[13.2.2]{kumar:kacmoody} there is a group homomorphism $\overline{e}: \C^* \ltimes G(\C[t]) \to \C^* \times G$ sending $(z, {g}(t))$ to $(z,{g}(0))$ obtained by evaluation at $t=0$. Define $\cB := \overline{e}^{-1}(\C^* \times B)$ and $\mathcal{U}:=  \overline{e}^{-1}(1 \times U)$ where $U \subset B$ is the unipotent subgroup of $B$. The restriction of the semidirect product defining $\cG$ to $\C^* \ltimes G$ is actually a direct product hence the standard maximal torus $\mathcal{T} \subset \cB$ is isomorphic to $\C^* \times T$. The subgroup $\cB$ is the standard Iwahori subgroup of $\cG$ and there is a semidirect product $\cB = \mathcal{U} \cdot (\C^* \times T)$. With these notations the affine flag variety $\Fla$ associated to the group $G$ is $\Fla:=\cG/\cB$.\begin{footnote}{By \cite[Corollary 13.2.9]{kumar:kacmoody} $\cG$ is closely related to the Kac-Peterson group discussed in \S 7.4 of {\em loc.~cit.}, which itself is a subgroup of the Kac-Moody group associated to $G$. Although all these groups are distinct, their flag varieties coincide; see p.~231 of {\em loc.cit.}}\end{footnote}

The flag variety $\Fla$ has a natural structure of a projective ind-variety, i.e. it has a filtration $\mathcal{X}_1 \subset \mathcal{X}_2 \subset \ldots. \subset \Fla$ where $\mathcal{X}_i$ are finite dimensional projective algebraic varieties and the inclusions $\mathcal{X}_i \subset \mathcal{X}_{i+1}$ are closed embeddings. This filtration is used to endow $\Fla$ with the 
{\it strong topology}. We consider the homology and cohomology relative to this topology, with $\Z$ coefficients.

As in the finite case, the {\em Schubert varieties} $X(w) := \overline{\cB w \cB/\cB}$ are irreducible complex projective varieties of dimension $\ell(w)$. Notice that we used the same notation $X(w)$ as in the finite case. The context should clarify any confusions; as a general rule, if $w \in W_\af$ (as is the situation here) then $X(w) \subset \Fla$. The fundamental classes $[X(w)] \in H_{2 \ell(w)}(\Fla)$ form a $\Z$-basis of $H_*(\Fla)$ as $w$ varies in $W_\af$. Denote by $\{\eps_w \mid w\in W\}$ the dual basis of $\coh^*(\Fla)$ relative to the
 natural ``cap" pairing $\coh^*(\Fla)\otimes H_*(\Fla)\to \Z$ sending $a \otimes b$ to $\int_{\Fla} a \cap b$. Thus
 $\langle \eps_w,[X_v]\rangle =\delta_{vw}$ 
 for all $v, w\in W_\af$, where $\delta_{vw}$ is the Kroenecker delta. 
 We refer to $\{\eps_w\mid w\in W_\af\}$ as the {\it Schubert basis} in $\coh^*(\Fla)$;
 note that $\eps_w\in \coh^{2\ell(w)}(\Fla)$. We will use the notation
$\ep_i:=\ep_{s_i}$, for $0\le i \le n$.
 
Let $\lambda_0,\ldots, \lambda_n$ denote a set of fundamental weights relative to $\Delta_\af$.
Then $\lambda_i\in \h_\af^*$ are determined by  $$\langle \lambda_i,\alpha_j^\vee\rangle =\delta_{ij}
\ {\rm and} \ \langle \lambda_i,d\rangle =0;\quad  0\le i,j\le n \/. $$
As in the finite case, for each integral weight $\lambda \in \h_\af^*$ there is an associated line bundle $\mathcal{L}_\lambda= \cG \times^\cB \C_{-\lambda}$ where the $\mathcal{T}$-module $\C_{-\lambda}$ is extended over $\cB$ by letting $\mathcal{U} \subset \cB$ act trivially. It is proved in \cite[p. 405]{kumar:kacmoody} that $c_1(\mathcal{L}_{\lambda_i}) = \eps_i$, for $0 \le i \le n$. { Since we could not find a reference, we include the proof of the following proposition.}

\begin{prop}\label{prop:delta trivial} The line bundle $\mathcal{L}_\delta$ associated to the imaginary root is trivial on $\Fla$. \end{prop}
\begin{proof} Take an equivalence class $[g,u] \in \cG \times^\cB \C_{-\delta}$ for $g \in \cG$ and $u \in \C$. { We
can choose a representative for the coset of $g \cB$ of the form $g=(1, g') \in \C^* \ltimes G(\C[t,t^{-1}])$.} The restriction of $\delta$ to $\mathcal{U} \cdot (1 \times T)$ is trivial. Hence if $(1,\tilde{g}) \in \cG$ is such that $[(1,\tilde{g}),\tilde{u}] = [(1,g'),u] $ in $\mathcal{L}_\delta$ (for some $\tilde{u} \in \C$) then $(1,\tilde{g}) \cB = (1,g') \cB$ and $\tilde{u}=u$. Therefore the application sending $[(1,g'), u]$ to $((1,g')\cB, u) \in \cG/\cB \times \C$ is well defined and it gives an isomorphism of line bundles between $\mathcal{L}_\delta$ and $\cG/\cB \times \C$. 
\end{proof}
We notice now that $\h_\af^* = \h^* \oplus \C \delta \oplus \C \lambda_0$. Therefore, by identifying $\eps_i$ with $\lambda_i$ and $[X(s_i)]$ with $\alpha_i^\vee$ ($0 \le i \le n$) we can identify the Poincar{\'e} pairing 
$\coh^2(\Fla) \otimes \coh_2(\Fla) \to \Z$ to the restriction of the evaluation pairing $\langle \ , \ \rangle $ to $(\oplus_{i=0}^n \Z \lambda_i) \times (\oplus_{j=0}^n \Z \alpha_j^\vee)$. With these notations, the {\em Chevalley formula} in $\coh^*(\Fla)$ states that if $0\le i \le n$ and $w\in W_\af$ then 
\begin{equation}\label{chaff} \ep_i\cdot \ep_w = \sum_{\alpha} \langle \lambda_i , \alpha^\vee \rangle \ep_{ws_\alpha},\end{equation}
where the sum runs over all positive real roots $\alpha \in \Pi_\af^{{re},+}$ such that $\ell(w s_\alpha)=\ell(w)+1.$ 
See \cite[Theorem 11.1.7 (i) and Corollary 11.3.17, Eq.~(3)]{kumar:kacmoody}.

\section{A morphism between the affine and finite nil-Coxeter rings} 
Let $R_\af$ and $R$ be the {\em nil-Coxeter rings} of divided difference operators associated to the affine Weyl group $W_\af$ respectively the finite Weyl group $W$. A generalization of these rings, called the {\em nil-Hecke rings}, has been studied by Kostant and Kumar \cite{Ko-Ku} in the more general setting of equivariant cohomology of Kac-Moody flag varieties; see also \cite[\S 11.1]{kumar:kacmoody}. The main goal of this section is to construct a ring homomorphism $\pi: R_\af \to R$; see Theorem \ref{T:ev} below. We recall next the relevant definitions.

Denote by $D_i: \coh^k(\Fla) \to \coh^{k-2}(\Fla)$ (where $0 \le i \le n$ and $k \ge 0$) the affine BGG operator acting on the Schubert basis of $\coh^*(\Fla)$ by 
\[ D_i (\varepsilon_v) = \begin{cases} \varepsilon_{vs_i} & \textrm{ if } \ell(vs_i) < \ell(v); \\ 0 & \textrm{otherwise} \/. \end{cases} \]  
Geometrically, these operators arise as ``push-pull" operators in a fibre diagram of $\P^1$-bundles on Kashiwara's ``thick" flag manifold; see \cite{Ka-Sh}. The operators $D_i$ satisfy the nilpotence and braid relations:
\begin{equation}\label{eq:braid} D_i^2 = 0; \quad D_i D_j \cdots = D_j D_i \cdots \quad  (m_{ij} {\rm ~factors}), \quad 0 \le i , j \le n \end{equation} where $m_{ij}$ is the order of $s_i s_j$ in $W_\af$. This implies that for each $w \in W_\af$ with a reduced word $w= s_{i_1} \cdots  s_{i_k}$ there is a well defined operator $D_w:= D_{i_1} \cdots  D_{i_k}$ which is independent of the choice of the word. Then $R_\af$ has a $\Z$-basis given by elements $D_w$ ($w \in W_\af$) \cite[Theorem 11.2.1]{kumar:kacmoody} with the multiplication given by composition. 

If we omit the ``affine" operator $D_0$, replace the affine Weyl group by the finite Weyl group $W$, and the cohomology ring $\coh^*(\Fla)$ by $\coh^*(G/B)$, we obtain the description of the finite nil-Coxeter ring $R$. To distinguish them from the affine case, we denote the finite BGG operators by $\partial_i$ and $\partial_w$ respectively. 
In fact, $\partial_i$ are just special cases of the classical divided difference operators $\partial_\alpha : \coh^*(G/B)\to
\coh^*(G/B)$, $\alpha \in \Pi$ (any finite root), which were defined by Bernstein, I.~M.~Gelfand, and S.~I.~Gelfand in \cite{BGG}.
More precisely, we have $\partial_i=\partial_{\alpha_i}$, $1\le i \le n$.
We note that for any $\alpha\in \Pi$, the operator $\partial_\alpha$ is originally the endomorphism of ${\rm Sym}(\h_\Q^*)$ given by
\begin{equation}\label{delt}\partial_\alpha := \frac{{\rm id}- s_\alpha}{\alpha}.\end{equation}
One then uses the  Borel presentation $\coh^*(G/B;\Q)={\rm Sym}(\h_\Q^*)/\langle {\rm Sym}(\h_\Q^*)^W_+\rangle$,
where $\langle {\rm Sym}(\h_\Q^*)^W_+\rangle$ is the ideal of ${\rm Sym}(\h_\Q^*)$ generated by the non-constant $W$-invariant polynomials. Recall that this presentation arises by identifying each $\omega_i$ with $c_1({\mathcal L}_{\omega_i})$, $1\le i \le n$ (the line bundles ${\mathcal L}_{\omega_i}$
are defined in \S \ref{prel} above). We notice that the operator $\partial_\alpha$ preserves integral cohomology classes. For future use we record the following well-known Leibniz properties satisfied by the BGG operators:

\begin{prop}\label{prop:leibniz} 
(a)  Let $\alpha$ be a root in the finite root system $\Pi$ and $x,y \in \coh^*(G/B)$. Then \[ \partial_{\alpha}(xy) = \partial_{\alpha}(x) y + x \partial_{\alpha}(y) - c_1(\mathcal{L}_\alpha) \partial_{\alpha}(x) \partial_{\alpha}(y) \/. \] 

(b)
 For any $0\le i \le n$ and any $x, y\in \coh^*(\Fla)$ we have 
$$ D_i (xy) = D_i (x) y + x D_i (y) - c_1(\mathcal{L}_{\alpha_i}) D_i (x) D_i (y).$$
\end{prop}

\begin{proof} (a) From (\ref{delt}) we deduce easily that for any $f, g \in S(\h_\Q^*)$ we have 
\begin{equation}\label{alpfg}\partial_\alpha(fg) =\partial_\alpha(f)g+s_\alpha(f)\partial_\alpha(g).\end{equation} 
It only remains to observe that $s_\alpha(f) =f- \alpha\partial_\alpha(f)$, and use the aforementioned 
Borel isomorphism $\coh^*(G/B;\Q)={\rm Sym}(\h_\Q^*)/\langle {\rm Sym}(\h_\Q^*)^W_+\rangle$. 

 (b) By \cite[Theorem 11.1.7 (4) and Theorem 11.3.9]{kumar:kacmoody} we have that $D_i(xy) = D_i(x)y+s_i (x)D_i(y)$. But from \cite[Theorem 11.3.9]{kumar:kacmoody} combined with \cite[Eq.~(7), p.~373]{kumar:kacmoody}, we have 
 $s_i(x) = x -c_1(\mathcal{L}_{\alpha_i})D_i(x).$ This finishes the proof.
\end{proof}
 
\begin{remark}\label{rmk:rw} The divided difference operator $\partial_\alpha$ can also be defined by \begin{equation}\label{eq:palpha} \partial_\alpha = r_w^* \partial_i r_{w^{-1}}^* \end{equation} where $w \in W$, $\alpha_i \in \Delta$ are such that $w(\alpha_i) = \alpha$, and $r_w^*: \coh^*(G/B;\Z ) \to \coh^*(G/B;\Z)$ is the degree $0$, $\Z$-algebra automorphism determined by the {\em right} Weyl group action of $w \in W$ on $\coh^*(G/B)$. We will use this alternate definition in \S \ref{frodu} below, and we refer to \cite{knutson:noncomplex,tymoczko:permutation} for the explicit construction and formulas for $r_w$ in the finite setting. 
Notice also that the same definition extends in the Kac-Moody generality, see e.g.~\cite[p. 387]{kumar:kacmoody}. \end{remark} 
 
The main result of this section is:

\begin{thm}\label{T:ev} There is a well-defined ring homomorphism $\pi: R_\af \to R$ sending $D_i$ to $\partial_i$ if $i \neq 0$ and $D_0$ to $\partial_{-\theta}$, where $\theta$ is the highest root of the finite root system $\Pi$.\end{thm} 

Before proving the theorem, we remark that the homomorphism $\pi$ appeared in Peterson's lecture notes \cite{Pe}, but we could not find a proof for its properties therein. The strategy of proof uses two facts: that the nil-Coxeter ring $R_\af$ has a presentation with generators $D_i$ and relations (\ref{eq:braid}), and that these relations are preserved under the push-forward by $\pi$.  

Define $A_\af$ the ring with generators $A_i$, for $0 \le i \le n$ and relations (\ref{eq:braid}) where we replace $D_i$ by $A_i$. There is a ring homomorphism $a: A_\af \to R_\af$ sending $A_i$ to $D_i$. 
\begin{lemma}\label{sgph}
The ring homomorphism $a:A_{\af}\to R_{\af}$ is an isomorphism.
\end{lemma}
This lemma appears to be known among experts, although we could not find a reference for it. We are grateful to S.~Kumar who suggested to us the approach used in the proof.
\begin{proof} Let $w\in W_\af$ with a reduced expression $w=s_{i_1}\cdots s_{i_k}$. There is a well defined element $A_w:=A_{i_1}\cdots A_{i_k}$ in $A_\af$. 
To finish the proof it suffices to show that $\{A_w \mid w\in W_\af\}$ is a $\Z$-basis of $A_{\af}$. 
First, this set is linearly independent, since $a(A_w)=D_w$ and $\{D_w\mid w\in W_\af\}$ is a basis of $R_{\af}$. 
To show that $A_w$ span the $\Z$-module $A_\af$, it suffices to show that for $i_1,\ldots, i_k\in \{0,1,\ldots, n\}$ we have
$$
A_{i_1}\cdots A_{i_k} 
=\begin{cases}
A_{s_{i_1}\cdots s_{i_k}} , \ {\rm if} \ s_{i_1} \cdots s_{i_k} \ {\rm is \ a \ reduced \ expression}\\
0, \ {\rm otherwise.}
\end{cases} $$
We prove this claim by induction on $k$. The case $k=1$ is clear, so let $k \ge 2$ and assume that $u:=s_{i_1} \cdots s_{i_k}$ satisfies $\ell(u) <k$. We have $u=u's_{i_k}$, where $u':= s_{i_1} \cdots s_{i_{k-1}}$.
If the word $s_{i_1} \cdots s_{i_{k-1}}$ is not reduced, then by the induction hypothesis,
$ A_{i_1}\cdots A_{i_k}=A_{i_1}\cdots A_{i_{k-1}} A_{i_k}=0$.
If $\ell(u') = k-1$ then
 $\ell(u's_{i_k})<\ell(u')$ and by the Exchange Condition (cf.~e.g.~\cite[p. $14$]{Hu2}), there exists a reduced expression for $u'$ which ends with $s_{i_k}$, that is
$u'=s_{j_1} \cdots s_{j_{k-2}}s_{i_k}$.
 We then have
$ A_{i_1}\cdots A_{i_k} = A_{i_1}\cdots A_{i_{k-1}} A_{i_k}
= A_{j_1}\cdots A_{j_{k-2}} A_{i_k}A_{i_k}=0$
since $A_{i_k}^2=0$.
\end{proof}

\begin{proof}[Proof of Theorem \ref{T:ev}] By Lemma \ref{sgph} it suffices to show that the relations (\ref{eq:braid}) are preserved under $\pi$. Since the map $W \to W_\af$ sending $s_i$ to $s_i$ ($i \neq 0$) is an injective group homomorphism, it suffices to check that
\begin{equation}\label{chbr} \partial_{-\theta}^2 = 0; \quad \partial_i \partial_{-\theta} \cdots  = \partial_{-\theta} \partial_i\cdots  \quad (m_{0i} \textrm{ factors }),\quad  1 \le i \le n\/, \end{equation} where $m_{0i}$ is the order of $s_0 s_i$ in the {\em affine} Weyl group $W_\af$.
 It is known that
\begin{equation}\label{mij}m_{0i}=\begin{cases}
2, \ {\rm if} \ \langle \alpha_0, \alpha_i^\vee\rangle \langle \alpha_i,\alpha_0^\vee\rangle= 0\\  
 3, \ {\rm if} \ \langle \alpha_0, \alpha_i^\vee\rangle \langle \alpha_i,\alpha_0^\vee\rangle= 1\\
4,   \ {\rm if} \ \langle \alpha_0, \alpha_i^\vee\rangle \langle \alpha_i,\alpha_0^\vee\rangle= 2\\
6,  \ {\rm if} \ \langle \alpha_0, \alpha_i^\vee\rangle \langle \alpha_i,\alpha_0^\vee\rangle= 3,
\end{cases}
\end{equation}
see for instance \cite[Proposition 3.13, p.~41]{Kac}. 
Observe that for any $1\le i \le n$ we have 
$$\langle \alpha_i ,\alpha_0^\vee \rangle=\langle \alpha_i ,-\theta^\vee \rangle \ {\rm and} \ 
  \langle \alpha_0,\alpha_i^\vee\rangle= \langle -\theta,\alpha_i^\vee\rangle.$$
  Consider the root system generated by $-\theta$ and $\alpha_i$. A case by case analysis of the extended Dynkin diagrams
  (see e.g.~\cite[Table Aff1, p.~44]{Kac}) shows that  if $\g$ is not of type $C$ the elements of the subsystem are
  $-\theta$, $\alpha_i$, and $-\theta+\alpha_i$ along with their negatives;
   if $\g$ is of type $C$, the subsystem consists of  $-\theta$, $\alpha_i$, $-\theta+\alpha_i$, and $-\theta+2\alpha_i$ along with their negatives.
   These roots are a root system in ${\rm Span}_\bR\{-\theta,\alpha_i\}$. 
   A system of simple roots is $\{-\theta, \alpha_i\}$. Thus, if $V:={\rm Span}_\Q\{\alpha_i^\vee,-\theta^\vee\}$, then
   the operators $\partial_{-\theta},\partial_i:{\rm Sym}(\h_\Q^*)\to {\rm Sym}(\h_\Q^*)$ leave ${\rm Sym}(V^*)$ invariant and we have
   $$\partial_{-\theta}^2|_{{\rm Sym}(V^*)} = 0; \quad \partial_i \partial_{-\theta} \cdots  |_{{\rm Sym}(V^*)}= \partial_{-\theta} \partial_i\cdots  |_{{\rm Sym}(V^*)}\quad (m_{0i} \textrm{ factors}),\quad  1 \le i \le n \/; $$ here ${\rm Sym}(V^*)$ denotes the symmetric algebra of $V^*$. 
   On the other hand, if $V^\perp$ denotes the orthogonal complement of $V$ in
   $\h_\Q$, it follows from (\ref{delt}) that both $\partial_{-\theta}$ and $\partial_i$ restricted to
 ${\rm Sym}( (V^\perp)^*)$  are identically 0. We take into account that 
 ${\rm Sym}(\h^*)={\rm Sym}(V^*)\otimes {\rm Sym}((V^\perp)^*)$ and use the Leibniz property (\ref{alpfg}) to deduce that
 for $f\in {\rm Sym}(V^*)$ and $g\in {\rm Sym}((V^\perp)^*)$ we have 
 $$\partial_i(fg)=\partial_i(f) g \quad {\rm and} \quad \partial_{-\theta}(fg) = \partial_{-\theta}(f) g.$$
 This proves equations (\ref{chbr}).
\end{proof}

\section{The ring homomorphism $\mathrm{e}_1^*:\coh^*(G/B) \to \coh^*(\Fla)$}\label{s:p}
It is well known that the finite flag variety $G/B$ is homotopically equivalent to $K/T_\bR$ where $K$ is a maximal compact subgroup of $G$ and $T_\bR$ its maximal (real) torus. Unpublished results of Quillen (see e.g.~\cite{pressley:decamp,mitchell:quillenthm}) show that the affine flag variety $\Fla$ is homotopically equivalent to $LK/T_\bR$, where $LK$ is the group of (unbased) continuous loops $f:S^1 \to K$. Therefore there exists a continuous map $\e_1: LK/T_\bR \to K/T_\bR$ obtained by evaluating a loop $f(t)$ to $t=1$. This induces a ring homomorphism \[ \mathrm{e}_1^*:\coh^*(K/T_\bR) = \coh^*(G/B) \to \coh^*(LK/T_\bR) = \coh^*(\Fla) \/. \] Mare proved in \cite{Ma3} that \begin{equation}\label{eq:php} \mathrm{e}_1^*(\sigma_i) = \eps_i - m_i \eps_0, \quad 1 \le i \le n, \end{equation} where the integers $m_i$ are the coefficients $\theta^\vee = m_1 \alpha_1^\vee + \cdots  + m_{n} \alpha_n^\vee$ of the dual of the highest root $\theta \in \Pi$ in terms of the simple coroots. Since the ring $\coh^*(G/B)$ is generated (over $\Q$) by monomials in the Schubert divisors $\sigma_i$, the identity (\ref{eq:php}) determines the morphism $\mathrm{e}_1^*$. The main goal of this section is to construct the morphism $\mathrm{e}_1^*$ in the {\em algebraic} category, and to study its properties. In particular we will reprove the identity (\ref{eq:php}). The main new result is Theorem \ref{thm:crucial}, which states that $\mathrm{e}_1^*$ commutes with divided difference operators, i.e. for any $a \in \coh^*(G/B)$ and $w \in W_\af$ there is an identity \[ D_w \mathrm{e}_1^*(a) = \mathrm{e}_1^*(\pi(D_w) a) \] where $\pi:R_\af \to R$ was defined in Theorem \ref{T:ev}. This identity is the key step in proving that the new quantum product we will define later is closed. 


Consider the composition of morphisms:
\[ \e_1: \xymatrix{(\C^* \ltimes G(\C[t,t^{-1}]))/(\C^* \times T)\ar[rr]^{\quad =} && G(\C[t,t^{-1}])/T \ar[rr]^{\quad t=1}& & G/T} \]
where the first morphism ``cancels" the loop action by $\C^*$ and the second is determined by the natural evaluation map $G(\C[t,t^{-1}]) \to G$ at $t=1$.  We abuse notation and denote the composition by $\e_1$, as in the topological case. Note that the ``algebraic" morphism $\e_1$ does not extend to one $\Fla=\cG/\cB \to G/B$ because the evaluation map does not send the standard Iwahori subgroup of $G(\C[t,t^{-1}])$ into the Borel group $B \subset G$. However, as explained in \cite[p. 400]{kumar:kacmoody} there is a fibre bundle $(\C^* \ltimes G(\C[t,t^{-1}]))/(\C^* \times T) \to \cG/\cB$ in the strong topology with fibre the unipotent group $\mathcal{U} \simeq \cB/\mathcal{T}$ which is contractible. This induces a ring isomorphism $\coh^*(\cG/\cB) \simeq \coh^*((\C^* \ltimes G(\C[t,t^{-1}]))/(\C^* \times T))$ obtained by pulling back from $\coh^*(\cG/\cB)$ and an isomorphism between the corresponding homology groups. Same discussion applies and it gives a ring isomorphism $\coh^*(G/B) \simeq \coh^*(G/T)$ and a group isomorphism $\coh_*(G/B) = \coh_*(G/T)$. Therefore there are well-defined ring, respectively group homomorphisms
$ \mathrm{e}_1^*: \coh^*(G/B) \to \coh^*(\cG/\cB)$ and $(\e_1)_*: \coh_*(\cG/\cB) \to \coh_*(G/B)$.

\begin{prop}\label{prop:inj} The morphism $\e_1^*: \coh^*(G/B) \to \coh^*(\Fla)$ is injective. \end{prop} \begin{proof} 
There is an (algebraic) isomorphism $G(\C[t,t^{-1}]) \simeq G \times G(\C[t,t^{-1}])/G$ sending a loop $\tilde{g} = \tilde{g}(t)$ to $(\tilde{g}(1), \tilde{g} G)$. This induces an isomorphism $G(\C[t,t^{-1}])/T \to G/T  \times G(\C[t,t^{-1}])/G$ and $\e_1$ is given by composing this with the first projection. But the trivial fibration $
G/T  \times G(\C[t,t^{-1}])/G \to G/T$ gives an injective map $\coh^*(G/T) \to \coh^*(G/T  \times G(\C[t,t^{-1}])/G)$. This and the considerations before the proposition prove the claim. \end{proof}

Let $\omega \in \h^*$ be an integral (finite) weight and consider the embedding $\iota: \h^* \subset \h^*_{\af} = \h^* \oplus \C \delta \oplus \C \lambda_0$ (cf.~ { \S \ref{akm}.}) Denote by $\tilde{\omega}:=\iota(\omega) \in \h^*_\af$. 

\begin{prop}\label{prop:pullback} Let $\omega \in \h^*$ be an integral (finite) weight. Then $\mathrm{e}_1^* \mathcal{L}_\omega \simeq \mathcal{L}_{\tilde{\omega}}$ as line bundles on $(\C^* \ltimes G(\C[t,t^{-1}]))/(\C^* \times T)$. In particular, $\mathrm{e}_1^*(c_1(\mathcal{L}_\omega)) = c_1(\mathcal{L}_{\tilde{\omega}})$ in $\coh^2(\Fla)$. \end{prop}

\begin{proof} Let $\e_1': G(\C[t,t^{-1}])/T \to G/T$ be the morphism obtained by evaluation at $t=1$. It suffices to prove the statement when $\e_1$ is replaced by $\e_1'$. { Since $\mathcal{L}_\omega$ is a $G$-equivariant bundle and $\e_1'$ is equivariant with respect to the evaluation map $\e_1: G(\C[t,t^{-1}]) \to G$, it follows that $\e_1^*\mathcal{L}_\omega$ is $G(\C[t,t^{-1}])$-equivariant. Thus it is determined by the character of its fibre at the identity coset. It is easy to check that this character is $\iota(\omega)$.}\end{proof}

If the expansion of $\omega$ in terms of the finite fundamental weights is $\omega = a_1 \omega_1 + \cdots  + a_n \omega_n$ then one calculates that $\tilde{\omega} = a_1 \lambda_1 + \cdots  + a_n \lambda_n - \langle \omega, \theta^\vee \rangle \lambda_0$. In particular, if $\omega = \omega_i$ is a fundamental weight in $\h^*$ then $\tilde{\omega_i} = \lambda_i - m_i 
\lambda_0$ and Proposition \ref{prop:pullback} gives an algebraic proof of the identity (\ref{eq:php}).

{\em From now on in this section we consider homology and cohomology with rational coefficients.}
The following is the main result for this section. 
\begin{thm}\label{thm:crucial} For any $a\in \coh^*(G/B)$ and any $w\in W_\af$ there an identity 
$$D_w(\mathrm{e}_1^*(a)) = \mathrm{e}_1^*(\pi(D_w)(a))$$ where $\pi:R_\af \to R$ is defined in Theorem \ref{T:ev}.
\end{thm}
\begin{proof} The Schubert classes $\sigma_i$ generate the ring $\coh^*(G/B)$ (over $\Q$), therefore we may assume that $a= \sigma_{i_1} \cdots \sigma_{i_k}$ for $1\le i_1,\ldots, i_k\le n$. 
The proof is by double induction, first on length of $w$, then on  $k$. We take first $w = s_i$, for $0 \le i \le n$. For any $1\le j \le n$ we have  \[ D_i (\mathrm{e}_1^* (\sigma_j))= D_i (\eps_j-m_j\eps_0)=
\delta_{ij}-m_j\delta_{i0} = \mathrm{e}_1^*(\pi(D_i) \sigma_j)\] where we used the identity $\partial_\alpha(c_1(\mathcal{L}_\omega)) = \langle \omega, \alpha^\vee \rangle$ $(\alpha \in \Pi)$ for $\partial_i(\sigma_j)$ and $\partial_{-\theta}(\sigma_j)$.~This identity follows immediately from (\ref{delt}). Let now $a$ be a monomial in $\coh^*(G/B)$ of degree $\ge 2$, and write $a = a_1 a_2$ where the degree of both $a_1$ and $a_2$ is strictly less than $\deg a$. By the Leibniz rule from Proposition \ref{prop:leibniz}: \begin{equation}\label{a1a2}D_i(\mathrm{e}_1^*(a_1 a_2)) = D_i(\mathrm{e}_1^*(a_1)) \mathrm{e}_1^*(a_2) + \mathrm{e}_1^*(a_1) D_i(\mathrm{e}_1^*(a_2)) - c_1(\mathcal{L}_{\alpha_i}) D_i(\mathrm{e}_1^*(a_1)) D_i(\mathrm{e}_1^*(a_2)) \/. \end{equation}
By induction hypothesis $D_i(\mathrm{e}_1^*(a_s)) = \mathrm{e}_1^*(\pi(D_i) a_s)$ for $s=1,2$. Further, $c_1(\mathcal{L}_{\alpha_i}) = \mathrm{e}_1^*(c_1(\mathcal{L}_{\alpha_i}))$ if $i \neq 0$, and $c_1(\mathcal{L}_{\alpha_0}) = \mathrm{e}_1^*(c_1(\mathcal{L}_{-\theta}))$ by Proposition \ref{prop:pullback} and using that $c_1(\mathcal{L}_\delta) = 0$ by Proposition \ref{prop:delta trivial}. Since $\mathrm{e}_1^*$ is a ring homomorphism, and invoking the Leibniz rule for $\partial_i$ respectively $\partial_{-\theta}$, these identities show that the right hand side of (\ref{a1a2}) equals $\mathrm{e}_1^*(\pi(D_i) (a_1 a_2))$. This finishes the case when $\ell (w) =1$, and assume now that $\ell(w) >1$. Write $w= w' s_i$, with $\ell(w') < \ell(w)$. Then $D_w = D_{w'} D_i$ and for $a \in \coh^*(G/B)$, \[D_w (\mathrm{e}_1^*a) = D_{w'} D_i(\mathrm{e}_1^*a) = D_{w'}(\mathrm{e}_1^*(\pi(D_i)a)) = \mathrm{e}_1^*(\pi(D_{w'} )\pi(D_i) a)  = \mathrm{e}_1^* (\pi(D_w) a) \/, \] where we used the induction hypothesis and the fact that $\pi$ is a ring homomorphism. This finishes the proof.
\end{proof} 

\begin{remark} Using the projection formula and Proposition \ref{prop:pullback} we obtain that \[ \int_{G/B} \sigma_j \cap (\e_1)_*[X(s_i)] = \int_{\Fla} \mathrm{e}_1^*(\sigma_j) \cap [X(s_i)] = \langle \lambda_j - m_j \lambda_0, \alpha_i^\vee \rangle \/, \] therefore we have the following identities in $\coh_2(G/B)$:
\begin{equation}\label{eq:pfp} (\e_1)_*[X(s_i)] = \begin{cases} [X(s_i)] & \textrm{ if } i > 0; \\ - m_1 [X(s_1)] - \cdots  - m_n [X(s_n)] & \textrm{ if } i=0 \/. \end{cases} \end{equation} 
Further, since the (affine or finite) Poincar{\'e} pairing is nondegenerate, one can define an action of the divided difference operators $D_u$ ($u \in W_\af$) and $\partial_v$ ($v \in W$) on {\em homology}, by duality: \[ \int_{\Fla} a \cap D_u(b) := \int_{\Fla} D_u(a) \cap b; \quad a \in \coh^*(\Fla), b \in \coh_*(\Fla) \/, \] and similarly for $\partial_v$. 
Then Theorem \ref{thm:crucial} and the projection formula implies that for any $w \in W_\af$ and any $b \in \coh_*(\Fla)$, $(\e_1)_*(D_w(b)) = \pi(D_w) (\e_1)_*(b)$.
\end{remark}

\section{Curve neighborhoods of affine Schubert varieties}

The goal of this section is to define curve neighborhoods of Schubert varieties in the affine flag manifolds. This notion sits at the heart of the definition of the affine quantum Chevalley operators from the next section. The definition extends the one  from \cite{bcmp:qkfinite,buch.mihalcea:nbhds} where curve neighborhods played a central role in the study of quantum cohomology and quantum K-theory of finite flag manifolds.

Recall that if $w \in W_\af$ then $X(w)$ denotes the Schubert variety $\overline{\cB w \cB/ \cB}\subset \Fla$ and that $X(w)$ is a complex projective algebraic variety of dimension $\dim X(w) = \ell(w)$. Denote also by $X(w)^o$ the Schubert cell $\cB w \cB/\cB$ which is the $\cB$-orbit of the $\mathcal{T}$-fixed point $e_w:= w\cB$. Recall that $\cB$ acts transitively on $X(w)^o$ and that $X(w)^o \cong\C^{\ell(w)}$ (\cite[Proposition 7.4.16]{kumar:kacmoody}). The Schubert variety is the union of its Schubert cells: $X(w) = \bigcup_{v \le w} X(v)^o$. The ind-variety structure of $\Fla$ is given by the filtration $\mathcal{X}_1 \subset \mathcal{X}_2 \subset \ldots. \subset \Fla$ where $\mathcal{X}_n = \bigcup_{v \in W_\af, \ell(v) \le n} X(v)$.

Let $d \in \coh_2(\Fla)$ be an effective degree. A {\em rational curve} of degree $d$ in $\Fla$ is a morphism of (ind-)varieties $f: C \to \Fla$ where $C$ is an algebraic curve of arithmetic genus $0$ (i.e. a tree of $\P^1$'s). In particular, the image of $f$ must be included in some stratum $\mathcal{X}_n$. We will often abuse notation and we will denote by $C$ the (scheme-theoretic) image of $f$. We can find $w \in W_\af$ sufficiently large such that $d \in \coh_2(X(w))$.~Recall from \cite[Theorem 1]{fulton.pandharipande:notes} that there exists a projective scheme $\Mb_{0,2}(X(w), d)$ which parametrizes $2$-point, genus $0$ stable rational curves in $X(w)$ of degree $d$. Denote the evaluation maps at the two points by $\ev_1,\ev_2:\Mb_{0,2}(X(w),d) \to X(w)$.

\begin{defn} Let $u,w$ be in the affine Weyl group such that $u \le w$ in the Bruhat ordering (thus $X(u) \subset X(w)$). Fix an effective degree $d \in \coh_2(X(w))$. The {\em $w$-curve neighborhood} $\Theta_d^w(u)$ of $X(u)$ is defined by \[ \Theta_d^w(u) := \ev_2( \ev_1^{-1} X(u)) \subset X(w) \/ \] endowed with the reduced scheme structure. \end{defn} Because the evaluation maps are proper, this is a closed { subscheme} of $X(w)$, and it consists of the closure of locus of points $x \in X(w)$ such that there exists a rational curve $f:\P^1 \to X(w)$ with $x \in f(\P^1)$ and $f(\P^1) \cap X(u) \neq \emptyset$.

\begin{thm}\label{T:existence} Let $u \in W_\af$ and $d \in \coh_2(\Fla)$ an effective degree. There exists a unique projective variety $\Theta_d(u) \subset \Fla$, called the {\em curve neighborhood} of $X(u)$, satisfying the properties:
\begin{itemize} \item $\Theta_d^v(u) \subset \Theta_d(u)$ for any $v \in W_\af$ such that $v \ge u$; \item there exists $w \in W_\af$ depending on $u$ and $d$ with $\Theta_d(u) = \Theta_d^w(u)$. \end{itemize}\end{thm}

{\em A priori}, this theorem can be deduced from the work of Atiyah \cite{atiyah}, who proved that the locus of points on rational curves of a fixed degree $d$, intersecting a fixed finite-dimensional subscheme of $\Fla$ is finite dimensional. However, we will give a different proof of this statement by analyzing chains of $\mathcal{T}$-stable of rational curves through $X(u)$, in the same spirit as in \cite{fulton.woodward}. We need the following general result:

\begin{lemma}\label{lemma:Tcurve} Let $Z$ be an irreducible complex projective variety and $H$ a complex algebraic torus acting on $Z$. Let $\Omega_1,\Omega_2 \subset Z$ be $H$-stable subschemes and assume that there exists a rational curve $C$ of degree $d= [C] \in \coh_2(Z)$ intersecting $\Omega_1$ and $\Omega_2$. Then there exists an $H$-stable rational curve of degree $d$ intersecting $\Omega_1$ and $\Omega_2$.\end{lemma}

\begin{proof} Consider the projective scheme $\Mb:=\Mb_{0,2}(Z, d)$ parametrizing $2$-point, genus $0$ stable maps of degree $d$ to $Z$. Define the Gromov-Witten variety $GW_d(\Omega_1,\Omega_2) := \ev_1^{-1}(\Omega_1) \cap \ev_2^{-1}(\Omega_2)$, which is a closed subscheme of $\Mb$, nonempty by hypothesis. The $H$ action on $Z$ induces an action on $\Mb$, and because $\Omega_1,\Omega_2$ are $H$-stable, it also induces an action on $GW_d(\Omega_1,\Omega_2)$. Borel fixed point theorem (see e.g. \cite[21.2]{humphreys}) implies that there exists an $H$-fixed point on $GW_d(\Omega_1,\Omega_2)$, and this corresponds to a rational curve having the claimed properties. \end{proof}

Recall that the moment graph of $\Fla$ is the graph with vertices given by $w \in W_\af$ and edges the irreducible $\mathcal{T}$-stable curves in $\Fla$. By \cite[Proposition 12.1.7]{kumar:kacmoody} there exists an edge between $u$ and $v$ iff $v=u s_\alpha$ where $\alpha$ is an affine real root. 

\begin{cor}\label{cor:Tstable} Let $d \in \coh_2(\Fla)$ be an effective degree and $u,v,w \in W_\af$ such that $u \le v \le w$. Assume that the intersection $\Theta_d^w(u) \cap X(v)^o$ is nonempty.~Then there exists a $\mathcal{T}$-stable rational curve of degree $d$ in $X(w)$ joining $e_v$ with a $\mathcal{T}$-fixed point $e_{z}$ where $z \le u$. \end{cor}

\begin{proof} By hypothesis there exists a rational curve $C$ of degree $d$ in $X(w)$ which intersects $X(u)$ and the cell $X(v)^o$. Since $\mathcal{B}$ acts transitively on $X(v)^o$, and because $X(u)$ is $\cB$-stable, we can find $ b \in \cB$ such that the translate $b.C$ contains $e_v$ and intersects $X(u)$. This shows that the Gromov-Witten variety $GW_d(X(u), \{e_v\}) \subset \Mb_{0,2}(X(w),d)$ is non-empty. Then we invoke Lemma \ref{lemma:Tcurve} to obtain a $\mathcal{T}$-stable rational curve of degree $d$ joining $X(u)$ to $e_v$. But any irreducible $\mathcal{T}$-stable curve in $\Fla$ is isomorphic to $\P^1$ and it joins a $\mathcal{T}$-fixed point $e_b$ ($b \in W_\af$) to $e_{b s_\alpha}$ where $\alpha$ is an affine real root. Therefore any $\mathcal{T}$-stable curve intersecting $X(u)$ contains a $\mathcal{T}$-fixed point $e_{z}$ in $X(u)$, and this finishes the proof. \end{proof}

\begin{proof}[Proof of Theorem \ref{T:existence}] Uniqueness is clear, so we will prove the existence of $\Theta_d(u)$. Note that for any $w \in W_\af$, the curve neighborhood $\Theta_d^w(u)$ is either empty, or $\cB$-stable; in the latter case it must be a finite union of ($\cB$-stable) Schubert varieties. Let now $w$ vary over all elements $w \in W_\af$ such that $\ell(w) \le n$, for $n$ fixed. Corollary \ref{cor:Tstable} implies that if there exists a rational curve $C$ of degree $d$, included in some $X(w)$, and intersecting the Schubert variety $X(u)$ and the Schubert cell $X(v)^o$, then there is a path of degree $d$ in the moment graph of $\Fla$ joining a $\mathcal{T}$-fixed point in $X(u)$ to $e_v$. As $n$ increases, since $d$ is a fixed finite degree, there are finitely many such paths which contain a fixed point from $X(u)$, and there are finitely many such $\mathcal{T}$-fixed points in $X(u)$. This implies that the set of those $v \in W_\af$ such that $e_v$ belongs to $\bigcup_{\ell(w) \le n} \Theta_d^w(u)$ stabilizes, which in turn implies that the variety $\bigcup_{\ell(w) \le n} \Theta_d^w(u)$ stabilizes for $n \ge n_0$, for some $n_0$. Then take $n \ge n_0$ and $w \in W_\af$ such that $w \ge v$ for any $v \in W_\af$ with $\ell(v) \le n$. The variety $\Theta_d(u):= \bigcup_{\ell(w) \le n} \Theta_d^w(u)$ satisfies all the requirements in the theorem.\end{proof}

We noticed in the proof that the curve neighborhood $\Theta_d(u)$ is a finite union of Schubert varieties.
The next corollary gives more precise information about this union. To state it, we recall the definition of the {\em Hecke product} $u \cdot v$ of two elements $u, v \in W_\af$. If $v= s_i$ is a simple reflection then
\[ u \cdot s_i = \begin{cases} us_i & \textrm{ if } \ell(us_i) > \ell(u) \\ u & \textrm{ otherwise } \/. \end{cases} \] In general, take a reduced expression $v=s_{i_1} \cdots  s_{i_k}$ and define $u \cdot v := (\ldots((u\cdot s_{i_1}) \cdot s_{i_2})\cdot \ldots \cdot s_{i_k})$. This endows $W_\af$ with a structure of an associative monoid; we refer to \cite[\S 3]{buch.mihalcea:nbhds}
for further details.

\begin{cor}\label{cor:zd} Let $z_d^1,\ldots,z_d^p$ be the Weyl group elements such that $\Theta_d(1)= X(z_d^1) \cup \ldots \cup X(z_d^p)$. Then:

(a) $z_d^i$ are the maximal elements in the Bruhat ordering so that there exists a path of degree $\le d$ in the moment graph of $\Fla$ containing $1$ and $z_d^i$.

(b)  Fix $1 \le i \le p$. Then there exist affine positive real roots $\alpha_{i_1}, \ldots, \alpha_{i_r}$ (not necessarily simple, and depending on $i$) such that $\sum_{j=1}^r \alpha_{i_j}^\vee = d$ and $z_d^i = s_{\alpha_{i_1}} \cdot \ldots\cdot  s_{\alpha_{i_r}}$, where $\cdot$ denotes the Hecke product.

(c) $\Theta_d(u) = X(u \cdot z_d^1) \cup \ldots \cup X(u \cdot z_d^p)$. 
\end{cor}

\begin{proof} The proof is similar to that of results proved in the finite dimensional case in \cite[\S 5]{buch.mihalcea:nbhds} so we will be brief. Part (a) follows from the construction of the curve neighborhood in the proof of Theorem \ref{T:existence}. Part (b) follows from properties of the Hecke product (see e.g. \cite[Proposition 3.1]{buch.mihalcea:nbhds}) and from the fact that there must be a $\mathcal{T}$-stable chain of curves (possibly non-reduced) from $1$ to each of $z_d^i$ of degree equal to precisely $d$.  Part (c) is the same as in the case when $X=G/B$, using Corollary \ref{cor:Tstable}; see \cite[Theorem 5.1]{buch.mihalcea:nbhds}.\end{proof}

\begin{remark} Unlike in the finite-dimensional case when $X=G/B$, the curve neighborhood of a Schubert variety may no longer be irreducible. Indeed, take $\Fla$ to be the flag variety for the affine Kac-Moody group of type $A_1^1$ and take $d = \alpha_0^\vee + \alpha_1^\vee$ the degree corresponding to the imaginary (co)root. Then Corollary \ref{cor:zd} implies that $\Theta_d(1) = X(s_0 s_1) \cup X(s_1 s_0)$, which is obviously reducible
(see Figure 1). However, we will prove in Lemma \ref{lemma:eq} below that the curve neighborhoods relevant to the quantum Chevalley product remain in fact irreducible. { A description for the curve neighborhoods of Schubert varieties in the affine Lie type $A_1^1$ has been obtained in \cite{mihalcea.norton:A11}.} \end{remark}
\begin{figure}[h]
\begin{center}
\epsfig{figure=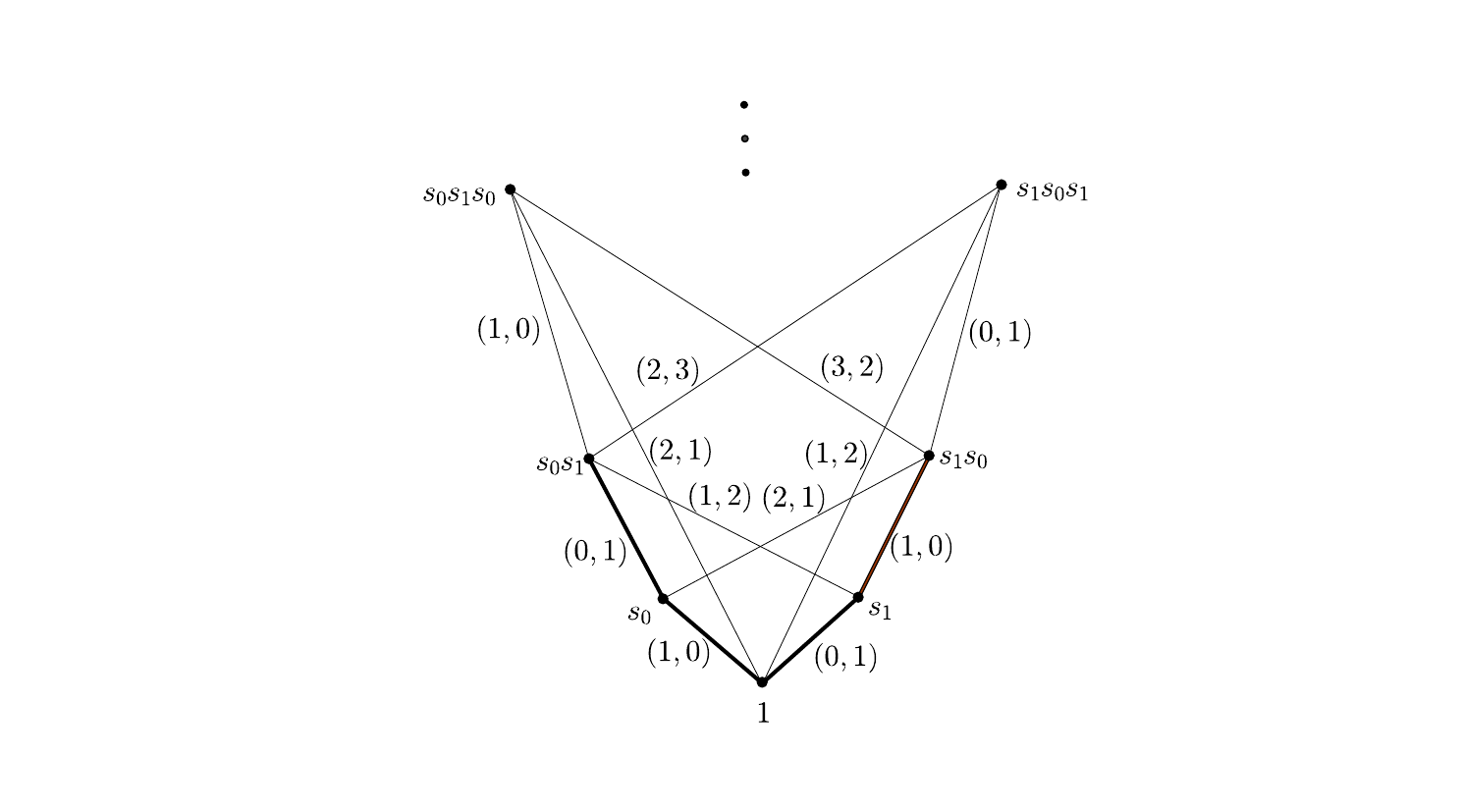,width=6.7in}
\end{center}
\begin{center}
\parbox{5in}{
\caption[]{The moment graph for the affine flag variety of type $A_1^1$. The weight $(k,\ell)$ symbolizes
the coroot $k\alpha_0^\vee + \ell \alpha_1^\vee$.  The two paths
of weight $\alpha_0^\vee+\alpha_1^\vee$ starting at $1$ are $1\to s_0\to s_0s_1$ and $1\to s_1\to s_1s_0$. }\label{figure} }
\end{center}
\end{figure}

\section{The affine quantum Chevalley operators}\label{section:rule} The goal of this section is to use the notion of curve neighborhoods from the previous section to define the affine quantum Chevalley operators. These are the main new objects in this paper, and their study will lead in \S \ref{sec:affqcohrings} below to definitions of certain affine quantum cohomology rings. The main technical result is Theorem \ref{thm:afqops}, which gives an explicit combinatorial formula for the action of these operators.

Recall that $\coh^*(\Fla)$ is a free graded $\Z$-module, with a basis given by Schubert classes $\eps_w$ ($w \in W_\af$) such that $\deg \eps_w = \ell(w)$. 
Denote by $q=(q_0, q_1, \ldots, q_n)$ the sequence of {\em quantum parameters}, which are indexed by the basis of $\coh_2(\Fla)$. This is analogous to the finite case, where the quantum parameters are indexed by the basis of the homology group $\coh_2$. We set $\deg q_i = 2$ for $0 \le i \le n$. For $d= d_0 [X(s_0)] + \cdots  + d_n [X(s_n)] \in \coh_2(\Fla)$ an {\em effective degree} (i.e satisfying $d_i \in \Z_{+}$), we denote by $q^d: =q_0^{d_0} \cdots  q_n^{d_n}$. Under the identification $\coh_2(\Fla) \simeq \oplus_{i=0}^n \Z \alpha_i^\vee$ a degree can be regarded as an element of the coroot lattice, so it has a well defined {\em height} defined by $\he(d) = d_0 + \cdots + d_n$. Then $\deg q^d = 2 \he(d)$. Consider $\quantum^*(\Fla):=\coh^*(\Fla) \otimes_\Z \Z[q]$ the free $\Z[q]$-module graded in the obvious way.

\begin{defn}\label{def:chevops} Let $0 \le i \le n$. Define the $\Z[q]$-linear, degree $1$, endomorphism of graded $\Z[q]$-modules $\Lambda_i: \quantum^*(\Fla) \to \quantum^*(\Fla)$ by 
\[ \Lambda_i(\eps_u) = \eps_i \cdot \eps_u + \sum_{d \in \coh_2(\Fla), d \neq 0} \gw{\eps_u, \eps_i,[X(w)]}{d} q^d \eps_w \/, \]
where $\eps_i \cdot \eps_w$ is the ordinary Chevalley multiplication in $\coh^*(\Fla)$ from (\ref{chaff}) above and $\gw{\eps_u, \eps_i,[X(w)]}{d}$ is the {\em affine Gromov-Witten invariant} defined by 
\begin{equation}\label{eq:defafGW} \gw{\eps_u, \eps_i,[X(w)]}{d} = \langle \lambda_i, d \rangle \cdot \int_{\Fla} \eps_u \cap [\Theta_d(w)] \/; \end{equation} here $\lambda_i$ is the affine fundamental weight from \S \ref{sec:afm}. Notice that the requirement that $\Lambda_i$ has degree $1$ implies that $\ell(u) + 1 = \deg q^d + \ell(w)$.
\end{defn}

\begin{remark}\label{rmk:deggq} To motivate that $\deg q_i=2$, recall that for the finite flag variety $G/B$ this degree arises as $\int_{G/B} c_1(T_{G/B}) \cap [X(s_i)]$ for $i \neq 0$, where $T_{G/B}$ is the tangent bundle of $G/B$. Since $\Fla$ is infinite dimensional, an argument is needed to show that the corresponding tangent bundle, or at least an analogue of its first Chern class, exists. As observed by Guest and Otofuji \cite{go} and Mare \cite{Ma3}, this was done in differential geometric setting by Freed \cite{freed:loop}. More recently Kashiwara provided an algebro-geometric approach in \cite[Appendix]{kumar:positivity}, where he calculated the canonical bundle of the ``thick" version of the flag manifold $\Fla$. \end{remark}

\begin{remark} The definition of the affine Gromov-Witten invariants is the natural generalization of a formula for the ordinary Gromov-Witten invariants on $G/B$ of the form $\gw{\sigma_u, \sigma_{s_i}, [X(w)]}{d}$ (for $i \neq 0$ and $w \in W$) which involves curve neighborhoods of Schubert varieties; see \cite[\S 7]{buch.mihalcea:nbhds}. In the finite case this formula is obtained from the divisor axiom in Gromov-Witten theory and a push-forward formula $[(\ev_1)_*(\ev_2^*[X(w)]) \in \coh_*(G/B)$ involving the evaluation maps $\ev_i: \Mb_{0,3}(G/B,d) \to G/B$; see \cite{bcmp:qkfinite} for a proof of this push-forward formula in cohomology and K-theory. However, the affine flag variety $\Fla$ is infinite dimensional, and an analogous moduli space has not been constructed. It would be interesting to construct the affine invariants above using moduli spaces. \end{remark}

Corollary \ref{cor:zd} implies that the quantity $\int_{\Fla} \eps_u \cap [\Theta_d(w)]$ is non-zero only if $u = w \cdot z_d^i$ for some $i$ and in this case it equals $1$. The latter condition turns out to be closely related to the upper bound $\ell(s_\alpha) \le 2 \he(\alpha^\vee) -1$, where $\alpha$ is a positive real affine root. This bound is well-known for finite root systems - see
for instance \cite[Lemma 4.3]{brenti.fomin.post:mixed}, \cite[Lemma 3.2]{Ma-1} or \cite[Theorem 6.1]{buch.mihalcea:nbhds}. To prove it in the affine case we need the following result: 

\begin{lemma}\label{lemma:instep} Let $\alpha$ be a positive, real, non-simple root. Then there exists a simple root $\alpha_i \in \Delta_\af$ such that $\langle \alpha_i, \alpha^{\vee}\rangle >0$. For any such $\alpha_i$ we have that $\ell(s_{s_i(\alpha)}) = \ell(s_\alpha) - 2$, $\he(s_i(\alpha)) = \he(\alpha) - \langle \alpha, \alpha_i^\vee \rangle$ and $\he(s_i(\alpha^\vee)) = \he(\alpha^\vee) - \langle \alpha_i, \alpha^\vee \rangle$. \end{lemma}

\begin{proof} Because $\alpha$ is not simple, there exists a simple reflection $s_i \in W_\af$ such that $\ell(s_\alpha s_i) < \ell(s_\alpha)$. Then $s_\alpha(\alpha_i) = \alpha_i - \langle \alpha_i, \alpha^\vee \rangle \alpha$ is a  negative root, thus both $\langle \alpha_i, \alpha^\vee \rangle >0$ and $\langle \alpha, \alpha_i^\vee \rangle >0$. Further, $s_{i}s_{\alpha}(\alpha_i)=(\langle\alpha,\alpha_i^{\vee}\rangle \langle \alpha_i, \alpha^{\vee}\rangle -1)\alpha_i-
\langle \alpha_i,\alpha^{\vee}\rangle \alpha$ is also a negative root. This proves the identity on lengths and the one on heights is obvious.
\end{proof}

\begin{prop}\label{letalpha} Let $\alpha\in \Pi^{re,+}_\af$ be an affine positive real root. Then there is an inequality  
$\ell(s_\alpha) \le 2 \min \{ \he({\alpha^\vee}), \he (\alpha) \} - 1$. 
If the equality $\ell(s_\alpha) = 2 \he (\alpha^\vee) - 1$ holds then $\alpha^\vee < c= \alpha_0^\vee +\theta^\vee$ where $c$ is the imaginary coroot defined in \S \ref{akm}. \end{prop}

%

\begin{proof} We use induction on $\he(\alpha^\vee)$. If $\ell(s_{\alpha})=1$, then $\alpha$ and $\alpha^{\vee}$ are simple, so $\he (\alpha^{\vee})= \he (\alpha) =1$. If $\he(\alpha) >1$ then { with the $\alpha_i$ from} Lemma \ref{lemma:instep} and by induction hypothesis \[ \ell (s_{\alpha})=\ell(s_{i}s_{\alpha}s_{i})+2\leq
2 \min \{ \he(s_i(\alpha)), \he( s_{i}(\alpha)^{\vee}) \} - 1+2 \le 2 \min \{ \he(\alpha), \he(\alpha^{\vee}) \} - 1.\] This proves the first part of the lemma. Assume now that $\ell(s_\alpha) = 2 \he(\alpha^\vee) -   1$. 
Recall that $(\Pi^{re,+}_\af)^\vee= \Pi^{\vee,+} \cup \{ m c+\beta^\vee:  m>0, \beta^\vee \in\Pi^\vee \}$. If $\alpha \in \Pi^{\vee,+}$ does not satisfy the inequality $\alpha^\vee < c$ then $\alpha^\vee = c + \gamma^\vee$, where $\gamma^\vee \in \Pi_\af^{\vee,+}$ is a positive real coroot. If $\gamma = \alpha_i$ for $i \neq 0$ then $\langle \alpha_i,\alpha^\vee\rangle =\langle \alpha_i, c+\alpha_i^\vee \rangle = \langle \alpha_i,\alpha_i^\vee \rangle = 2$; if $\gamma = \alpha_0$ then $\langle \alpha_0,\alpha^\vee\rangle =\langle \alpha_0, 2c- \theta^\vee \rangle = \langle \theta, \theta^\vee \rangle = 2$. Thus by Lemma \ref{lemma:instep} \[ \ell(s_\alpha) = \ell(s_i s_\alpha s_i) + 2 \le  2 \he(s_i(\alpha^\vee)) + 1 = 2 \he(\alpha^\vee) - 3 \/. \]
Let now $\he(\gamma^\vee) >1$. Because $\gamma^\vee$ is real and positive, Lemma \ref{lemma:instep} shows that there exists a simple root $\alpha_i$ such that $\langle \alpha_i, \gamma^\vee \rangle > 0$. But then $\langle \alpha_i, \alpha^\vee \rangle = \langle \alpha_i, c + \gamma^\vee \rangle = \langle \alpha_i, \gamma^\vee \rangle >0$. Together with the fact that $s_i(\gamma^\vee)>0$ (because $\gamma$ is not simple) this implies that \[ s_i(\alpha^\vee) =  \alpha^\vee - \langle \alpha_i, \alpha^\vee \rangle \alpha_i^\vee = c+ \gamma^\vee - \langle \alpha_i, \gamma^\vee \rangle \alpha_i^\vee  = c + s_i (\gamma^\vee) > c \/. \] Then $s_i(\alpha)$ is a positive real root with $\he (s_i(\alpha^\vee)) <\he (\alpha^\vee)$ and $s_i(\alpha)^\vee >c$. By the induction hypothesis and Lemma \ref{lemma:instep} 
\[ \ell (s_{\alpha})=\ell(s_{i}s_{\alpha}s_{i})+2 <
2\he(s_{i}(\alpha)^{\vee})+1=2\he(\alpha^{\vee})- 2 \langle \alpha_i, \alpha^{\vee}\rangle +1  \le
2\he(\alpha^{\vee})-1 \/. \]
Thus $\alpha^\vee <c$ and the proof is finished. \end{proof}

The following result is the key to the explicit calculation of the action of $\Lambda_i$, and it generalizes to $\Fla$ a similar result from \cite{buch.mihalcea:nbhds} in the case $G/B$.

\begin{lemma}\label{lemma:eq} Let $d \in \coh_2(\Fla)$ be an effective, non-zero degree. Assume that $\ell(u) + 1 = \deg q^d + \ell(w)$ and that $u = w \cdot z_d^i$ in the Bruhat ordering for some $i$. Then the following hold:
\begin{enumerate}

\item $d = \alpha^\vee$ for some real affine coroot and $d <  c= \alpha_0^\vee + \theta^\vee$;

\item $\ell(s_\alpha) = \deg q^{\alpha^\vee}-1$;

\item the curve neighborhoods $\Theta_d(X(id))$ and $\Theta_d(X(w))$ are given by \[ \Theta_d(X(id))= X(s_\alpha); \quad \Theta_d(X(w)) = X(w s_\alpha) \/. \]
\end{enumerate} \end{lemma}

\begin{proof} By Corollary \ref{cor:zd} we can find affine, positive real roots such that $z_d^i = s_{\alpha_{i_1}} \cdot \ldots \cdot s_{\alpha_{i_r}}$ and $\sum_{j=1}^r \alpha_{i_j}^\vee = d$. Since $\ell(s_\alpha) \le 2 \he(\alpha^\vee) -1$ by Proposition \ref{letalpha}, and because $\deg q^d = 2 \he(d)$, \[ \begin{split} \ell(u) = \ell(w \cdot z_d^i) \le \ell(w) + \ell(s_{\alpha_{i_1}} \cdot \ldots \cdot s_{\alpha_{i_r}}) \le \ell(w) + \sum_{j=1}^r (2 \he(\alpha_{i_j}^\vee) - 1)\\ = \ell(w) + 2 \he(d) - r = \ell(w) + \deg q^d - r = \ell(u) + 1 - r 
\/. \end{split} \]
This implies that $r=1$ and that we have equality throughout. Thus $z_d^i = s_\alpha$ and $d=\alpha^\vee$ with $\ell(s_\alpha) = 2 \he(\alpha^\vee) -1$. Moreover, the Hecke product $w \cdot s_\alpha$ coincides with the usual product $w s_\alpha$ in $W_\af$. The equality of curve neighborhoods follows from Corollary \ref{cor:zd}. This finishes the proof.\end{proof}

Lemma \ref{lemma:eq} implies immediately the following formula for $\Lambda_i$:

\begin{thm}\label{thm:afqops} Let $0 \le i \le n$. Then the affine quantum Chevalley operator $\Lambda_i$ is given by:\[ \Lambda_i(\eps_u) = \eps_i \cdot \eps_u + \sum \langle \lambda_i, \alpha^\vee \rangle q^{\alpha^\vee} \eps_{u s_\alpha} \] where $\eps_i \cdot \eps_u$ is given by the Chevalley formula from (\ref{chaff}), and the sum is over affine, positive, real roots $\alpha$ satisfying $\ell(us_\alpha) = \ell(u) +1 - \deg q^{\alpha^\vee}$.~Further, any such $\alpha$ must satisfy $\ell(s_\alpha) = 2 \he(\alpha^\vee) -1$, therefore in particular $\alpha^\vee < c = \alpha_0^\vee + \theta^\vee$. 
\end{thm}

The condition on $\alpha$ and the definition of the divided difference operator $D_{s_\alpha}$ on $\coh^*(\Fla)$ implies that the expression for $\Lambda_i$ can be rewritten as:
\begin{equation}\label{eq:qafBGG} \Lambda_i(\eps_u) = \eps_i \cdot \eps_u  + \sum \langle \lambda_i, \alpha^\vee \rangle q^{\alpha^\vee} D_{s_\alpha}(\eps_{u}) \end{equation} where the sum is over affine, positive, real roots $\alpha$ with $\ell(s_\alpha) = 2 \he(\alpha^\vee) - 1$. This generalizes the quantum Chevalley operators from \cite{Pe,Ma0} and it will be used in the next section to prove that various quantum products we will define are closed.

\section{Two affine quantum cohomology rings}\label{sec:affqcohrings} {\bf From now on all the cohomology rings will be taken with coefficients over $\Q$.} 
Denote by $\mathrm{H}^\#(\Fla)$ the graded subring of $\coh^*(\Fla)$ generated by the Schubert divisors $\eps_i$ for $0 \le i \le n$. Unlike in the finite case, the Schubert divisors {\em do not} generate the cohomology ring (even over $\Q$), therefore $\mathrm{H}^\#(\Fla)$ is a strict subring of $\coh^*(\Fla)$. Define the graded $\Q[q]$-modules \[ \quantum^\#(\Fla) := \mathrm{H}^\#(\Fla) \otimes_\Q \Q[q]; \quad \quantum^*_\af (G/B):= \coh^*(G/B) \otimes_\Q \Q[q] \/; \] the first is a graded submodule of $\coh^*(\Fla) \otimes_\Q \Q[q]$, and the grading on the second is determined by $\deg \sigma_w = \ell(w)$ (for $w \in W$) and $\deg q_i = 2$. The goal of this section is to define the main new rings in this paper: \begin{itemize} \item a ring structure on the $\Q[q]$-module $\quantum^*_\af(G/B)$; \item a ring structure on the $\Q[q]/\langle q^c \rangle$-module  $\quantum^\#(\Fla)/\langle q^c \rangle$ where $q^c = q_0 q^{\theta^\vee}$ is the product of the quantum parameters determined by the imaginary coroot. \end{itemize} Recall from \S \ref{s:p} the formula $\mathrm{e}_1^*(\sigma_i) = \eps_i - m_i \eps_0$ ($1 \le i \le n$) where the integers $m_i$ are defined by $\theta^\vee = \sum_{i=1}^n m_i \alpha_i^\vee$. Our first result is:
\begin{thm}\label{thm:invChev} (a) Let $0 \le i \le n$. Then the Chevalley operator $\Lambda_i$ preserves the submodule $\quantum^\#(\Fla)$. 

(b) Let $1 \le i \le n$. Then the modified Chevalley operator $\Lambda_i - m_i \Lambda_0$ preserves the submodule $\mathrm{e}_1^*(\coh^*(G/B)) \otimes_\Q \Q[q]$. \end{thm}

\begin{proof} To prove (a), we use the expression (\ref{eq:qafBGG}) for $\Lambda_i$. It suffices to show that the divided difference operators $D_k$ preserve the subring $\mathrm{H}^\#(\Fla)$, for all $0 \le k \le n$. This reduces to checking whether $D_k(\eps_{i_1} \cdots  \eps_{i_j}) \in \mathrm{H}^\#(\Fla)$. This follows from induction on the number of terms in the monomial $\eps_{i_1} \cdots  \eps_{i_j}$: if $j=1$ then $D_k(\eps_i) = \langle \lambda_i, \alpha_k^\vee \rangle$, and if $j>1$ one uses the Leibniz formula from Proposition \ref{prop:leibniz} and the induction hypothesis. We now turn to the proof of part (b). Let $w \in W$. The identities $D_{s_\alpha}(\mathrm{e}_1^*(\sigma_w)) = \mathrm{e}_1^*(\pi(D_{s_\alpha})(u))$ (Theorem \ref{thm:crucial}) and $\mathrm{e}_1^*(\sigma_i) = \eps_i - m_i \eps_0$ (Proposition \ref{prop:pullback}) imply that \begin{equation}\label{eq:Lambdabar} (\Lambda_i - m_i \Lambda_0)(\mathrm{e}_1^*(\sigma_w))  = \mathrm{e}_1^*(\sigma_i) \cdot \mathrm{e}_1^*(\sigma_w) + \sum \langle \lambda_i - m_i \lambda_0, \alpha^\vee \rangle q^{\alpha^\vee} \mathrm{e}_1^*(\pi(D_{s_\alpha})(\sigma_w)) \/, \end{equation} where the sum is as in (\ref{eq:qafBGG}). Since $\mathrm{e}_1^*:\coh^*(G/B) \to \coh^*(\Fla)$ is a ring homomorphism, this finishes the proof. \end{proof}

To construct the quantum products advertised above, we need the following commutativity properties of the operators $\Lambda_i$ and $\Lambda_i - m_i \Lambda_0$. The proof will be given in the next section.

\begin{thm}\label{thm:comm} (a) The operators $\Lambda_i$ commute up to the imaginary coroot, i.e.~for any $w \in W_\af$ and any $0 \le i, j \le n$ we have \[ \Lambda_i \Lambda_j (\eps_w) = \Lambda_j \Lambda_i(\eps_w)  \mod q^c = q_0 q^{\theta^\vee} \/. \] 

(b) The operators $\Lambda_i - m_i \Lambda_0$ commute (without any additional constraint), i.e.~for any $w \in W_\af$ and any $1 \le i, j \le n$ we have \[(\Lambda_i - m_i \Lambda_0)(\Lambda_j - m_j \Lambda_0)(\eps_w) = (\Lambda_j - m_j \Lambda_0)(\Lambda_i - m_i \Lambda_0)(\eps_w)  \/. \] 
\end{thm}

\begin{remark}\label{rmk:ex} Without the condition on$\mod q^c$ the commutativity in (a) fails already for $G= \SL_2(\C)$. Recall that in this case $W_\af$ is the infinite dihedral group with generators $s_0 $ and $s_1$ and $c = \alpha_0^\vee + \alpha_1^\vee$. Then $q^c = q_0 q_1$ and $\Lambda_0 \Lambda_1(\eps_{s_0 s_1} )- \Lambda_1 \Lambda_0(\eps_{s_0 s_1}) = q^c$.
\end{remark}

\subsection{A quantum product on $\quantum^*_\af(G/B)$} Since the ring homomorphism $\mathrm{e}_1^*:\coh^*(G/B) \to \coh^*(\Fla)$ is injective by Proposition \ref{prop:inj}, one can use the expression (\ref{eq:Lambdabar}) to define a $\Q[q]$-linear endomorphism on $\quantum^*_\af(G/B)$, denoted $\bar{\Lambda}_i$, by \begin{equation}\label{eq:Lambdafin} \bar{\Lambda}_i(\sigma_w) = \sigma_i \cdot \sigma_w + \sum \langle \lambda_i - m_i \lambda_0, \alpha^\vee \rangle q^{\alpha^\vee} \pi(D_{s_\alpha})(\sigma_w); \quad 1 \le i \le n; \quad w \in W \/; \end{equation} where the sum is over {\em affine} positive real roots $\alpha$ such that $\ell(s_\alpha) = 2 \he(\alpha^\vee) - 1$. Recall that $\pi(D_{s_\alpha})$ is the operator acting on $\coh^*(G/B)$ defined in Theorem \ref{T:ev} above and that $\pi(D_{s_\alpha}) (\coh^{k}(G/B)) \subset \coh^{k - 2 \ell(s_\alpha)}(G/B)$ for $k \ge 0$.~This implies that $\bar{\Lambda}_i$ has degree $+ 1$
on $\quantum^*_\af(G/B)$. Finally, the commutativity of the operators $\Lambda_i - m_i \Lambda_0$ from Theorem \ref{thm:comm} implies that the operators $\bar{\Lambda}_i$ commute as well.

Denote by $\mathcal{Q}$ the (commutative) subring of the endomorphism ring $\mathrm{End}_{\Q[q]}(\quantum^*_\af(G/B))$ generated by the operators $\bar{\Lambda}_1, \ldots, \bar{\Lambda}_n$. Then $\mathcal{Q}$ is also a $\Q[q]$-module and there is a well defined morphism of $\Q[q]$-modules
\[ \varphi: \mathcal{Q} \to \quantum^*_\af(G/B); \quad \varphi(\Phi) = \Phi(1) \/. \]


\begin{thm}\label{thm:defafqcoh} (a) The kernel of the morphism $\varphi$ is an ideal in $\mathcal{Q}$.

(b) The morphism $\varphi$ is surjective.
\end{thm}

Before proving the theorem we recall a graded version of Nakayama Lemma; see e.g.~\cite[Ex.~4.6]{eisenbud:view} or \cite[Lemma 4.1]{mihalcea:giambelli}:
 \begin{lemma}\label{lemma:Nak}
Let $R$ be a commutative ring graded by nonnegative integers and let $I$ be an ideal in $R$ which consists 
of elements of strictly positive degree. Let $M$ be an $R$-module graded by nonnegative integers and let $a_1,a_2 ,\ldots \in M$ be a set of homogeneous elements (possibly infinite) whose images generate $M/IM$ as an $R/I$-module. Then $a_1,a_2,\ldots$ generate
$M$ as an $R$-module.
\end{lemma}
\begin{proof} Let $a$ be a nonzero homogeneous element of $M$. We use induction
on the degree of $a$. If $\deg a = 0$ the hypothesis implies that
\begin{equation}\label{m} a = r_{i_1} a_{i_1} + \cdots  + r_{i_k} a_{i_k} \mod{IM}  \end{equation}
where $r_{i_j}$ are elements in $R$. Since $I$ contains only elements
of positive degree, it follows that the equality holds in $M$ as
well. Let now $\deg a > 0$. Writing $a$ as in (\ref{m}), implies
that $a - \sum_{s=1}^k r_{s} a_{i_s} = \sum_{j=1}^p r_j' a'_{i_j}$ for some (finitely
many) $r_j' \in I$ and $a'_{i_j} \in M$. Again, since $I$ contains only
elements of positive degree, $\deg a'_{i_j} < \deg a$ for each $j$.
The induction hypothesis implies that each $a'_{i_j}$ is an
$R$-linear combination of $a_i$'s, which finishes the proof.
\end{proof}

\begin{proof}[Proof of Theorem \ref{thm:defafqcoh}] Part (a) follows immediately from the fact that $\mathcal{Q}$ is commutative. We now turn to the proof of (b). Since the $\Q$-algebra $\coh^*(G/B)$ is finitely generated and free it follows that there exists a finite set $\mathcal{I}$ of elements of the form $\sigma({\bf i}) := \sigma_{i_1} \cdots  \sigma_{i_s} \in \coh^*(G/B)$  such that $\coh^*(G/B)$ is generated as a $\Q$-module by $\sigma({\bf i})$ for ${\bf i} \in \mathcal{I}$. For each ${\bf i}=(i_1, \ldots, i_s) \in \mathcal{I}$ define the elements $\sigma^{(q)}({\bf i}) = \bar{\Lambda}_{i_1} \bar{\Lambda}_{i_2} \cdots  \bar{\Lambda}_{i_s}(1) \in \quantum^*_\af(G/B)$. By the graded Nakayama Lemma above the elements $\sigma^{(q)}({\bf i})$ (${\bf i} \in \mathcal{I}$) generate the $\Q[q]$-module $\quantum^*_\af(G/B)$. But these elements are also in the image $\varphi(\mathcal{Q})$, thus $\varphi$ must be surjective. \end{proof}

Theorem \ref{thm:defafqcoh} implies that we can define a product structure $\star_\af$ on $\quantum^*_\af(G/B)$ by:\begin{equation}\label{def:staraf} \sigma_u \star_\af \sigma_v := \bar{\Lambda}_u \bar{\Lambda}_v (1)= \bar{\Lambda}_u(\sigma_v) \end{equation} where $\bar{\Lambda}_u, \bar{\Lambda}_v$ are any elements in the preimages of $\sigma_u$ and $\sigma_v$ respectively through $\varphi$. Then one extends this product by $\Q[q]$-linearity. For example, $\varphi(\bar{\Lambda}_i) = \sigma_i$ 
and \begin{equation}\label{eq:divmult} \sigma_i \star_\af \sigma_j = \sigma_i \cdot \sigma_j + \delta_{ij} q_i + m_i m_j q_0 \textrm{ where } \theta^\vee = \sum_{i=1}^n m_i \alpha_i^\vee \/. \end{equation} In this language, the expression (\ref{eq:Lambdafin}) gives an {\em affine quantum Chevalley formula} for $\sigma_i \star_\af \sigma_w$ in $\quantum^*_\af(G/B)$. For further examples we refer to \S \ref{ss:table} below, where we work out the multiplication in $\quantum^*_\af(\SL_3(\C)/B)$. 

\subsection{Properties of the ring $(\quantum^*_\af(G/B), \star_\af)$.}\label{ss:properties} By construction it follows that $(\quantum^*_\af(G/B), \star_\af)$ is a graded, commutative, $\Q[q]$-algebra with a $\Q[q]$-basis given by classes $\sigma_w$, for $w \in W$. It is generated by classes $\sigma_i$, where $1 \le i \le n$. We will prove in \S \ref{s:Toda} below that the relations among the generators can be described using the integrals of motion for the dual periodic Toda lattice. While proving these facts we will also develop the Frobenius/Dubrovin formalism for $\quantum^*_\af(G/B)$: we will prove in \S \ref{frodu} below that it has a Frobenius structure, and that the quantum multiplication $\star_\af$ determines a flat Dubrovin connection on the trivial bundle $\coh^*(G/B) \times \coh^2(G/B) \to \coh^2(G/B)$. 

Finally, the ring $\quantum^*_\af(G/B)$ is closely related to the ordinary quantum cohomology algebra $\quantum^*(G/B)$ for $G/B$. Recall that the latter is the graded $\Q[q_1, \ldots, q_n]$-module $\coh^*(G/B) \otimes \Q[q_1, \ldots, q_n]$ with the basis given by the usual Schubert classes $\sigma_w$ ($w \in W$) and the usual grading. The product, denoted by $\star$, is given by \[ \sigma_u \star \sigma_v = \sum_{d \in \coh_2(G/B), w \in W} \gw{\sigma_u, \sigma_v, [X(w)]}{d} q^d \sigma_w \] where $d = d_1[X(s_1)] + \cdots  + d_n [X(s_n)]$, $q^d = q_1^{d_1} \cdots  q_n^{d_n}$, and $\gw{\sigma_u, \sigma_v, [X(w)]}{d}$ are the (ordinary) Gromov-Witten invariants which count rational curves of degree $d$ in $G/B$
intersecting general translates of Schubert varieties representing the cycles $\sigma_u, \sigma_v$ and $[X(w)]$. In particular, the structure constants of $\quantum^*(G/B)$ are non-negative integers. Then the ring $\quantum^*_\af(G/B)$ is a deformation of the ordinary quantum cohomology ring in the sense that there is an isomorphism of graded $\Q[q]/\langle q_0 \rangle$-algebras \[ \quantum^*_\af (G/B)/\langle q_0 \rangle \simeq \quantum^*(G/B) \/ \] preserving Schubert classes. This follows because the classes $\sigma_i$ ($1 \le i \le n$) generate both algebras (over the appropriate coefficient rings), and because the specialization $q_0=0$ of the affine quantum Chevalley formula (\ref{eq:Lambdafin}) coincides with the ordinary quantum Chevalley formula conjectured by Peterson and proved by Fulton and Woodward \cite{fulton.woodward}. However, unlike for $\quantum^*(G/B)$, the structure constants for $\quantum^*_\af(G/B)$ are {\em not} positive in general. For instance, one can calculate that for $G= \SL_3(\C)$ the coefficient of $q_0$ in $\sigma_1 \star_\af \epsri_{s_1s_2}$ is
$$\langle \lambda_1 - \lambda_0, \alpha_0^\vee \rangle \pi(D_0)(\sigma_{s_1 s_2})= -\partial_{-\theta}(\epsri_{s_1s_2})= \partial_{\theta}(\epsri_{s_1s_2}) = \epsri_{2}-\epsri_{1} \/.$$ The last equality follows because $\sigma_{s_1 s_2} = \sigma_2 \cdot \sigma_2$ by the ordinary Chevalley formula and from the Leibniz formula (Proposition \ref{prop:leibniz}).

\subsection{A quantum product on $\quantum^\#(\Fla)/ \langle q^c \rangle$} We will abuse notation and denote again by $\star_\af$ the quantum product on $\quantum^\#(\Fla)/\langle q^c \rangle $. Its definition is similar to that from the previous section, so we will be brief. Define $\mathcal{Q}^\#$ to be the subring of $\mathrm{End}_{\Q[q]/\langle q^c \rangle} ( \quantum^\#(\Fla)/\langle q^c \rangle)$ generated by $\Lambda_0, \ldots, \Lambda_n$. By Theorem \ref{thm:comm} (b) this is a commutative ring and a $\Q[q]/\langle q^c \rangle $-module. One can define a $\Q[q]/\langle q^c \rangle$-module homomorphism $\varphi^\#: \mathcal{Q}^\# \to \quantum^\#(\Fla)/\langle q^c \rangle$ by $\varphi^\#(\Phi) = \Phi(1)$. Since $\mathcal{Q}^\#$ is commutative, the kernel of $\varphi^\#$ is an ideal in $ \mathcal{Q}^\#$. By hypothesis the $\Q$-module $\mathrm{H}^\#(\Fla)$ has a countable set of generators given by the images modulo $q_0, \ldots, q_n$ of the monomials $\Lambda_{i_1} \ldots \Lambda_{i_k}(1)$. By the Nakayama-type result  from Lemma \ref{lemma:Nak} this implies that the set of these monomials generate $\quantum^\#(\Fla)/\langle q^c \rangle$ as a $\Q[q]/\langle q^c \rangle$-module, which implies that $\varphi^\#$ is surjective. Then for any $a, b \in \coh^\#(\Fla)$ one can define a product:\[ a \star_\af b := {\Lambda}_a {\Lambda}_b (1)= {\Lambda}_a(b) \] where ${\Lambda}_a, {\Lambda}_b$ are elements in the preimages of $a$ and $b$ respectively through $\varphi^\#$. Then one extends this product by $\Q[q]/\langle q^c \rangle$-linearity. { This endows $\quantum^\#(\Fla)/\langle q^c \rangle$ with an associative, commutative quantum product. When $q_0 = 0$ one recovers the ordinary quantum cohomology ring $\quantum^\star(G/B)$, thus $\quantum^\#(\Fla)/\langle q^c \rangle$ is another deformation of $\quantum^\star(G/B)$.}

\section{The commutativity of the Chevalley operators: proof of Theorem \ref{thm:comm} }\label{combin}


The goal of this section is to prove part (a) of theorem \ref{thm:comm}:
for any $w \in W_\af$ and any $0 \le i, j \le n$ we have \begin{equation}\label{eq:commLambda} \Lambda_i \Lambda_j (\eps_w) = \Lambda_j \Lambda_i(\eps_w)  \mod q^c = q_0 q^{\theta^\vee} \/. \end{equation}

\subsection{Combinatorial preliminaries on roots appearing in Chevalley operators} Denote by $\tilde{\Pi}_\af^{{re},+}$ the set of all positive real affine roots $\alpha$ with the 
property that $\ell(s_\alpha)= 2\he ({\alpha^\vee})-1$. According to Proposition \ref{letalpha} any such root satisfies $\alpha^\vee < c = \alpha_0^\vee + \theta^\vee$, and by Theorem \ref{thm:afqops} these are precisely the roots relevant for the Chevalley operators. The following lemma shows that roots $\alpha \in \tilde{\Pi}_\af^{{re},+}$ can be constructed inductively:

\begin{lemma}\label{lemma:induction} Let $\alpha \in \tilde{\Pi}_\af^{{re},+} $ be a non-simple root. Then there exists a simple root $\alpha_i \in \Delta_\af$ such that $s_i(\alpha) \in \tilde{\Pi}_\af^{{re},+}$, $\ell(s_i s_\alpha s_i) = \ell(s_\alpha) - 2$, and $s_i(\alpha)^\vee = \alpha^\vee - \alpha_i^\vee$. \end{lemma}

\begin{proof} Take the simple reflection $s_i$ guaranteed by Lemma \ref{lemma:instep} above. Then \[ 2 \he(\alpha^\vee) -3 = \ell(s_\alpha) - 2 =  \ell(s_i s_\alpha s_i) \le 2 \he (s_i(\alpha)^\vee) -1 = 2 \he(\alpha^\vee) - 2 \langle \alpha_i, \alpha^\vee \rangle -1 \/, \] where the inequality follows from Proposition \ref{letalpha}. But $\langle \alpha_i, \alpha^\vee \rangle > 0$ again by Lemma \ref{lemma:instep} and this implies that $\langle \alpha_i, \alpha^\vee \rangle = 1$ and we have equality throughout.
\end{proof}

Proposition \ref{posir} and Lemmas \ref{lemma:positive} and \ref{lem:be} below are affine versions of \cite[Proposition~3.1, Lemma 3.2 and Lemma 3.3]{Ma0}.

\begin{prop}\label{posir} A positive real  root $\alpha$ is in $\tilde{\Pi}_\af^{{re},+}$ if and only if it is simple, or else there exist indices $i_1, \ldots, i_k \in  \{0,\ldots , n\}$ ($k \ge 2$), where all indices different from $0$ are possibly repeated, such that:\begin{itemize}
\item each root ${ \alpha_{i_j}:=}s_{i_j}\cdots   s_{i_2}(\alpha_{i_1})$ is in $\tilde{\Pi}_\af^{{re},+}$ and it satisfies $\langle \alpha_{i_{j}}, s_{i_{j-1}}\cdots  s_{i_2}(\alpha_{i_1})^\vee\rangle  = -1$ for all $2 \le j\le k$;
\item $\alpha = s_{i_k} \cdots  s_{i_2}(\alpha_{i_1})$ and the expression
$s_{\alpha} = s_{i_k} \cdots  s_{i_2}s_{i_1}s_{i_2}\cdots   s_{i_k}$
is reduced. \end{itemize} In particular, $\alpha^\vee= \alpha_{i_1}^\vee + \cdots   + \alpha_{i_k}^\vee$. \end{prop}

\begin{proof} Applying repeatedly Lemma \ref{lemma:induction} starting with $\alpha \in \tilde{\Pi}_\af^{{re},+}$ yields the set of indices $i_1,\ldots, i_k$ and the roots with the claimed properties. To prove the converse, we use induction on $j \ge 2$. Let $\beta := s_{i_{j-1}}\cdots   s_{i_2}(\alpha_{i_1})$ as in the hypothesis and denote $i:=i_j$ We will show that $s_i s_\beta s_i$ is reduced and that $s_i(\beta)$ satisfies the required properties. First, $s_i(\beta^\vee) = \beta^\vee + \alpha_i^\vee$ because $\langle \alpha_i, \beta^\vee \rangle = - 1$. In particular, $\ell(s_i s_\beta) = \ell(s_\beta) +1$. By induction hypothesis, we know that $\beta \in \tilde{\Pi}_\af^{{re},+}$. We have that $s_i s_\beta(\alpha_i) = (\langle \beta, \alpha_i^\vee \rangle \langle \alpha_i, \beta^\vee \rangle - 1 ) \alpha_i - \langle \alpha_i, \beta^\vee \rangle \beta$ is a positive root, which shows that $\ell (s_i s_\beta s_i) = \ell(s_\beta) +2$. This finishes the proof.\end{proof}

\begin{cor}\label{cor:ADE} If $G$ is simply laced then the set $ \tilde{\Pi}_\af^{{re},+}$ coincides with the set of positive real roots $\alpha \in \Pi^{re, +}_\af$ such that $\alpha^\vee < c$. \end{cor}

\begin{proof} We use induction on $\he (\alpha) \ge 1$. The case when $\he(\alpha) =1$ is clear, and let $\he(\alpha) > 1$. We claim that there exists a simple root $\alpha_i$ such that $\langle \alpha_i,\alpha^\vee \rangle = 1$. If $\alpha$ is a root in the finite root system, we can take any $\alpha_i$ such that $\alpha - \alpha_i$ is a root. Since $G$ is simply laced, the claim follows in this case from the table on \cite[p.~45]{Hu1}. If $\alpha$ is not in the finite root system, and because $\alpha^\vee < c$, we have $\alpha^\vee = c - \tilde{\alpha}^\vee$ where $\tilde{\alpha} \in \Pi^+$. Notice that $\alpha \neq \alpha_0$ and $\alpha^\vee <c$ implies that $\tilde{\alpha}^\vee < \theta^\vee$.
Thus there exists a finite simple root $\alpha_i$ such that  $\langle \alpha_i,\tilde{\alpha}^\vee \rangle = - 1$ (otherwise $\tilde{\alpha}$ would be a dominant root, hence equal to the highest root,
see e.g.~\cite[p.~371]{Go-Wa1}). This implies that $\langle \alpha_i,\alpha^\vee \rangle = - \langle \alpha_i, \tilde{\alpha}^\vee \rangle = 1$ and the claim is proved. Let $\beta:= s_i (\alpha)$. This is a positive root because $\alpha$ is not simple. Further, $\he (\beta) < \he (\alpha)$ and by induction hypothesis we have that $\beta \in  \tilde{\Pi}_\af^{{re},+}$. We calculate that $\langle \alpha_i, \beta^\vee \rangle = \langle \alpha_i, \alpha^\vee - \alpha_i^\vee \rangle = -1$, so by Proposition \ref{posir} the root $\alpha = s_i(\beta) $ is in $ \tilde{\Pi}_\af^{{re},+}$, which concludes the proof.\end{proof}

\begin{remark} The Corollary above is false if $G$ is not simply laced. In fact, the sets $\tilde{\Pi}_\af^{{re},+} \cap \Pi^+ \neq \Pi^+$ for each $G$ not simply laced. A list with all roots in $\tilde{\Pi}_\af^{{re},+} \cap \Pi^+$ can be found in \cite{buch.mihalcea:nbhds}.\end{remark}

\begin{lemma}\label{lemma:positive} (a) Let $\alpha, \beta$ be two positive real roots such that $\ell(s_\alpha s_\beta) = \ell(s_\alpha) + \ell(s_\beta)$. Then $\langle \alpha,\beta^\vee\rangle\le 0$.

(b) Assume in addition that $s_\alpha s_\beta \neq s_\beta s_\alpha$ and that $\alpha^\vee +\beta^\vee$ is not a multiple of the imaginary coroot $c= \alpha_0^\vee + \theta^\vee$. Then at least one of $\langle \alpha, \beta^\vee\rangle $ and $\langle \beta, \alpha^\vee\rangle$  is  equal to $-1$.

\end{lemma}

\begin{proof} Since $s_\alpha s_\beta$ and $s_\beta s_\alpha$ are inverses to each other, $\ell(s_\alpha s_\beta) = \ell(s_\beta s_\alpha)$. The hypothesis implies that the root $s_\alpha(\beta) = \beta - \langle \beta,\alpha^\vee\rangle \alpha  >0$ and $s_\beta(\alpha) = \alpha - \langle \alpha, \beta^\vee \rangle \beta > 0$; cf. \cite[p.~116]{Hu2}. If $\langle \alpha,\beta^\vee\rangle > 0$ then $\langle \beta,\alpha^\vee\rangle > 0$ also, because $\alpha, \beta$ are not multiples of the imaginary root $\delta$, and by Remark \ref{rth} and equation (\ref{alphv}). This implies that $\alpha \ge \beta$ and $\beta \ge \alpha$, therefore $\alpha = \beta$, which implies that $\ell (s_\alpha s_\beta) = 0$, a contradiction. This proves part (a). 

For part (b), we first notice that the hyothesis implies that $\langle \alpha, \beta^\vee\rangle <0$  and $\langle \beta, \alpha^\vee\rangle <0 $. Assume that the claim in the  lemma is not true, i.e.
$\langle \alpha, \beta^\vee\rangle =\frac{2(\alpha | \beta)}{(\beta | \beta)}\le -2 \ {\rm and} \ 
\langle \beta, \alpha^\vee\rangle =\frac{2(\alpha | \beta)}{(\alpha | \alpha)}\le -2$.
Thus $-(\alpha | \beta)\ge (\alpha | \alpha)$ and $-(\alpha | \beta)\ge (\beta | \beta)$, which implies that
$(\alpha | \beta)^2 \ge (\alpha | \alpha)\cdot (\beta | \beta)$. On the other side, by Cauchy-Schwartz 
inequality $(\alpha | \beta)^2 \le (\alpha | \alpha)\cdot(\beta | \beta)$, and thus we have equality. This also forces equalities $- (\alpha |  \beta) = (\alpha | \alpha) = (\beta |  \beta)$, therefore
$(\alpha + \beta |  \alpha + \beta) = 0$. By Remark \ref{rth} this implies that $\alpha + \beta = m \delta$ for some $m \in \mathbb{R}$. In fact, since $\alpha, \beta$ and $\delta$ are in the root lattice, and $\delta$ is part of an integral basis for this lattice, it follows that $m \in \Z$. Then $\nu^{-1}(\alpha + {\beta}) = mc$ hence
$$\alpha^\vee +{\beta}^\vee =  \frac{2m}{(\alpha | \alpha)}c \/.$$
But by hypothesis $\alpha^\vee +\beta^\vee$ is not a multiple of $c$, and this finishes the proof.
\end{proof}

\begin{lemma}\label{lem:be} Let $\alpha, \beta \in \tilde{\Pi}_\af^{{re},+} $ such that $\ell(s_\alpha s_\beta) = \ell(s_\alpha) + \ell(s_\beta)$, $s_\alpha s_\beta \neq s_\beta s_\alpha$ and in addition $\alpha^\vee + \beta^\vee < c$. Then $\alpha^\vee + \beta^\vee =\gamma^\vee$ where $\gamma$ is a root in $\tilde{\Pi}_\af^{{re},+} $.
\end{lemma}

\begin{proof} By Lemma \ref{lemma:positive} we can assume that
 $\langle \beta, \alpha^{\vee}\rangle =-1$. We need to show that $\gamma:=s_{\beta}(\alpha)\in
\tilde{\Pi}_\af^{{re},+} $ (since
$s_{\beta}(\alpha^{\vee}) = \alpha^{\vee}
+\beta^{\vee}$, this proves the claim). We will use induction on $\ell(s_{\beta})$. If
$\beta$ is simple, the result follows from Proposition
\ref{posir}. Consider now the case when 
$\beta\in \tilde{\Pi}_\af^{{re},+} $ is non-simple. Take the index $i$ guaranteed by Lemma \ref{lemma:induction}. Then $s_\beta = s_i s_{{\beta}_1} s_i$ where $\ell(s_{{\beta}_1}) = \ell(s_\beta) - 2$, and $\langle \alpha_i, \beta^\vee \rangle = 1$. Combined with the fact that $\ell(s_\alpha s_\beta) = \ell(s_\alpha) + \ell(s_\beta)$ this implies that $\ell(s_\alpha s_i) > \ell(s_\alpha)$, thus $s_\alpha(\alpha_i) = \alpha_i - \langle \alpha_i, \alpha^\vee \rangle \alpha$ is a positive root. Then $\langle \alpha_i, \alpha^{\vee}\rangle\leq 0$ and $\alpha_i \neq \alpha$. 



We claim that the quantity $\langle \alpha_i,\alpha^{\vee}\rangle \in \{-1,0\}$. Assume $\langle \alpha_i,\alpha^{\vee}\rangle\le -2$ and let $\alpha = m \delta + \tilde{\alpha}$ where $\tilde{\alpha}$ is a finite root, possibly negative. By the description of positive real roots and coroots from \S \ref{akm}, it follows that $\alpha_i = a \delta + \tilde{\alpha}_i$ 
and { $\alpha^\vee = m c + \tilde{\alpha}^\vee$} where $a, m\in \{ 0 ,1 \}$ (because $0 < \alpha^\vee < c$ and $\alpha_i < \beta < \delta$) and $\tilde{\alpha}_i$ is in the finite root system, possibly negative. { In fact, $\tilde{\alpha_i}$ is either simple, or it equals $- \theta$, the negative of the highest root.} Notice that $
\langle \alpha_i,\alpha^{\vee}\rangle = \langle \tilde{\alpha}_i,\tilde{\alpha}^{\vee}\rangle$. We claim that $\tilde{\alpha}_i \neq - \tilde{\alpha}$. This follows from analysis of the four possibilities for $a,m \in \{ 0 ,1 \}$. The only situations when the equality occurs are when $\alpha^\vee + \alpha_i^\vee = c$, but this is impossible since $\alpha^\vee + \alpha_i^\vee \le \alpha^\vee + \beta^\vee < c$. We deduce that $(\alpha_i |  \alpha_i) = (\tilde{\alpha}_i |  \tilde{\alpha}_i) > (\tilde{\alpha} |  \tilde{\alpha}) = (\alpha | \alpha)$, where the inequality follows from Table on \cite[p.~45]{Hu1}. On the other side, the hypothesis that $\langle \alpha_i, \beta^{\vee}\rangle =1$ implies that $(\beta | \beta) \ge (\alpha_i | \alpha_i)$, and the condition that $\langle \beta, \alpha^{\vee}\rangle =-1$ implies that $(\alpha | \alpha) \ge (\beta | \beta)$. Therefore $(\alpha | \alpha) \ge (\alpha_i | \alpha_i)$, which is a contradiction, and finishes the proof of the claim.

In what follows we will prove that the roots $\beta_1 := s_i(\beta)$ and $s_i(\alpha)$ satisfy the induction hypothesis. First, notice that by construction $\ell(s_{\beta_1}) < \ell(s_\beta)$, and that ${\beta_1} \in \tilde{\Pi}_\af^{{re},+}$. It is also clear that $s_{{\beta_1}} s_{s_i(\alpha)} \neq s_{s_i(\alpha)} s_{{\beta_1}}$, so it remains to check that $s_i(\alpha) \in \tilde{\Pi}_\af^{{re},+}$, $\langle {\beta_1}, s_i(\alpha)^\vee \rangle = -1$ and that
$\ell(s_{{\beta_1}} s_{s_i(\alpha)}) = \ell(s_{{\beta_1}}) + \ell( s_{s_i(\alpha)})$. When this is done, it implies that $\gamma':= s_{{\beta_1}}(s_i(\alpha)) \in \tilde{\Pi}_\af^{{re},+}$, and to finish the proof we will show that $\gamma=s_i(\gamma')$ is also in $\tilde{\Pi}_\af^{{re},+}$. We distinguish the following two situations:

\noindent {\it Case 1.} $\langle \alpha_i, \alpha^{\vee}\rangle=0$. This implies
$s_i(\alpha)=\alpha$, and because $s_{{\beta_1}} s_\alpha$ has a decomposition given by a reduced subword of $s_\beta s_\alpha$,  
we obtain that $\ell(s_{{\beta_1}} s_\alpha) = \ell(s_{{\beta_1}})+ \ell( s_\alpha)$. Further,
 $$-1=\langle \beta, \alpha^{\vee}\rangle =\langle s_i({\beta_1}), \alpha^{\vee}\rangle =
\langle {\beta_1}, s_i(\alpha)^{\vee}\rangle =\langle {\beta_1}, \alpha^{\vee}\rangle .$$
From the induction hypothesis,
$\gamma':=s_{{\beta_1}}(\alpha)$ is in
$\tilde{\Pi}_\af^{{re},+}$. Furthermore
$$\langle \alpha_i, (\gamma')^\vee \rangle = \langle \alpha_i, s_is_{\beta}s_i(\alpha^{\vee})\rangle=
-\langle \alpha_i, s_{\beta}(\alpha^{\vee})\rangle =-\langle \alpha_i, \alpha^{\vee}
+\beta^{\vee}\rangle =-1.$$ Then by Proposition
\ref{posir}, the root
$s_i(\gamma')$ is in $\tilde{\Pi}_\af^{{re},+}$, and the proof in this case is done.

\noindent {\it Case 2.} $\langle \alpha_i, \alpha^{\vee}\rangle=-1$. In this case the root $s_i(\alpha)$ is in $\tilde{\Pi}_\af^{{re},+}$, by Proposition \ref{posir}. Further, 
$$-1=\langle \beta, \alpha^{\vee}\rangle =\langle s_i({\beta_1}), \alpha^{\vee}\rangle =
\langle {\beta_1}, s_i(\alpha)^{\vee}\rangle.$$ 
A simple calculation shows that $s_i s_\alpha (\alpha_i) = - (\langle \alpha, \alpha_i^\vee \rangle + 1 ) \alpha_i + \alpha$ and that 
$s_{{\beta_1}}s_is_{\alpha}(\alpha_i) =
\alpha-(1+\langle \alpha, \alpha_i^{\vee}\rangle )\alpha_i-(1+\langle \alpha, \beta^{\vee})\rangle {\beta_1}$, and both of these are positive roots. Then
$$\ell(s_{{\beta_1}}s_{s_i(\alpha)})=\ell(s_{{\beta_1}}s_is_{\alpha}s_i) =
\ell(s_{{\beta_1}}s_is_{\alpha})+1=\ell(s_{{\beta_1}})+l(s_is_{\alpha}s_i)=
\ell(s_{{\beta_1}})+\ell(s_{s_i(\alpha)})\/.$$ From the induction
hypothesis we deduce that
$\gamma':= s_{{\beta_1}}s_i(\alpha)$ belongs to $\tilde{\Pi}_\af^{{re},+}$. But $\langle \alpha_i, (\gamma')^{\vee}\rangle =-\langle \alpha_i, \alpha^{\vee}+\beta^{\vee}\rangle = 0$, therefore 
$\gamma= s_i(\gamma')= \gamma'$ is again in $\tilde{\Pi}_\af^{{re},+}$.
\end{proof}

\begin{lemma}\label{lemma:multiple} Let $\alpha, \beta \in \tilde{\Pi}_\af^{re,+}$ such that $s_\alpha s_\beta\neq s_\beta s_\alpha$, $\ell(s_\alpha s_\beta) = \ell(s_\alpha) + \ell(s_\beta)$ and $\alpha^\vee + \beta^\vee \nless c = \alpha_0^\vee + \theta^\vee$. Then $\alpha^\vee + \beta^\vee = c$. \end{lemma}

\begin{proof} Assume first that $\alpha^\vee + \beta^\vee$ is not a multiple of $c$. The hypothesis implies that $\langle \alpha, \beta^\vee \rangle \neq 0$ therefore one of $\langle \alpha, \beta^\vee\rangle$ and $\langle \beta,\alpha^\vee\rangle$ must equal $-1$ by Lemma \ref{lemma:positive}(b). Since the statement is symmetric in $\alpha$ and $\beta$ (because $\ell(s_\alpha s_\beta) = \ell(s_\beta s_\alpha)$), we only need to consider the case when $\langle \alpha, \beta^\vee\rangle=-1$. Then $s_\alpha(\beta^\vee) = \beta^\vee -\langle \alpha, \beta^\vee\rangle \alpha^\vee =\alpha^\vee + \beta^\vee$ is a positive coroot, therefore $s_\alpha(\beta)$ is a positive root. 
On the other hand,
$$s_\alpha s_\beta (\alpha) = s_\alpha (\alpha - \langle \alpha, \beta^\vee\rangle \beta) 
=s_\alpha (\alpha +\beta) =
-\alpha +\beta -\langle \beta, \alpha^\vee\rangle \alpha
=\beta -(\langle \beta,\alpha^\vee\rangle +1)\alpha,$$
which is  a positive root because $\langle \beta,\alpha^\vee\rangle \le -1$.
This implies
$$\ell(s_{s_\alpha(\beta)}) =\ell(s_\alpha s_\beta s_\alpha) > \ell(s_\alpha s_\beta)=
2{\rm ht} (\alpha^\vee +\beta^\vee) -2= 2{\rm ht}(s_\alpha(\beta)^\vee) -2.$$
Proposition \ref{letalpha} implies that
$\ell(s_{s_\alpha(\beta)}) = 2{\rm ht}(s_\alpha(\beta)^\vee) -1$
and $s_\alpha(\beta)^\vee < c$. Since $s_\alpha(\beta)^\vee = \alpha^\vee + \beta^\vee$, this is a contradiction, therefore $\alpha^\vee + \beta^\vee$ must be a multiple of $c$. Invoking again Proposition \ref{letalpha} we obtain that $\alpha^\vee, \beta^\vee <c$, thus the only possibility is $\alpha^\vee + \beta^\vee = c$.  
\end{proof}

\begin{lemma}\label{lemma:extra1} Let $\gamma$ be a non-simple root in $\tilde{\Pi}_\af^{{re},+}$ and $\eta$ a real positive root in $\Pi^{re,+}_\af$ such that $\ell(s_\gamma s_\eta) = \ell(s_\gamma) -  1$. Then $\langle \eta,\gamma^\vee \rangle = 1 $.\end{lemma}

\begin{proof} 
Consider the reduced decomposition $s_\gamma = s_{i_k} \cdots  s_{i_1} \cdots  s_{i_k}$ guaranteed by Proposition \ref{posir}. Since $\ell(s_\gamma s_\eta) < \ell (s_\gamma)$, the Strong Exchange Condition \cite[p.~117]{Hu2} implies that there is a unique index $i_j$ which is removed from the expression of $s_\gamma$ such that $s_\gamma  s_\eta = s_{i_k} \cdots  \widehat{s_{i_j}} \cdots  s_{i_k}$. Assume that the removed index is in the second half of the decomposition of $s_\gamma$, i.e. $s_\gamma s_\eta = s_{i_k} \cdots  s_{i_1} \cdots \widehat{s_{i_j}} \cdots  s_{i_k}$. Then $s_\eta = s_{i_k} \cdots  s_{i_1} \cdots  s_{i_{j-1}} s_{i_j} s_{i_{j-1}}\cdots  s_{i_1} \cdots  s_{i_k}$, therefore $\eta = s_{i_k} \cdots  s_{i_1} \cdots  s_{i_{j-1}}(\alpha_{i_j})$, which is a positive root because $s_{i_k} \cdots  s_{i_1} \cdots  s_{i_j}$ is reduced. Finally, \[ \begin{split} \langle \eta, \gamma^\vee \rangle = \langle s_{i_k} \cdots  s_{i_1} \cdots  s_{i_{j-1}}(\alpha_{i_j}), s_{i_k} \cdots  s_{i_2} (\alpha_{i_1}^\vee) \rangle = \langle s_{i_1} \cdots  s_{i_{j-1}}(\alpha_{i_j}), \alpha_{i_1}^\vee \rangle \\  = \langle \alpha_{i_j}, s_{i_{j-1}} \cdots  s_{i_2} s_{i_1}(\alpha_{i_1}^\vee) \rangle = - \langle \alpha_{i_j}, s_{i_{j-1}} \cdots  s_{i_2} (\alpha_{i_1}^\vee) \rangle = 1 \/. \end{split}\] A similar calculation works when $s_\gamma s_\eta = s_{i_k} \cdots \widehat{s_{i_j}} \cdots  s_{i_1} \cdots  s_{i_k}$.
\end{proof}

\begin{prop}\label{prop:pairbij} There is a bijection $\Psi$ between the sets \[ A:= \{ (\alpha, \beta) \in \tilde{\Pi}_\af^{{re},+} \times \tilde{\Pi}_\af^{{re},+}: \langle \alpha, \beta^\vee \rangle \neq 0, ~\ell(s_\alpha s_\beta) = \ell(s_\alpha) + \ell(s_\beta), \alpha^\vee + \beta^\vee < c \} \] and \[ B:=\{ (\gamma, \eta) \in \tilde{\Pi}_\af^{{re},+} \times \Pi_{\af}^{re,+}: \ell(s_\gamma s_\eta) = \ell(s_\gamma) - 1 >0 \} \] sending $(\alpha, \beta)$  to $(\gamma, \eta)$ such that $\gamma^\vee = \alpha^\vee + \beta^\vee$ and $s_\gamma s_\eta = s_\alpha s_\beta$. \end{prop}

\begin{proof} We first prove that $\Psi$ is well defined. From Lemmas \ref{lemma:positive} and \ref{lem:be} we obtain that the pair $(\gamma, \eta)$ is the following:
\begin{equation}\label{eq:gammae} (\gamma, \eta) = \begin{cases} (s_\alpha(\beta), \alpha) = \bigl(\beta - \langle \beta, \alpha^\vee \rangle \alpha , \alpha\bigr) & \textrm{if } \langle \alpha, \beta^\vee \rangle = -  1; \\ (s_\beta(\alpha),s_\beta s_\alpha(\beta))=  \bigl(\alpha - \langle \alpha, \beta^\vee \rangle \beta, \alpha - (\langle \alpha, \beta^\vee \rangle +1) \beta\bigr) & \textrm{if } \langle \alpha, \beta^\vee \rangle <  -  1\/. \end{cases} \end{equation}
(In fact one can easily check the the formula for $\eta$ in the branch for $\langle \alpha, \beta^\vee \rangle < -  1$ works also for $\langle \alpha, \beta^\vee \rangle =  - 1$.) Further, we have that $\ell(s_\gamma s_\eta) = \ell(s_\alpha s_\beta) = 2 \he (\alpha^\vee + \beta^\vee) -  2 = 2 \he (\gamma^\vee ) -  2 = \ell (s_\gamma)  - 1$. This implies that $(\gamma, \eta) \in B$. Let now $(\alpha_1,\beta_1) \in A $ such that $\Psi((\alpha, \beta)) = \Psi((\alpha_1, \beta_1))$. One calculates that \[ \langle \gamma, \eta^\vee \rangle = \begin{cases}  - \langle \beta, \alpha^\vee\rangle & \textrm{if } \langle \alpha, \beta^\vee\rangle = - 1; \\
- \langle \alpha, \beta^\vee\rangle & \textrm{if } \langle \alpha, \beta^\vee\rangle < - 1 \/. \end{cases}\]
Recall also that $\langle \eta, \gamma^\vee \rangle =1 $ by Lemma \ref{lemma:extra1}. These formulas imply immediately that $(\alpha, \beta) = (\alpha_1, \beta_1)$ if $\langle \alpha, \beta^\vee\rangle = \langle \alpha_1, \beta_1^\vee\rangle$.
If $\langle \alpha, \beta^\vee\rangle \neq \langle \alpha_1, \beta_1^\vee\rangle$ we can assume that $\langle \alpha, \beta^\vee \rangle = \langle \beta_1, \alpha_1^\vee \rangle = - 1$. This implies that $\langle \eta, \gamma^\vee \rangle = \langle \gamma, \eta^\vee \rangle =1$, which forces $\langle \alpha, \beta^\vee\rangle = \langle \alpha_1, \beta_1^\vee\rangle$, a contradiction. We conclude that $\Psi$ is injective. {To prove surjectivity,  take $(\gamma,\eta)\in B$.
 Consider the
reduced decomposition $s_{\gamma}=s_{i_k} \cdots 
s_{i_2}s_{i_1}s_{i_2} \cdots  s_{i_k}$ given by Proposition
\ref{posir} and recall that $\langle \eta, \gamma^\vee \rangle = 1$ by  Lemma \ref{lemma:extra1}. 
By the Strong Exchange Condition we distinguish the following two
cases:

{\it Case 1.} We have
\begin{equation} \label{sik}
s_{\gamma}s_{\eta}=s_{i_k}\cdots  s_{i_2}s_{i_1}s_{i_2} \cdots  \hat{s}_{i_j} \cdots  s_{i_k}
\end{equation}
for some $j$ between
$2$ and $k$. This implies
 $s_{\gamma}= s_{i_k} \cdots 
s_{i_{j+1}}s_{i_j}s_{i_{j+1}} \cdots  s_{i_k}$, thus $\gamma=s_{i_k} \cdots 
s_{i_{j+1}}(\alpha_{i_j})$, the latter being a positive root since
the expression in (\ref{sik}) is reduced. Set $\alpha=\eta$ and
$\beta=s_{\eta}(\gamma)$. Notice that $s_{\eta}(\gamma)=s_{i_k} \cdots  \hat{s}_{i_j} \cdots  s_{i_2}(\alpha_{i_1})$
is a positive root, because the right-hand side in (\ref{sik}) 
is reduced. Then $\beta^{\vee}=
s_{\eta}(\gamma^{\vee})=\gamma^{\vee}-\eta^{\vee}$, which
implies that $\gamma^{\vee}=\alpha^{\vee}+\beta^{\vee}$. We obviously
have $s_{\gamma}s_{\eta}=s_{\alpha}s_{\beta}$, hence
$$\ell(s_{\alpha}s_{\beta})=2{\rm ht}(\gamma^{\vee})-2= 2{\rm ht}(\alpha^{\vee}) -1 +
2{\rm ht}(\beta^{\vee})-1.$$ From Proposition \ref{letalpha} we deduce that
$\alpha$  and $\beta$ are both in $\tilde{\Pi}_\af^{{re},+}$ and
$\ell(s_{\alpha}s_{\beta}) =\ell(s_{\alpha})+\ell(s_{\beta})$. If we had
$s_{\alpha}s_{\beta} = s_{\beta}s_{\alpha}$, then
$s_{ \gamma}s_{\eta}=s_{\eta}s_{\gamma}$, which is impossible, since
$\langle\eta, \gamma^{\vee}\rangle >0$.

{\it Case 2.} 
We have
\begin{equation} \label{sikk}
s_{\gamma}s_{\eta}=s_{i_k}\cdots  \hat{s}_{i_j} \cdots  s_{i_2}s_{i_1}s_{i_2} \cdots   s_{i_k}
\end{equation}
for some $j$ between
$2$ and $k$. We set
$\beta=-s_{\gamma}(\eta)$, and $\alpha=-s_{\gamma}s_{\eta}
(\gamma)$.  Since $\ell(s_\gamma s_\eta) < \ell(s_\gamma)$ it follows that $\beta >0$. The identity (\ref{sikk}) and the expression for $s_\gamma$ imply that \[ s_\gamma s_\eta s_\gamma = s_{i_k} \cdots  s_{i_j} \cdots  s_{i_k} \/, \] therefore $\ell(s_\gamma s_\eta s_\gamma) \le 2 (k-j) + 1 < 2 (k-1) = \ell(s_\gamma s_\eta)$. Thus $\alpha > 0$. We have that $\beta^{\vee}=
-\eta^{\vee}+\langle \gamma, \eta^{\vee}\rangle \gamma^{\vee}$ and
$\alpha^{\vee}=\eta^{\vee}-(\langle \gamma, \eta^{\vee}\rangle-1)\gamma^{\vee}$,
which implies that $\alpha^{\vee}+\beta^{\vee}=\gamma^{\vee}$. We
can also easily check that $s_{\alpha}s_{\beta}=s_{\gamma} s_{\eta}$.
Same arguments as in the previous case show that $\alpha$ and $\beta$
are both in $\tilde{\Pi}_\af^{{re},+}$,  that
$\ell(s_{\alpha}s_{\beta})=\ell(s_{\alpha})+\ell(s_{\beta})$, and that
$s_{\alpha}s_{\beta} \neq s_{\beta}s_{\alpha}$. This finishes the proof.} \end{proof}

For later use, we also record the following result:

\begin{lemma}\label{lemma:nonzero}  Let $\alpha, \beta \in \tilde{\Pi}_\af^{{re},+}$ and $\gamma \in {\Pi}_\af^{{re},+}$ such that $\ell(s_\alpha s_\beta) = \ell(s_\alpha) + \ell(s_\beta)$ and $\gamma^\vee = \alpha^\vee + \beta^\vee $. Then $\gamma \in \tilde{\Pi}_\af^{{re},+}$ if and only if $\langle \alpha, \beta^\vee \rangle \neq 0$. \end{lemma}
\begin{proof} If $\langle \alpha, \beta^\vee \rangle \neq 0$ then Lemma \ref{lemma:multiple} implies that $\alpha^\vee + \beta^\vee < c$. Then $\gamma \in \tilde{\Pi}_\af^{{re},+}$ by Lemma \ref{lem:be}. Conversely, assume that $\gamma \in \tilde{\Pi}_\af^{{re},+}$ but $\langle \alpha, \beta^\vee \rangle = 0$. Let $u \in W_\af$ such that $s_\gamma u =s_\alpha s_\beta$. Then \[ \ell(s_\gamma u) =  \ell (s_\alpha s_\beta) = 2 \he(\alpha^\vee + \beta^\vee) - 2 = 2 \he (\gamma^\vee) -  2 = \ell(s_\gamma) - 1 \/.\] As in the proof of Lemma \ref{lemma:extra1}, one uses the Strong Exchange Condition to obtain that $u = s_\eta$ for $\eta \in {\Pi}_\af^{{re},+}$. Same lemma implies that $\langle \eta, \gamma^\vee \rangle =1$, so in particular $s_\gamma$ and $s_\eta$ do not commute. But $s_\gamma s_\eta = s_\alpha s_\beta = (s_\alpha s_\beta)^{-1} = s_\eta s_\gamma$ which is a contradiction. \end{proof}

\subsection{Quantum Bruhat chains and the proof of { Theorem \ref{thm:comm} (a)}}

In this section we introduce the notion of {\em quantum Bruhat chains}, which is our main tool in the proof of the commutativity of the Chevalley operators. The definition of these chains arises naturally from the study of the terms which appear in the quantum Chevalley formula, and it generalizes to the affine case the chains in the {\em quantum Bruhat graph} defined by Brenti, Fomin and Postnikov \cite{brenti.fomin.post:mixed} for ordinary flag varieties.

\begin{defn}  A  {\bf (weighted) quantum Bruhat cover} is a pair $u\to v$  where $u,v$ are in $W_\af$
and one of the following two conditions is satisfied:
\begin{itemize}
\item There exists  $\alpha\in \Pi^{re,+}_\af$ such that $v=us_\alpha$ and
$\ell(v)=\ell(u)+1$. The weight of this cover is 1 and this situation is denoted by
$u\stackrel{\alpha^\vee}{\longrightarrow}v$.
\item There exists  $\alpha\in \Pi^{re,+}_\af$ such that $v=us_\alpha$ and $\ell(v)=\ell(u)+1-2\he({\alpha^\vee})$.
The weight of this cover is $q^{\alpha^\vee}$ and this situation is denoted by
 $u\stackrel{q^{\alpha^\vee}}{\longrightarrow} v$. Notice that in this case $\ell(s_\alpha) = 2 \he (\alpha^\vee) -1$ by Lemma \ref{lemma:eq} thus $\alpha \in \tilde{\Pi}_\af^{re,+}$.
\end{itemize} 
A {\bf (weighted) quantum Bruhat chain} of weight $q^\kappa$ is an oriented sequence $u \to u s_\alpha \to u s_\alpha s_\beta$ of two quantum Bruhat covers such that the product of the two weights equals $q^\kappa$.
\end{defn}
Notice that we only use length 2 chains, although this notion can be obviously extended  to any length.
There are three types of quantum Bruhat chains: 
\begin{itemize}
\item[(1q)] $u\stackrel{\alpha^\vee}{\longrightarrow} us_\alpha\stackrel{q^{\beta^\vee}}{\longrightarrow} us_\alpha s_\beta$,
with weight $q^{\beta^\vee}$.
\item[(q1)]  $u\stackrel{q^{\alpha^\vee}}{\longrightarrow} us_\alpha\stackrel{\beta^\vee}{\longrightarrow} us_\alpha s_\beta$,
with weight $q^{\alpha^\vee}$.
\item[(qq)] $u\stackrel{q^{\alpha^\vee}}{\longrightarrow} us_\alpha\stackrel{q^{\beta^\vee}}{\longrightarrow} us_\alpha s_\beta$,
with weight $q^{\alpha^\vee +\beta^\vee}$.
\end{itemize}
We say that the last chain is of type {\rm (qq)'} if $\langle \alpha, \beta^\vee\rangle =0$; otherwise we say it is of type {\rm (qq)''}.

Given $u, v\in W_\af$ and $\kappa$ as before, we will determine all chains of weight $q^\kappa$ between
$u$ and $v$. Then we will use this information to show that the coefficient of $q^\kappa \eps_v$ in $\Lambda_i \Lambda_j (\eps_u)$ is symmetric in $i$ and $j$. This, together with the fact that $\coh^*(\Fla)$ is an associative ring (hence commutativity holds modulo $q_0, \ldots, q_n$) will complete the proof of part (a) of Theorem \ref{thm:comm}.

In our analysis we will repeatedly use the following Lemma:

\begin{lemma}\label{lem:unu} 
Let  $\alpha, \beta$ be in $\tilde{\Pi}_\af^{{re},+}$   and $u,v$  in $W_\af$ such that  $v=us_\alpha s_\beta$ and \begin{equation}\label{ust}u\stackrel{q^{\alpha^\vee}}{\longrightarrow} us_\alpha\stackrel{q^{\beta^\vee}}{\longrightarrow} us_\alpha s_\beta\end{equation}
is a quantum Bruhat chain. 
Then we have $\ell(s_\alpha s_\beta) =\ell(s_\alpha) + \ell(s_\beta)$.  
 \end{lemma}

 \begin{proof} The quantum Bruhat chain conditions for (\ref{ust}) give $\ell(us_\alpha s_\beta) = \ell(u) +2 -2\he({\alpha^\vee +\beta^\vee})$.
 On the other hand, we have
 $$\ell(us_\alpha s_\beta) \ge  \ell(u)-\ell(s_\alpha s_\beta) \ge \ell(u)-\ell(s_\alpha) -\ell(s_\beta) \ge
 \ell(u) +2 -\deg(q^{\alpha^\vee+\beta^\vee}),$$
 thus all inequalities here must actually be equalities.
    \end{proof}

To prove commutativity, we will fix $u, v \in W_\af$ and we distinguish three main cases: there exists a quantum Bruhat chain from $u$ to $v$ of weight $q^\kappa$ such that: (1) $\kappa\neq \gamma^\vee$ for any $\gamma \in \tilde{\Pi}_\af^{re,+}$; (2) $\kappa = \gamma^\vee$ is a non-simple coroot, for $\gamma \in \tilde{\Pi}_\af^{re,+}$; (3) $\kappa = \gamma^\vee$ is a simple coroot. 



\subsubsection{Case 1: There exists a chain from $u$ to $v$ of weight $q^\kappa$ and $\kappa \neq \gamma^\vee$, for any $\gamma \in \tilde{\Pi}_{\af}^{re,+}$.} { In this case there exists a chain
\[ \xymatrix{ u \ar[rr]^{q^{\alpha^\vee}} && us_\alpha \ar[rr]^{q^{\beta^\vee}} && v = us_\alpha s_\beta} \/, \] where
$\alpha, \beta \in \Pi_\af^{re,+}$ satisfy $\kappa = \alpha^\vee + \beta^\vee$. If $\kappa \nless c$ and $\langle \alpha, \beta^\vee\rangle \neq 0$ then Lemma \ref{lemma:multiple} implies that $\kappa=\alpha^\vee + \beta^\vee = c$. We are only interested in commutativity modulo $q^c$, thus by Proposition \ref{prop:pairbij} and Lemma \ref{lemma:nonzero} we may assume from now on that $\langle \alpha, \beta^\vee \rangle = 0$.}
Then there exists another chain of the form \[ \xymatrix{ u \ar[rr]^{q^{\beta^\vee}} && us_\beta \ar[rr]^{q^{\alpha^\vee}} && v = us_\beta s_\alpha} \/, \] and the exchange of $\alpha$ and $\beta$ gives an involution on the set of chains from $u$ to $v$ of weight $q^\kappa$. The coefficient of $q^\kappa \eps_v$ in $\Lambda_i \Lambda_j(\eps_u)$ is \[ \sum_{\alpha^\vee +\beta^\vee = \kappa} \langle \lambda_i, \alpha^\vee \rangle \langle \lambda_j, \beta^\vee \rangle + \sum_{\alpha^\vee +\beta^\vee = \kappa} \langle \lambda_i, \beta^\vee \rangle \langle \lambda_j, \alpha^\vee \rangle \/, \] which is symmetric in $i,j$ and the proof is finished in this case.


\subsubsection{Case 2: There exists a chain from $u$ to $v$ of weight $q^\kappa$, where $\kappa = \gamma^\vee$ is a non-simple coroot with $\gamma \in \tilde{\Pi}_\af^{re,+}$.} Denote by $\pi: u \to v$ the chain from $u$ to $v$ given by the hypothesis. 


We claim that the Weyl group element $s_\gamma u^{-1}v$ is a root reflection corresponding to an affine real root. This is clear if the chain $\pi$ is of type (q1). If it is of type (1q) then $u^{-1} v s_\gamma$ is a root reflection, but then so is $s_\gamma u^{-1}v = s_\gamma (u^{-1} v s_\gamma) s_\gamma$. Finally, if $\pi$ is of type (qq) this follows from Proposition \ref{prop:pairbij}. Then we can define $\eta, \eta' \in \Pi_\af^{re,+}$ as the unique positive real roots given by \[ s_\eta:= s_\gamma u^{-1}v; \quad s_{\eta'}:= u^{-1} v s_\gamma \/. \] Notice that from definition it follows that $\eta' = \pm s_\gamma(\eta)$, and this leads to $2$ situations, according to whether the plus or minus sign occurs.

\noindent {\em Case 2.1: $\eta' = s_\gamma(\eta)$.} { We first notice that there is no chain of type (qq) between $u$ and $v$. If this were the case, by Proposition \ref{prop:pairbij} there would be one of type (q1) of the form $\xymatrix{ u \ar[r]^{q^{\gamma^\vee}} & us_\gamma \ar[r]^{{\eta''^\vee}} & v = us_\gamma s_{\eta''}}$ with $\ell(s_\gamma s_{\eta''}) = \ell(s_\gamma) -   1$. But then $\eta'' = \eta$ and $s_\gamma(\eta'') < 0$ which is a contradiction.} 
We now claim that there exist exactly $2$ quantum Bruhat chains of weight $q^\kappa$ between $u$ and $v$ given by:
\begin{equation}\label{eq:2chains} \xymatrix{ u \ar[r]^{q^{\gamma^\vee}} & us_\gamma \ar[r]^{{\eta^\vee}} & v = us_\gamma s_\eta}, \quad \xymatrix{u \ar[r]^{{\eta'^\vee}} & us_{\eta'} \ar[r]^{q^{\gamma^\vee}} & v = us_{\eta'} s_\gamma} \/.\end{equation}

To prove the claim note that the root $u(\eta') > 0$ if and only if $us_\gamma(\eta) > 0$. Since one of these two chains exists (by hypothesis and from the definition of a quantum Bruhat chain) the equivalence shows the other exists as well.
Again by definition of a quantum Bruhat chain and the uniqueness of $\gamma$ it follows that there cannot be another chain of type (1q) or (q1). The claim is proved.


Then the coefficient of $q^\kappa \eps_v$ in $\Lambda_i \Lambda_j(\eps_u)$ is \[ \begin{split} \langle \lambda_j, \gamma^\vee \rangle \langle \lambda_i, \eta^\vee \rangle + \langle \lambda_j, \eta'^\vee \rangle \langle \lambda_i, \gamma^\vee \rangle = \langle \lambda_j, \gamma^\vee \rangle \langle \lambda_i, \eta^\vee \rangle + \langle \lambda_j, s_\gamma(\eta^\vee) \rangle \langle \lambda_i, \gamma^\vee \rangle \\ = \langle \lambda_j, \gamma^\vee \rangle \langle \lambda_i, \eta^\vee \rangle + \langle \lambda_i, \gamma^\vee \rangle \langle \lambda_j, \eta^\vee \rangle -  \langle \gamma, \eta^\vee \rangle \langle \lambda_i, \gamma^\vee \rangle \langle \lambda_j, \gamma^\vee \rangle \/. \end{split}
 \] This is clearly symmetric in $i$ and $j$.

 \noindent{\em Case 2.2: $\eta'= -   s_\gamma(\eta)$.} Notice that in this case $us_\gamma(\eta) > 0$ if and only if $u(\eta') < 0$ therefore exactly one of the two chains from (\ref{eq:2chains}) can exist: if $us_\gamma(\eta) > 0$ then it is the first chain, otherwise it is the second. In both situations Proposition \ref{prop:pairbij} yields a unique chain of the form \[ \xymatrix{ u \ar[rr]^{q^{\alpha^\vee}} && us_\alpha \ar[rr]^{{q^{\beta^\vee}}} && v = us_\alpha s_\beta} \] where $\alpha^\vee + \beta^\vee = \gamma^\vee$, $s_\alpha s_\beta \neq s_\beta s_\alpha$ and $s_\alpha s_\beta = s_\gamma s_\eta= s_{\eta'} s_\gamma$. From the formulas (\ref{eq:gammae}) it follows that \[ \gamma^\vee = \alpha^\vee + \beta^\vee; \quad \eta^\vee = - \beta^\vee - \langle \alpha, \beta^\vee \rangle \gamma^\vee; \quad \eta'^\vee = \langle \gamma, \eta^\vee \rangle \gamma^\vee - \eta^\vee \/.  \] To conclude, in this case there exist exactly two chains between $u$ and $v$, one of type (qq)" and the other of type (1q) or (q1), depending on the sign of the root $u (\eta')$. We will analyze  both cases.
 
Assume that we have the chain \[  \xymatrix{ u \ar[rr]^{q^{\gamma^\vee}} && us_\gamma \ar[rr]^{{{\eta^\vee}}} && v = us_\gamma s_\eta} \/. \] Then the coefficient of $q^\kappa \eps_v$ in $\Lambda_i \Lambda_j(\eps_u)$ is \[ \begin{split} \langle \lambda_j,\gamma^\vee \rangle \langle \lambda_i, \eta^\vee \rangle + \langle \lambda_j, \alpha^\vee \rangle \langle \lambda_i, \beta^\vee \rangle =  - \langle \lambda_j,\gamma^\vee \rangle \langle \lambda_i, \beta^\vee \rangle 
- \langle \alpha, \beta^\vee \rangle \langle \lambda_i, \gamma^\vee \rangle \langle \lambda_j, \gamma^\vee \rangle 
+ \langle \lambda_j, \alpha^\vee \rangle \langle \lambda_i, \beta^\vee \rangle \\ = \langle \lambda_i, \beta^\vee \rangle \langle \lambda_j, \alpha^\vee - \gamma^\vee \rangle -  \langle \alpha, \beta^\vee \rangle \langle \lambda_i, \gamma^\vee \rangle \langle \lambda_j, \gamma^\vee \rangle = - \langle \lambda_i, \beta^\vee \rangle \langle \lambda_j, \beta^\vee \rangle -  \langle \alpha, \beta^\vee \rangle \langle \lambda_i, \gamma^\vee \rangle \langle \lambda_j, \gamma^\vee \rangle \/. \end{split} \]

If we have the chain \[ \xymatrix{ u \ar[rr]^{{\eta'^\vee}} && us_{\eta'} \ar[rr]^{q^{\gamma^\vee}} && v = us_{\eta'} s_\gamma}\] then the coefficient of $q^\kappa \eps_v$ in $\Lambda_i \Lambda_j(\eps_u)$ is \[ \begin{split} & \langle \lambda_j, \eta'^\vee \rangle \langle \lambda_i, \gamma^\vee \rangle +   \langle \lambda_j, \alpha^\vee \rangle \langle \lambda_i, \beta^\vee \rangle = \\  & = - \langle \lambda_j, \eta^\vee \rangle \langle \lambda_i, \gamma^\vee \rangle +\langle \gamma, \eta^\vee \rangle \langle \lambda_j, \gamma^\vee \rangle \langle \lambda_i, \gamma^\vee \rangle + \langle \lambda_j, \alpha^\vee \rangle \langle \lambda_i, \beta^\vee \rangle \\ & =   \langle \lambda_j, \beta^\vee+ \langle \alpha, \beta^\vee \rangle \gamma^\vee \rangle \langle \lambda_i, \gamma^\vee \rangle +  \langle \lambda_j, \alpha^\vee \rangle \langle \lambda_i, \beta^\vee \rangle + \langle \gamma, \eta^\vee \rangle \langle \lambda_j, \gamma^\vee \rangle  \langle \lambda_i, \gamma^\vee \rangle\\ & = \langle \alpha, \beta^\vee \rangle \langle \lambda_i, \gamma^\vee \rangle   \langle \lambda_j, \gamma^\vee \rangle + \langle \lambda_j, \beta^\vee \rangle \langle \lambda_i, \gamma^\vee \rangle + \langle \lambda_j, \alpha^\vee \rangle \langle \lambda_i, \beta^\vee \rangle + \langle \gamma, \eta^\vee \rangle \langle \lambda_j, \gamma^\vee \rangle  \langle \lambda_i, \gamma^\vee \rangle \\  & = \langle \alpha, \beta^\vee \rangle \langle \lambda_i, \gamma^\vee \rangle   \langle \lambda_j, \gamma^\vee \rangle + \langle \lambda_j, \beta^\vee \rangle \langle \lambda_i, \alpha^\vee +\beta^\vee \rangle + \langle \lambda_j, \alpha^\vee \rangle \langle \lambda_i, \beta^\vee \rangle \\ & + \langle \gamma, \eta^\vee \rangle \langle \lambda_j, \gamma^\vee \rangle  \langle \lambda_i, \gamma^\vee \rangle \\ & = \langle \alpha, \beta^\vee \rangle \langle \lambda_i, \gamma^\vee \rangle   \langle \lambda_j, \gamma^\vee \rangle + (\langle \lambda_j, \beta^\vee \rangle \langle \lambda_i, \alpha^\vee \rangle + \langle \lambda_j, \alpha^\vee \rangle \langle \lambda_i, \beta^\vee \rangle) + \langle \lambda_j, \beta^\vee \rangle \langle \lambda_i, \beta^\vee \rangle  \\ & + \langle \gamma, \eta^\vee \rangle \langle \lambda_j, \gamma^\vee \rangle  \langle \lambda_i, \gamma^\vee \rangle \/. \end{split} \] 

In both situations the coefficient is symmetric in $i$ and $j$ and we are done.

\subsubsection{Case 3: There exists a chain from $u$ to $v$ of weight $q^\kappa$, where $\kappa = \gamma^\vee$ is a simple coroot.} This case is similar to Case 2 above, but simpler because there cannot be a chain of type (qq) between $u$ and $v$ (since $\gamma$ is a simple root). So we will be brief. As in Case 2 we can define the affine real positive roots $\eta, \eta' \in \Pi_\af^{re,+}$ by $s_\eta:= s_\gamma u^{-1}v$ and $s_{\eta'}:= u^{-1} v s_\gamma$.  We have again two situations, depending on whether $s_\gamma(\eta)$ equals $\eta'$ or $-\eta'$. The case when $s_\gamma(\eta) = \eta'$ is identical to the case 2.1, except for the fact that we do not need to prove again the non-existence of the chain of type (qq). Assume now that $s_\gamma(\eta) = -\eta'$. Then $s_\gamma(\eta) = \eta - \langle \eta, \gamma^\vee \rangle \gamma $ is a negative root. Since $\gamma$ is a simple root this implies that $\eta = \gamma$ and $\langle \eta, \gamma^\vee \rangle = 2$, thus $\eta'= -\eta$. Therefore $u=v$ and there exists exactly one chain between $u$ and $v$:
\[ \begin{cases} \xymatrix{ u \ar[r]^{q^{\gamma^\vee}} & u s_\gamma \ar[r]^{\gamma^\vee} & v=u} & \textrm{ if } u(\gamma) < 0 \\   \xymatrix{ u \ar[r]^{\gamma^\vee} & u s_\gamma \ar[r]^{q^{\gamma^\vee}} & v=u} & \textrm{ if } u(\gamma) > 0 \/. \end{cases} \] In both situations, the coefficient of $q^\kappa \eps_v$ in $\Lambda_i \Lambda_j(\eps_u)$ is $\langle \lambda_i, \gamma^\vee \rangle \langle \lambda_j, \gamma^\vee \rangle$, which is clearly symmetric in $i$ and $j$.

\subsection{{ Proof of Theorem \ref{thm:comm} (b): commutativity of the operators $\Lambda_i - m_i \Lambda_0$}} 

Recall that $\theta^\vee = m_1\alpha_1^\vee +\cdots  + m_n \alpha_n^\vee$.
In this section we prove:

\begin{thm}\label{genstar} For any $1\le i , j \le n$ and any $u\in W_\af$ we have
\begin{equation}\label{fori}(\Lambda_i - m_i \Lambda_0)(\Lambda_j - m_j \Lambda_0) \ep_u = (\Lambda_j - m_j \Lambda_0)(\Lambda_i - m_i \Lambda_0) \ep_u.\end{equation}\end{thm}

\begin{proof} As before, we need to show that for any $v\in W_\af$ and any $\kappa$ in the affine coroot lattice, the
coefficient of $q^\kappa \ep_v$ in the left hand side of (\ref{fori}) is symmetric in $i$ and $j$.
By Theorem \ref{thm:comm} (a) proved in the previous section this is true for  $\kappa < c$, so from now on we will assume that $\kappa \nless c=\alpha_0^\vee + \theta^\vee$. Then Proposition \ref{letalpha} and Lemma \ref{lemma:multiple} imply that
there exists a quantum Bruhat chain of type (qq):
\[ \xymatrix{ u \ar[rr]^{q^{\alpha^\vee}} && us_\alpha \ar[rr]^{{q^{\beta^\vee}}} && v = us_\alpha s_\beta} \] 
such that $\kappa = \alpha^\vee + \beta^\vee$, and either $s_\alpha s_\beta = s_\beta s_\alpha$ or $s_\alpha s_\beta \neq s_\beta s_\alpha$ but then $\kappa= \alpha^\vee + \beta^\vee = c$. The assumption on $\kappa$ implies that there cannot be any chains of type (1q) or (q1). If $s_\alpha s_\beta = s_\beta s_\alpha$ then as in Case 1 from the previous section the coefficient of $q^{\alpha^\vee +\beta^\vee}\ep_{us_\alpha s_\beta}$ in the left hand side of
(\ref{fori}) is symmetric in $i$ and $j$. If $s_\alpha s_\beta \neq s_\beta s_\alpha$ then $\alpha^\vee + \beta^\vee = c$ and the coefficient of $q^{\alpha^\vee +\beta^\vee}\ep_{us_\alpha s_\beta}$ in the left hand side of
(\ref{fori}) is
\[ \langle \lambda_j,\alpha^\vee\rangle \langle \lambda_i,\beta^\vee\rangle
-m_i\langle \lambda_j,\alpha^\vee\rangle \langle \lambda_0,\beta^\vee\rangle
-m_j\langle \lambda_0,\alpha^\vee\rangle \langle \lambda_i,\beta^\vee\rangle
+m_i m_j \langle \lambda_0,\alpha^\vee\rangle \langle \lambda_0,\beta^\vee\rangle \/.\]
After ignoring the last term { (since it is symmetric in $i,j$)}, replacing $\beta^\vee$ by $c -\alpha^\vee$, 
and using that $\langle \lambda_i, c\rangle =m_i$ for $i\neq 0$ as well as
$\langle \lambda_0, c\rangle =1$, the expression above is equal to 
\[\begin{split}  -\langle \lambda_i,\alpha^\vee\rangle \langle \lambda_j,\alpha^\vee\rangle + m_i \langle \lambda_j, \alpha^\vee \rangle \langle \lambda_0, \alpha^\vee \rangle +  m_j \langle \lambda_i, \alpha^\vee \rangle \langle \lambda_0, \alpha^\vee \rangle  \\ = \langle m_i \lambda_j + m_j \lambda_i, \alpha^\vee \rangle \langle \lambda_0, \alpha^\vee \rangle - \langle \lambda_i,\alpha^\vee\rangle \langle \lambda_j,\alpha^\vee\rangle \/,\end{split} \]
which is symmetric in $i$ and $j$. \end{proof}
 
\section{The Frobenius property and the Dubrovin formalism}\label{frodu}

The goal of this section is to show that the algebra $\quantum^*_\af(G/B)$ is a Frobenius algebra, and to define an analogue $\nabla^\hbar$ of the Dubrovin connection, parametrized by a complex parameter $\hbar$. Then we will show how the associativity of the product $\star_\af$ implies that $\nabla^\hbar$ is flat. These facts will be used in the next section to define the Givental-Kim formalism for $\quantum^*_\af(G/B)$, which eventually leads to a presentation of this algebra by generators and relations. For the reminder of the paper we will consider the cohomology $\mathrm{H}^*(G/B)$ with complex coefficients, and $\QH^*_\af(G/B)$ as an algebra over $\C[q]:=\C[q_0, \ldots, q_n]$.

We abuse notation and denote by $\langle \cdot, \cdot \rangle$ the Poincar{\'e} pairing defined by \[ \langle a , b \rangle = \int_{G/B} a \cdot b  \] where $a, b \in \coh^*(G/B)$ and the integral means the push-forward to a point. (Equivalently this equals the coefficient of $[pt]$ in $\sigma_u \cdot \sigma_v$.) We extend this pairing by $\C[q]$-linearity to define a pairing on $\quantum^*_\af(G/B)$. 
We prove in Theorem \ref{thm:frobenius} below that relative to this pairing, $\quantum^*_\af(G/B)$ is a Frobenius algebra (see e.g.~\cite[\S 9.2]{Gu} for this notion). We need the following lemma. 
\begin{lemma}\label{lemma:symmp} Let $a,b \in \coh^*(G/B)$ and let $\alpha \in \Pi$ be in the finite root system. Then $\langle \partial_\alpha (a), b \rangle = \langle a, \partial_\alpha(b) \rangle$.  \end{lemma}

\begin{proof} We first prove the lemma in the case $\alpha = \alpha_i$ is a simple root. Let $P_i$ be the minimal parabolic subgroup corresponding to $\alpha_i$ and let $\pi_i: G/B \to G/P_i$ be the projection. It is well known that $\partial_i (a) = \pi^*_i (\pi_i)_*(a)$ where $(\pi_i)_*: \coh^k (G/B) \to \coh^{k-2}(G/B)$ is the Gysin push-forward. (This formula for $\partial_i$ can be traced back to \cite[\S 5.2]{BGG}; we refer e.g.~  to \cite{fulton.macpherson} or \cite[Appendix]{mihalcea:positivity} for more on Gysin push-forwards.) Then by the projection formula \[ \int_{G/B} \partial_i(a) \cdot b = \int_{G/B} \pi^*_i (\pi_i)_*(a) \cdot b  = \int_{G/P_i} (\pi_i)_*(a) \cdot (\pi_i)_*(b) \/. \] Then the identity in the lemma for $\alpha = \alpha_i$ follows because the last expression is symmetric in $a$ and $b$. For the general case we use the definition of $\partial_\alpha$ from (\ref{eq:palpha}) above. Let $w \in W$ such that $w(\alpha_i) = \alpha$ for some $1 \le i \le n$. Then $\partial_\alpha = r_w^* \partial_i r_{w^{-1}}^*$ where $r_w^*:\coh^*(G/B) \to \coh^*(G/B)$ is the ring automorphism induced by the right action of $W$ on cohomology.
 Because $r_w^* r_{w^{-1}}^* = id$ it follows that $r_w^*[pt] = r^*_{w^{-1}}[pt] = \pm [pt] \in \coh^*(G/B)$.\begin{footnote}{In fact it can be shown that $r_w^*[pt ] = (-1)^{\ell(w)} [pt]$. This follows because $[pt] = \frac{1}{|W|} \prod_{\beta \in \Pi^+} \beta$ where the right hand side is interpreted in the BGG presentation described after (\ref{delt}) above.}\end{footnote} 
This implies that $\langle r_w^*(a_1) , a_2 \rangle = \pm \langle a_1, r_{w^{-1}}^* (a_2) \rangle$, for any $ a_1, a_2 \in \coh^*(G/B)$. Then \[ \begin{split} \langle \partial_\alpha (a), b \rangle = \langle r_w^* \partial_i r_{w^{-1}}^*(a), b \rangle = \pm \langle \partial_i r_{w^{-1}}^*(a), r_{w^{-1}}^*(b) \rangle \\ = \pm \langle r_{w^{-1}}^*(a), \partial_i r_{w^{-1}}^*(b) \rangle = \langle a , r_w^* \partial_i r_{w^{-1}}^*(b) \rangle = \langle a, \partial_\alpha(b) \rangle\/. \end{split} \] The third equality used the previously proved identity $\langle \partial_i (a), b \rangle = \langle a, \partial_i(b) \rangle$.  
\end{proof}

\begin{thm}\label{thm:frobenius} For any $a_1, a_2, a_3 \in \quantum_\af^*(G/B)$ there is an identity $\langle a_1 \star_\af a_2, a_3\rangle  = \langle a_1, a_2 \star_\af a_3 \rangle$. In particular, the algebra $(\quantum^*_\af(G/B), \star_\af)$ is a Frobenius algebra with respect to $\langle \cdot, \cdot \rangle$. \end{thm} 

\begin{proof} Recall that $\quantum^*_\af(G/B)$ has a $\bC[q]$-basis given by Schubert classes, and it is generated as an algebra over $\bC[q]$ by the Schubert divisors $\sigma_1, \ldots, \sigma_n$. Then we can assume that $a_1 = \sigma_u, a_3 = \sigma_v$ and $a_2 = \sigma_i$ for some $u, v \in W$ and $1 \le i \le n$. By the Chevalley formula from (\ref{eq:Lambdafin}) above we calculate that \[ \langle \sigma_u \star_\af \sigma_i, \sigma_v \rangle = \langle \sigma_u \cdot \sigma_i, \sigma_v \rangle + \sum_{\alpha} \langle \lambda_i - m_i \lambda_0, \alpha^\vee \rangle q^{\alpha^\vee} \langle \pi(D_{s_\alpha}) (\sigma_u), \sigma_v \rangle \/, \] where the sum is as in (\ref{eq:Lambdafin}) and $\pi: R_\af \to R$ is the homomorphism between the nil-Coxeter rings defined in Theorem \ref{T:ev}. Notice now that \[ \langle \sigma_u \cdot \sigma_i, \sigma_v \rangle = \int_{G/B} \sigma_u \cdot \sigma_i \cdot \sigma_v = \langle \sigma_u , \sigma_i \cdot \sigma_v \rangle \/.\] 
Thus to prove the theorem it remains to show that $\langle \pi(D_{s_\alpha})(\sigma_u), \sigma_v \rangle = \langle \sigma_u, \pi(D_{s_\alpha})(\sigma_v) \rangle $. Choose a reduced word $s_\alpha = s_{i_1} s_{i_2} \cdots  s_{i_k}$ where the indices $i_j \in \{ 0, 1,\ldots, n\}$.~Theorem \ref{T:ev} implies that $\pi(D_{s_\alpha}) = \pi(D_{i_1}) \cdots  \pi(D_{i_k})$ and each $\pi(D_{i_j})$ equals to either $\partial_{i_j}$ or $\partial_{-\theta}$. Then we can successively apply Lemma \ref{lemma:symmp} to get
\[ \langle \pi(D_{s_\alpha})(\sigma_u), \sigma_v \rangle = \langle \pi(D_{i_1}) \cdots  \pi(D_{i_k})(\sigma_u), \sigma_v \rangle = \langle \sigma_u, \pi(D_{i_k})\cdots \pi(D_{i_1}) (\sigma_v) \rangle  = \langle \sigma_u, \pi(D_{s_\alpha})(\sigma_v) \rangle \/, \] where the last equality follows because the reduced word of $s_\alpha$ is symmetric. \end{proof} 

We now turn to the definition of the Dubrovin connection associated to the quantum cohomology ring $\quantum^*_\af(G/B)$. Recall that a connection on a vector bundle $E \to M$ is an operator $\nabla: \Gamma(E) \otimes \Gamma(TM) \to \Gamma(E)$ where $TM$ denotes the tangent bundle of $M$ and $\Gamma(E)$ denotes the ring of sections of $E$. The operator $\nabla$ must satisfy the following properties:
\begin{itemize} \item $\nabla_X (f \sigma + \tau) = X(f) \sigma + f \nabla_X(\sigma) +\nabla_X (\tau)$ for any $\sigma, \tau \in \Gamma(E)$, $f \in \cO_M$ (the ring of functions on $M$) and $X \in \Gamma(TM)$;

\item $\nabla_{f X + Y}(\sigma) = f \nabla_X(\sigma) + \nabla_Y(\sigma)$, for any $f \in \cO_M$, $\sigma \in \Gamma(E)$ and $X,Y \in \Gamma(TM)$. \end{itemize}

Consider now the complex vector space $M: = \coh^2(G/B)$ with coordinates $z= (z_1, \ldots, z_n)$ corresponding to the basis $\sigma_1, \ldots, \sigma_n$. { This means that $z_i$ is the function $\coh^2(G/B) \to \C$ defined by $z_i(\sigma_j) = \delta_{i,j}$, the Kronecker delta symbol; it can also be identified to the homology class $[X(s_i)] \in \coh_2(G/B) $ of the corresponding Schubert curve.} We regard $M$ as a formal scheme, i.e. $\cO_{M}= \C[[z]]:= \C[[z_1, \ldots, z_n]]$. There is a trivial bundle $E:= M \times \coh^*(G/B)$. A section $s \in \Gamma(E)$ can be written as a {\em finite} sum $\sum a_{w}(z) \sigma_w$, where $a_{w}(z) \in \cO_{M}$. 

\begin{defn}\label{defn:connection} Let $\hbar \in \C^*$ be a complex parameter. Define $\nabla^\hbar: \Gamma(E) \otimes \Gamma(TM) \to \Gamma(E)$ to be the unique connection on $E$ which satisfies \[ \nabla^\hbar_{\partial / \partial z_i} (\sigma_w) = \frac{1}{\hbar}\sigma_i \star_{\af} \sigma_w, \quad (1 \le i \le r) \/. \] In order to regard $\sigma_i \star_{\af} \sigma_w$ in $\Gamma(E)$ we use the substitutions \begin{equation}\label{E:hyp} q_i := e^{z_i}~  (1 \le i \le n); \quad  q_0 := e^{- (m_1 z_1 + \cdots  + m_n z_n)}\/.\end{equation} \end{defn}

\begin{remark}\label{rmk:subs}To motivate the last substitution $q_0 = e^{- (m_1 z_1 + \cdots  + m_n z_n)}$ recall that the quantum
 product $\star_\af$ we defined on $\coh^*(G/B)$ used the affine quantum Chevalley operators which we transported from $\coh^*(\Fla)$ via the injective ring homomorphism $\e_1^*:\coh^*(G/B) \to \coh^*(\Fla)$. Let $z_\af:=(z_0', z_1', \ldots, z_n')$ be the coordinates on $\coh^2(\Fla)$ corresponding to the Schubert basis $\eps_0, \ldots, \eps_n$.
Then we identify $\sigma_i \leftrightarrow \e_1^*(\sigma_i) = \eps_i - m_i \eps_0$, $1\le i \le n$, which leads to identifications
$z_i \leftrightarrow z'_i$. In addition, notice that $\e_1^*(\coh^2(G/B))$ is defined by the equation $z'_0+m_1z'_1 +\cdots  +m_n z'_n=0$ in $\coh^2(\Fla)$, so one can formally define an extra coordinate $z_0$ on $\coh^2(G/B)$ by $z_0 := -(m_1 z_1 + \cdots  + m_n z_n)$. The quantum parameters $q_i'= e^{z_i'}$ ($0 \le i \le n$) and $q_i=e^{z_i}$ ($1 \le i \le n$) are regarded as functions on $\coh^2(\Fla)$, respectively $\coh^2(G/B)$, and they transform with respect to the dual of $\e_1^*$, which is $(\e_1)_*: \coh_2(\Fla) \to \coh_2(G/B)$. By the above identifications (or by applying the identity (\ref{eq:pfp}) above) one obtains that $q_i' \leftrightarrow q_i$ for $1 \le i \le n$ and $q_0' \leftrightarrow q_0:=e^{z_0} = q_1^{-m_1} \cdots q_n^{- m_n}$. 
 \end{remark} 

Of particular importance for us later will be the subgroup of sections of the form \[\Gamma(E)':=\{s:   s = \sum g_{d,w}(z) q^d \sigma_w;\quad  g_{d,w}(z) \in \C[z] \} \] where the (possibly infinite) sum is over $d = (d_0, d_1, \ldots, d_n) \in (\Z_{\ge 0})^{n+1}$ and $w \in W$. Then for any $1 \le i \le n$ and $g(z) \in \C[z]$, we record the following:
\begin{equation}\label{E:goodsections} \nabla^\hbar_{\partial/ \partial z_i} (g(z) q^d \sigma_w) = \Bigl(\frac{\partial g(z)}{\partial z_i}+ (d_i - m_i d_0) g(z)\Bigr) q^d \sigma_w + \frac{1}{\hbar} g(z) q^d \sigma_i \star_{\af} \sigma_w \/. \end{equation}

\begin{thm}\label{thm:flat} The Dubrovin connection $\nabla^\hbar$ is flat for any $\hbar \in \C^*$. \end{thm}

\begin{proof} Since the vector fields $\partial /\partial z_i$ and $\partial /\partial z_j$ on $TM$ commute with respect to the Lie bracket, it suffices to show that \[ \nabla_{\partial /\partial z_i} ^\hbar \nabla_{\partial /\partial z_j}^\hbar (\sigma_w) =  \nabla_{\partial /\partial z_j}^\hbar  \nabla_{\partial /\partial z_i}^\hbar (\sigma_w) \/.\]
A calculation based on the Chevalley formula from (\ref{eq:Lambdafin}) and the identity (\ref{E:goodsections}) above shows that \[ \nabla_{\partial /\partial z_i} ^\hbar \nabla_{\partial /\partial z_j}^\hbar (\sigma_w) = \frac{1}{\hbar} \sum_{\alpha} \langle \lambda_j - m_j \lambda_0, \alpha^\vee \rangle \cdot \langle  \lambda_i - m_i \lambda_0, \alpha^\vee \rangle q^{\alpha^\vee} \pi(D_{s_\alpha})(\sigma_w) + \frac{1}{\hbar^2} \sigma_i \star_\af (\sigma_j \star_\af \sigma_w) \/, \] where the sum is over $\alpha \in \tilde{\Pi}^{re,+}_\af$ such that $\ell(w s_\alpha) = \ell(w) + 1 - 2 \he(\alpha^\vee)$. The sum is symmetric in indices $i$ and $j$ and $\sigma_i \star_\af (\sigma_j \star_\af \sigma_w) = \sigma_j \star_\af (\sigma_i \star_\af \sigma_w)$ by Theorem \ref{thm:comm}(b) and the definition of $\star_\af$ from (\ref{def:staraf}). Therefore the roles of $i$ and $j$ can be interchanged, and this finishes the proof.\end{proof}

\section{Towards relations in $\quantum^*_\af(G/B)$: the Givental-Kim formalism}\label{s:givform} Our goal from now on is to show that the ideal of relations in the ring $\quantum^*_\af(G/B)$ is given by the non-constant integrals of motion for the periodic Toda lattice associated to the dual of the extended Dynkin diagram of $G$. In fact, a result of Guest and Otofuji \cite{go} for $G$ of type A, generalized by Mare to types A-C  in \cite{Ma3}, shows that if there exists a quantum cohomology ring for $\Fla$ satisfying certain natural properties, then the relations are obtained from the conserved quantities of the periodic Toda lattice. { The ring $\quantum^*_\af(G/B)$ satisfies the analogue of all these properties, thus one can obtain the ideal of relations at least in types A-C.}



{ We are pursuing here a slightly different approach, which emphasizes the role of the quantum differential equations and of the Dubrovin - Givental formalism.~This will lead to relations in all types.}~
We begin by adapting a method of Givental and Kim in \cite{givental:egw} and \cite{Kim} which produces relations in quantum cohomology; our approach is inspired by the presentation of this method given by Cox and Katz in \cite{cox.katz:mirror}. This method uses the {\em quantum differential equations} associated to a renormalization of the Dubrovin connection, called the {\em Givental connection}, and certain differential operators acting on solutions of these equations. A particular specialization of these operators leads to relations in the quantum cohomology ring. In the next section we will identify these operators with integrals of motion for the quantum, dual version of the periodic Toda lattice.

Consider the ring of operators $\C[[e^z]][\partial/\partial z][\hbar]$ where \begin{equation}\label{def:eu} z=(z_1, \ldots, z_n);  \quad \partial/\partial z = (\partial/\partial z_1, \ldots, \partial/\partial z_n); \quad e^z = (e^{z_0}, e^{z_1},\ldots,e^{z_n}) \/,\end{equation} and recall the convention $q_0=e^{z_0} = e^{-(m_1 z_1 + \cdots  + m_n z_n)}$ from (\ref{E:hyp}). These operators act on $\C[[e^z]][z][\hbar, \hbar^{-1}]$ as follows. First, $\hbar, \hbar^{-1}$ and $e^{z_i}$ act by multiplication. Now let $g(z; \hbar) \in \C[z; \hbar,\hbar^{-1}]$ and $d= (d_0, \ldots, d_n) \in \Z_{\ge 0}^{n+1}$; set $e^{zd} = e^{d_0 z_0 + \cdots  + d_n z_n}$. For $1 \le i \le n$ the operator $\partial/\partial z_i$ acts by \begin{equation}\label{E:partialzi} \frac{\partial}{\partial z_i}  \bigl(g(z;\hbar) e^{zd}\bigr) = \Bigl( \frac{\partial g}{\partial z_i}(z;\hbar) + (d _i - m_i d_0) g(z;\hbar)\Bigr) e^{zd} \/. \end{equation} (This is the action compatible with the Dubrovin connection; see equation (\ref{E:goodsections}) above.)

\begin{defn}\label{def:giventalcon} Let $\hbar \in \C^*$. The {\em Givental connection} $\nabla$ is defined to be  $\nabla := \hbar \nabla^{-\frac{1}{\hbar}}$, i.e. if $s= g(z) q^d \sigma_w \in \Gamma(E)'$ is as in (\ref{E:goodsections}) then\[ \nabla_{\partial/\partial z_i}(s) =  \hbar \Bigl(\frac{\partial g(z)}{\partial z_i} + (d_i - m_i d_0) g(z) \Bigr) q^d \sigma_w -  g(z)  q^d \sigma_i \star_{\af} \sigma_w \/.\] (Technically, $\frac{1}{\hbar} \nabla$ is a connection, but we follow tradition from \cite{givental:egw,Kim} to use this form.) We will also use the ``dual" Givental connection, defined by $\hat{\nabla} := \hbar \nabla^{\frac{1}{\hbar}}$. \end{defn}

The flatness of Dubrovin connection from Theorem \ref{thm:flat} implies that both $\nabla$ and $\hat{\nabla}$ are flat.

\begin{defn}\label{def:qdiffeqs}
The system of {\em quantum differential equations} is the system of PDE's given by \begin{equation}\label{E:qdiff} \hbar \partial/\partial z_i (s) = \sigma_i \star_{\af} s \Longleftrightarrow \nabla_{\partial/\partial z_i} (s) = 0;  \quad i \in \{ 1, \ldots, n \} \/.\end{equation} \end{defn} 

Since $\nabla$ is flat, the classical Frobenius theorem (see e.g.~\cite[\S 4]{sharpe} or \cite[p.~36]{Gu} for a context similar to ours) ensures that $C^\infty$ solutions of this system exist in some neighborhood of $z=0$. However, in what follows we require the existence of formal solutions $s \in \Gamma(E)'$. In the ordinary quantum cohomology a fundamental solution of this system was constructed by Givental \cite{givental:egw} (see also \cite[\S 10.2]{cox.katz:mirror}) using $2$-point Gromov-Witten invariants with a gravitational descendent. Since in our case we do not have a moduli space to help define the quantum multiplication, we will need to construct these solutions directly. For this we rely on results from the aforementioned paper \cite{Ma3} where the existence of such solutions is proved.

The following lemma is the analogue of Proposition 6.1 and Corollary 6.3 from \cite{Ma3}.

{ \begin{lemma}\label{lemma:solexist} There exists a solution $g_w \in \Gamma(E)' \otimes \C[\hbar^{-1}]$ of the system (\ref{E:qdiff}) of quantum differential equations of the form \[ g_w = \sigma_w + \sum_v g_{0,v}(z; \hbar^{-1}) \sigma_v + \sum_{v,d} g_{d,v}(z; \hbar^{-1}) q^d \sigma_v\] where $g_{d,v}(z; \hbar^{-1}) \in \C[z][\hbar^{-1}]$ are polynomials. Both sums are over $v \in W$, and the second sum is also over $d \in (\Z_{\ge 0})^{n+1}$ such that $d \neq 0$ and $d$ is not a multiple of $c=\alpha_0^\vee + \theta^\vee$; further, the constant term of the polynomial $g_{0,v}$ relative to the variables $z$ is equal to $0$.\end{lemma}}




\begin{proof} For each $1 \le i \le n$ we denote by $A_i$ the matrix of the ``quantum Chevalley" $\C[q]$-linear endomorphism $\overline{\Lambda}_i: \quantum^*_\af(G/B) \to \quantum^*_\af(G/B)$ sending $\sigma_w$ to $\sigma_i \star_\af \sigma_w$. Thus $A_i$ is a $|W| \times |W|$ matrix with coefficients in $\C[q]$. One can see from the proof of Theorem \ref{thm:flat} that the flatness of the Dubrovin connection is equivalent to the fact that the matrices $A_i$ mutually commute and that $\partial/\partial z_i (A_j) = \partial/\partial z_j (A_i)$ for any $i,j$. Further, up to some reordering of the Weyl group elements, and because of the affine quantum Chevalley formula, we can write $A_i$ as $A_i = A_i'(q^d) + A_i''$ where \begin{itemize} \item $A_i'(q^d)$ is strictly lower triangular, and its coefficients are linear combinations
of $q^d=e^{zd} := e^{d_0 z_0}\cdots e^{d_n z_n}$ where $d \ge 0$ and $\langle \lambda_i - m_i \lambda_0, d \rangle \neq 0$; \item $A_i''$ is strictly upper triangular and its coefficients do not depend on $q$ (in particular, the diagonal of $A_i''$ is identically zero). \end{itemize} 
We temporarily suspend the convention $q^c=q_0q_1^{m_1} \cdots q_n^{m_n} =1$ from (\ref{E:hyp}), thus $q_0=e^{z_0}$ is now independent from $q_1=e^{z_1}, \ldots ,q_n=e^{z_n}$. Since the Chevalley rule in $\quantum^*_\af(G/B)$ only involves powers $q^d$ where $|d| < 1 + m_1 + \cdots  + m_n$, this does not affect the coefficients of the matrices $A_i$. Further, consider the formal power series ring $\C[z_1, \ldots , z_n; \hbar, \hbar^{-1}][[e^{z_0}, \ldots , e^{z_n}]]$ and define a ``twisted derivative" operator $\partial /\partial {z_i}$ acting as in (\ref{E:partialzi}): for $g(z;\hbar) \in \C[z_1, \ldots , z_n; \hbar, \hbar^{-1}]$ set \begin{equation}\label{E:twistder} \frac{\partial}{\partial z_i}  \bigl(g(z;\hbar) e^{d_0 z_0}\cdots e^{d_n z_n} \bigr) = \Bigl( \frac{\partial g}{\partial z_i}(z;\hbar) + (d _i - m_i d_0) g(z;\hbar)\Bigr) e^{d_0 z_0}\cdots e^{d_n z_n} \/. \end{equation} The key fact is that this operator satisfies $\partial/ \partial z_i ( e^{z_0} e^{m_1 z_1} \cdots  e^{m_n z_n}) = 0$, therefore it also satisfies:
\begin{equation}\label{E:restr} \frac{\partial}{\partial z_i} \Bigl(g(z;\hbar)e^{zd}|_{e^{z_0} e^{m_1 z_1} \cdots e^{m_n z_n} =1}\Bigr) = \Bigl(\frac{\partial}{\partial z_i} (g(z;\hbar)e^{zd})\Bigr)|_{e^{z_0} e^{m_1 z_1} \cdots e^{m_n z_n} =1} \end{equation} as formal power series in $z_1, \ldots ,z_n$, where on the left hand side we considered the usual partial derivative with respect to $z_i$, and on the right the {twisted derivative}; the bars denote restrictions to the given relation. Consider now the $\C[z;\hbar^{-1}]$-module $(\mathbb{C}[z;\hbar^{-1}])^{|W|}$, and a function $G = \sum_{d \ge 0} G_d(z; \hbar^{-1}) e^{d_0 z_0}\cdots e^{d_n z_n} $, where $G_d(z; \hbar^{-1}) \in (\mathbb{C}[z;\hbar^{-1}])^{|W|}$ and $d=(d_0, \ldots, d_n) \in \Z^{n+1}_{\ge 0}$. Then $\partial / \partial z_i$ acts componentwise on such functions. Consider also the system of PDE's { \[ \frac{\partial}{\partial z_i }(G) = (\frac{1}{\hbar}A_i) G, \quad  1 \le i \le n \/. \]
(The convention $q^c =1$ is still suspended and $\frac{\partial}{\partial z_i}$ is the twisted derivative.)}
According to \cite[Proposition 6.1]{Ma3} this system has a solution $G = \sum_{d \ge 0} G_d(z; \hbar^{-1}) e^{zd}$ which is uniquely determined by the degree $0$ part $G_d^0= G_d^0(\hbar^{-1})$  in the variables $z$ of $G_d$ for those degrees $d$ satisfying $\langle \lambda_i - m_i \lambda_0, d \rangle =0$. Fix an identification of $\coh^*(G/B) \otimes \C[\hbar^{-1}]$ to $\C[\hbar^{-1}]^{|W|}$, as $\C[\hbar^{-1}]$-modules, and pick the vector $e_w \in \C[\hbar^{-1}]^{|W|}$ which corresponds to the Schubert class $\sigma_w$. Define $G_0^0 = e_w$ and $G_d^0 = 0$ for the degrees $d \neq 0$ mentioned before, and denote by $G_w$ the resulting solution of the PDE system. The identity (\ref{E:restr}) implies that the restriction of $G_w$ to $q^c =1$ satisfies the system of PDE's \[ \frac{\partial}{\partial z_i} (G_w | _{e^{z_0} e^{m_1 z_1} \cdots e^{m_n z_n} =1}) = (\frac{1}{\hbar} A_i) (G_w)|_{e^{z_0} e^{m_1 z_1} \cdots e^{m_n z_n} =1}; \quad 1 \le i \le n,  \/ \] where in the left hand side we have the usual partial derivative. Then the Lemma follows by taking $g_w$ to correspond to $G_w$ restricted to $q^c =1$, under the fixed identification. \end{proof}

Let now $P(e^z,\hbar \partial/\partial z, \hbar) \in \C[[e^z]][\hbar \partial/\partial z][\hbar]$ be a differential operator where \[ \hbar \partial/\partial z = (\hbar \partial/\partial z_1,\ldots, \hbar \partial/\partial z_n) \/.\] 
We will always regard $e^z$ to the left of $\partial/\partial z$. We impose a grading on these operators with respect to: $$\deg q_j = 2, ~(0\le j \le n); \quad \deg \hbar = 1;\quad \deg \partial/\partial z_i = 0,~
(1\le i \le n)\/.$$ It will be convenient to define an operator denoted $P_{\hat{\nabla}}$ which is obtained from $P$ after making substitutions $e^{z_i} =q_i$ for $0 \le i \le n$, and $\hbar \partial/ \partial z_i= \hat{\nabla}_{\partial/\partial z_i}$, for $1 \le i \le n$, where $\hat{\nabla}$ is the dual Givental connection.

The following result, observed for the ordinary quantum cohomology ring in \cite[Theorem 10.3.1]{cox.katz:mirror} (where it is attributed to B. Kim) will be repeatedly used:

{ \begin{lemma}\label{lemma:Paction} Let $P = P(e^z,\hbar \partial/\partial z, \hbar) \in \C[[e^z]][\partial/\partial z][\hbar]$ be a differential operator. Then \begin{equation}\label{E:Paction} P(e^z,\hbar \partial/\partial z, \hbar) \langle g_w, 1 \rangle = \langle g_w, P_{\hat{\nabla}}(1) \rangle \end{equation} where $1 = \sigma_e \in \coh^0(G/B)$, $\langle \cdot , \cdot \rangle$ denotes the Poincar{\'e} pairing, and $g_w$ is the flat section defined in Lemma \ref{lemma:solexist}. \end{lemma}}




\begin{proof} Let $F= \sum f_{d,w}(z; \hbar) q^d \sigma_w$ and $G = \sum g_{d,w} (z;\hbar) q^d \sigma_w$ where the sums are over $d= (d_0, \ldots, d_n) \in (\Z_{\ge 0})^{n+1}$ and $w \in W$, and where $f_{d,w}, g_{d,w} \in \C[z; \hbar, \hbar^{-1}]$. Then using the Frobenius property $\langle \sigma_u \star_{\af} \sigma_i, \sigma_v \rangle = \langle \sigma_u, \sigma_i \star_{\af} \sigma_v\rangle$ from Theorem \ref{thm:frobenius} we obtain \[ \hbar \partial/\partial z_i \langle F, G \rangle = \langle \nabla_{\partial/\partial z_i} F, G \rangle + \langle F, \hat{\nabla}_{\partial/\partial z_i} G \rangle \/. \] Take $F:=g_w$ and $G: =1$. The lemma follows because $\nabla_{\partial/\partial z_i} g_w = 0$ for all $i$.\end{proof}

The following analogue of Givental's result \cite[Corollary 6.4]{givental:egw} (see also \cite[Lemma 6.4]{Ma3}) shows how this formalism leads to relations in $\quantum^*_\af(G/B)$. 
\begin{prop}\label{qcohrels} Let $P=P(e^z,\hbar \partial/\partial z, \hbar) \in \C[[e^z]][\hbar \partial/\partial z][\hbar]$ be a differential operator such that $\deg P < \deg q_0 q^{\theta^\vee} = 2(1+ m_1 + \cdots  + m_n)$ and which satisfies $P\langle g_w, 1 \rangle_d = 0$ for all $w \in W$ and
all $d=(d_0, \ldots, d_n)$ with $d_0 + \cdots  + d_n \le m_1 + \cdots  +m_n$
(the subscript $d$ indicates the coefficient of $e^{zd}$). Then the specialization $P(q,\sigma_i \star_{\af}, 0) = 0$ in the quantum cohomology ring $\quantum^*_{\af}(G/B)$.\end{prop}
 
\begin{proof} We adapt the proof of \cite[Theorem 10.3.1]{cox.katz:mirror}, using ideas from \cite[Lemma 6.4]{Ma3}. By Lemma \ref{lemma:Paction} we obtain that
\[ 0 = P\langle g_w, 1 \rangle = \langle g_w, P_{\hat{\nabla}}(1) \rangle \/. \] Write $P_{\hat{\nabla}}(1) = P^{(0)} + \cdots  + P^{(k)}$ where $P^{(i)}$ contains the terms which are multiples of $q_0^{d_0} \cdots  q_n^{d_n}$ with $d_0 + \cdots  + d_n = i$ (for now we regard $\hbar$ and the Schubert classes as parameters). Notice that $k\le m_1+\cdots  +m_n$. Let $g_w^0 \in \coh^*(G/B) \otimes \C[z; \hbar^{-1}]$ be the component of $g_w$ which does not depend on $q$. Since the term independent of the variables $z$ of $g_w^0$ equals $\sigma_w$ (by Lemma \ref{lemma:solexist}), it follows that the elements $g_w^0$ are a basis of $\coh^*(G/B; \C[z; \hbar^{-1}])$. Notice now that { the part of degree $0$ with respect to $q$ of $\langle g_w, P_{\hat{\nabla}}(1) \rangle$ equals $\langle g_w^0, P^{(0)} \rangle $}. Since this is true for all $w \in W$, we deduce that $P^{(0)} = 0$. Using this, we obtain that the { part of total degree $1$ with respect to $q$} of $\langle g_w, P_{\hat{\nabla}}(1) \rangle$ equals $\langle g_w^0, P^{(1)} \rangle $, thus $P^{(1)} = 0$. Continuing this process we obtain that $P^{(i)} = 0$ for all $0 \le i \le k$, thus $P_{\hat{\nabla}}(1) = 0$. Now notice that \[ \hat{\nabla}_{\partial / \partial_{z_{i_1}}} \ldots \hat{\nabla}_{\partial / \partial_{z_{i_k}}} (1) = \sigma_{i_1} \star_{\af} \ldots \star_{\af} \sigma_{i_k} + \hbar P_1 \] where $P_1$ depends on $q$, Schubert classes $\sigma_w$ and $\hbar$. Then we can write \[ P_{\hat{\nabla}}(1) = P(q, \sigma_i \star_{\af}, 0 ) + \hbar P_2 (q, \sigma_i \star_\af, \hbar) \] where $P_2$ contains { nonnegative} 
powers of $\hbar$. This is an identity in $\hbar$, therefore we can make $\hbar = 0$ to deduce that $P(q, \sigma_i \star_{\af}, 0 ) = 0$ as claimed.\end{proof}

An operator $P=P(e^z,\hbar \partial/\partial z, \hbar)$ satisfying $P \langle g_w , 1 \rangle = 0$ for all $w \in W$ is called a {\em quantum differential operator}. When studying the ordinary quantum cohomology ring $\quantum^*(G/B)$, such operators annihilate the J-function of $G/B$. A first example of a quantum differential operator is the Hamiltonian of the dual version of the quantum periodic Toda lattice:
 \begin{equation}\label{def:qTodaH} \mathcal{H} := \sum_{i, j =1}^n (\alpha_i^\vee | \alpha_j^\vee) \frac{\hbar \partial}{\partial z_i}\frac{\hbar \partial}{ \partial z_j }- (\theta^\vee | \theta^\vee ) e^{z_0} - \sum_{i=1}^n  ( \alpha_i^\vee | \alpha_i^\vee ) e^{z_i} \/, \end{equation} where $( \cdot | \cdot )$ is the Killing form on $\h$ from \S \ref{prel} above. (We shall discuss more about this Hamiltonian in the next section.) { Indeed,} 
by Lemma \ref{lemma:Paction} $\mathcal{H}\langle g_w, 1 \rangle = \langle g_w, \mathcal{H}_{\hat{\nabla}}(1) \rangle$, so it suffices to show that $\mathcal{H}_{\hat{\nabla}}(1) = 0$. This is equivalent to showing that \[ \sum_{i, j =1}^n (\alpha_i^\vee | \alpha_j^\vee) \sigma_i \star_\af \sigma_j   = (\theta^\vee | \theta^\vee ) q_0 + \sum_{i=1}^n  ( \alpha_i^\vee | \alpha_i^\vee ) q_i \/. \] The last identity holds because $ \sum_{i, j =1}^n (\alpha_i^\vee | \alpha_j^\vee) \sigma_i \cdot \sigma_j   = 0$ (as this is a non-constant element of ${\rm Sym} (\h^*_\Q)^W$ in the Borel presentation of $\coh^*(G/B)$), and using the formula (\ref{eq:divmult}) for $\sigma_i \star_\af \sigma_j$. 
 
Next is the key technical result which will determine the ideal of relations of $\quantum^*_\af(G/B)$. For the ordinary quantum cohomology this was proved by B.~Kim \cite[Lemma 1]{Kim}, and in a more general context by the first author in \cite[Theorem 5.1]{Ma3}.

\begin{thm}\label{thm:higherrels} Let $\mathcal{D}= \mathcal{D}(e^z, \hbar \partial / \partial z, \hbar)$ be an operator such that:

\begin{enumerate} \item $\mathcal{D}_{\hat{\nabla}}(1) \equiv 0 $ modulo $q_0, q_1, \ldots, q_n$, i.e. $\mathcal{D}$ is a deformation of an ordinary Borel relation in ${\rm Sym}(\mathfrak{h}^*)^W$ for $\coh^*(G/B)$;
\item $\mathcal{D}$ has degree at most $\deg q^{\theta^\vee} = 2 (m_1 + \cdots  + m_n)$;
\item $\mathcal{D}$ commutes with the Hamiltonian $\mathcal{H}$ from (\ref{def:qTodaH}), i.e. $[\mathcal{H},\mathcal{D}]=0$.
\end{enumerate}
Then for all $w \in W$ and all $d= (d_0, \ldots,d_n) \in (\Z_{\ge 0})^{n+1}$ with $d_0 + \cdots  + d_n \le m_1 + \cdots  + m_n$,
the identity $\mathcal{D}\langle g_w, 1 \rangle_d = 0$ holds. In particular, $\mathcal{D}(q, \sigma_i \star_{\af}, 0) = 0$ in $\quantum^*_{\af}(G/B)$. 
\end{thm}
The proof requires the following lemma proved in \cite[Lemma 6.5]{Ma3}:

\begin{lemma}\label{lemma:deg} Let $g = g_0 + \sum_{d > 0} g_d(z; \hbar,\hbar^{-1}) e^{zd} \in \C[z; \hbar^{-1},\hbar][[e^z]]$ be a formal power series which satisfies $g_0 = 0$ and $\mathcal{H}(g) = 0$. Then $g_d = 0$ for any degree $d= (d_0, \ldots,d_n) \in (\Z_{\ge 0})^{n+1}$ such that $d_0 + \cdots  + d_n \le m_1 + \cdots  + m_n$.\end{lemma}

\begin{proof}[Proof of Theorem \ref{thm:higherrels}] By Lemma \ref{lemma:Paction},
\[ 0 = \mathcal{D} \langle g_w, \mathcal{H}_{\hat{\nabla}}(1) \rangle = \mathcal{D}\mathcal{H} \langle g_w, 1 \rangle = \mathcal{H}\mathcal{D} \langle g_w, 1 \rangle = \mathcal{H} \langle g_w, \mathcal{D}_{\hat{\nabla}}(1) \rangle \/. \] Let  $g:=\langle g_w, \mathcal{D}_{\hat{\nabla}}(1) \rangle \in \C[z; \hbar,\hbar^{-1}][[e^z]]$. 
The hypothesis (1) implies that $g_0 = 0$, and Lemma \ref{lemma:deg} that $g_d = 0$ for $d_0 + \cdots  + d_n \le m_1 + \cdots  + m_n$. Since $g= \mathcal{D} \langle g_w, 1 \rangle$, we can use Proposition \ref{qcohrels} and get to the desired conclusion.\end{proof}

\section{The periodic Toda lattice and relations in $\quantum^*_\af(G/B)$}\label{s:Toda} The strategy to obtain the ideal of relations for $\quantum^*_\af(G/B)$ is to produce some differential operators $\mathcal{H}_i$ which deform the Borel relations in $\coh^*(G/B)$ and commute with the Hamiltonian operator $\mathcal{H}$ from (\ref{def:qTodaH}). Then by Theorem \ref{thm:higherrels} each $\mathcal{H}_i$ gives a relation, and it is easy to show that these generate the ideal of relations. It turns out that the differential operators we need are the integrals of motion for the integrable system called the {\em quantum periodic Toda lattice} associated to the {\em dual} of an extended Dynkin diagram. Such operators were constructed by Etingof \cite{etingof} in the non-dual case and all Lie types, and in the dual case by Goodman and Wallach \cite{Go-Wa1} for $G$ of Lie types $A_n - D_n$ and $E_6$. { In \cite{M}, Mare used Dynkin automorphisms and the results from \cite{Go-Wa1} to construct such operators in the remaining Lie types $F_4$ and $G_2$.}
This { is consistent to} 
the philosophy of B.~Kim \cite{Kim} that the relations in quantum cohomology of $G/B$ are given by integrals of motion for the Toda lattice of the Langlands {\em dual} root system.  In what follows we recall the relevant facts about the (quantum) periodic Toda lattice and its integrability.

%

\subsection{The dual, quantum, periodic Toda lattice}\label{s:qToda}

  Recall that $\g$ denotes a complex simple Lie algebra, with the simple coroot system ${\Delta}^\vee \subset {\Pi}^\vee$. 
The {\it   periodic quantum Toda lattice} associated to $\Delta^{\vee}\cup\{-\theta^\vee\}$ is determined by the differential
operator ${\mathcal H}$ given by equation (\ref{def:qTodaH}). The integrals of motion are differential operators
${\mathcal D}\in \bC [e^z, \hbar\partial/\partial z, \hbar]$ that commute with ${\mathcal H}$, i.e.~$[{\mathcal H},{\mathcal D}]=0$.
In the case one can find $n = rank(\Delta)$ integrals that are independent and are in involution then the
system is completely integrable.

\begin{remark}\label{rem:two} In the literature there are two integrable systems associated to $\g$ with the name of (classical) periodic Toda lattice. One is associated to $\Delta \cup \{ - \theta \}$, the other is the ``dual" one  above, associated to $\Delta^\vee \cup \{ -\theta^\vee \}${, which corresponds to a twisted affine Lie algebra}. 
The first system was classically studied by Kostant \cite{Kos}. { Adler and van Moerbeke \cite{Ad-vM,adler.vanmoerbeke:abelian} studied both systems in the context of geometry of Abelian varieties.}


\end{remark}

{ We recall the main ideas behind the construction of integrals of motion for the dual quantum periodic Toda lattice. We follow the method of Goodman and Wallach \cite{Go-Wa1}, based on the notion of the $(ax+b)$-(Lie) algebra associated to the 
 extended simple root system $\Delta^\vee \cup \{-\theta^\vee\}$. For finite Lie types, this is the Lie algebra $\mfb^\vee/ [\mfu^\vee , \mfu^\vee]$ where $\mfb$ is the Borel subalgebra for $\Delta$, and $\mfu$ the nilpotent algebra. It played a key role in Kim's paper \cite{Kim}.}
Consider a complex vector space $\mfu$ of dimension $n+1$ along with a basis $\{X_{\alpha} \mid 
\alpha \in \{-\theta, \alpha_1,\ldots, \alpha_n\}\}$. Put  $\mfb:=\hc^*\oplus \mfu$, and 
  define on this space the Lie bracket $[ \ , \ ]$ as follows:
  \begin{itemize}
  \item[(i)] $[ \ , \ ]$ is identically zero on both $\hc^*$ and $\mfu$
  \item[(ii)] for any $\lambda \in \hc^*$ and any  $\alpha\in
\{-\theta, \alpha_1,\ldots, \alpha_n\}$ one has
  $[\lambda, X_\alpha]=\langle \lambda, \alpha^\vee \rangle X_\alpha$.
  \end{itemize}
Let $\omega_1, \ldots, \omega_n\in \hc^*$ be  the fundamental weights, i.e.~$\langle \omega_i, \alpha_j^\vee\rangle = \delta_{ij}$ (the Kronecker symbol), for all $1\le i,j\le n$.
The corresponding PBW basis of the universal enveloping algebra 
$U(\mfb)$ consists of  $\omega_{i_1}^{a_{i_1}} \cdots \omega_{i_k}^{a_{i_k}} X_{\beta_{j_1}}^{b_{j_1}}\cdots X_{\beta_{j_r}}^{b_{j_r}}$, where $a_{i_k}, b_{j_\ell} \in \Z_{\ge 0}$ and $\beta_{j_\ell} \in \{ - \theta, \alpha_1, \ldots , \alpha_n \}$. 
A key role will be played by the subspace $U(\mfb)_{ev} \subset U(\mfb)$ which is spanned by the elements as before but where all powers $b_{j_{\ell}}$ are even. Consider the  canonical filtration  
$U_1(\mfb)\subset U_2(\mfb) \subset \ldots \subset U(\mfb)$.
We say that  $P\in U(\mfb)$ has degree $k$ if $P\in U_k(\mfb)\setminus U_{k-1}(\mfb)$. 
The {\em Laplacian} of $\mfb$ is the element \[ \Omega:= \sum_{i,j=1}^n (\alpha_i^\vee | \alpha_j^\vee ) 
{\omega}_i{\omega}_j + \sum_{\alpha} X_\alpha^2 \quad \in U(\mfb) \/. \] 
The projection $\mfb \to \h^*$ induces a degree preserving algebra homomorphism $\mu: U(\mfb) \to U(\h^*) = {\rm Sym}(\h^*)$ called the {\em symbol homomorphism}. Choose fundamental homogeneous generators $f_1, \ldots, f_n$ of
${\rm Sym}({\mathfrak{h}}^*)^{{W}}$; by convention  
 take $f_1:=\sum_{i,j=1}^n ( \alpha_i^\vee | \alpha_j^\vee ) {\omega}_i{\omega}_j$ and assume that $\deg f_i \le \deg f_j$ whenever
 $i<j$. The sequence of degrees of $f_i$, together with the quantity $m_1 + \ldots +m_n = \frac{1}{2} \deg q^{\theta^\vee}$ are recorded in the table below (cf.~\cite[\S 3.7]{Hu2}).

 \begin{tabular}{|c|c|c|}
\hline
Type & $\deg f_1, \ldots , \deg f_n$ & $\sum m_i$\\
\hline
$A_n$ & $2,3, \ldots, n+1$ & $n$\\
$B_n$ & $2,4,6, \ldots, 2n$ & $2n-2$\\
$C_n$ & $2,4,6,\ldots, 2n$ & $n$\\
$D_n$ & $2,4,6,\ldots,2n-2,n$ & $2n-3$\\
$E_6$ & $2,5,6,8,9,12$ & $11$\\
$E_7$ & $2,6,8,10,12,14,18$ & $17$\\
$E_8$ & $2,8,12,14,18,20,24,30$ & $29$\\
$F_4$ & $2,6,8,12$ & $8$\\
$G_2$ & $2,6$  & $3$\\
\hline
\end{tabular}

%
%

%

\begin{thm}\label{thm:Todaint}{\rm (\cite{Go-Wa1}, \cite{etingof}, \cite{M})} Fix $i\in \{1,\ldots, n \}$. There exists a unique 
$\Omega_i \in U(\mfb)_{ev}$ such that:

\begin{itemize} 
\item $[\Omega_i, \Omega]=0$;

\item $\Omega_i$ has degree equal to $\deg f_i$;

\item the symbol $\mu(\Omega_i) = f_i $ is the chosen  generator of ${\rm Sym}(\h^*)^W$.
\end{itemize} 

Further, for $\Omega_1:=\Omega$ and for any $1 \le i, j \le n$, $[\Omega_i, \Omega_j]=0$. 
\end{thm}

This result was proved by Goodman and Wallach \cite{Go-Wa1} for types $A_n - D_n$ and $E_6$, and by { Mare \cite{M}} in the types 
$F_4$ and $G_2$. Etingof \cite{etingof} proved earlier the integrability of the (non-dual) quantum periodic Toda lattice, for all Lie types, and this implies the above result for the simply laced Lie algebras. We refer to \cite{M} for details.

\subsection{Quantization of $U(\mfb)$ and commuting differential operators} 
We will quantize the operators $\Omega_i$ from Theorem \ref{thm:Todaint}, by adapting the procedure outlined e.g.~in \cite{Kim}. We recall from \S \ref{s:givform} above the ring $\C[[e^z]][\partial/\partial z][\hbar]$ consisting of operators acting on $\C[z][[e^z]][\hbar, \hbar^{-1}]$.
Define the linear map $\rho:\mfb \to \C[[e^z]][\partial/\partial z][\hbar, \hbar^{-1}]$ given by 
\[ \rho ({\omega}_i):=2\frac{\partial}{\partial z_i}, ~ \rho (X_{\alpha_i})=\frac{2
\sqrt{-1}}{\hbar}
 \sqrt{ (\alpha_i^{\vee} | \alpha _i^{\vee})}
e^{\frac{z_i}{2}}, (1 \le i \le n),~  \rho (X_{-\theta})=\frac{2
\sqrt{-1}}{\hbar}
 \sqrt{(\theta^{\vee} | \theta^{\vee})}
e^{\frac{z_0}{2}} \/. \]
Then one can check that $\rho$ preserves the Lie bracket, thus it is a morphism of Lie algebras. For example, if $g(z; \hbar) \in \C[z; \hbar, \hbar^{-1}]$ then \[ \begin{split} \rho(\omega_i) \rho(\omega_j) (g(z;\hbar) e^{zd}) =  2 e^{zd} \Bigl( \frac{\partial^2 g(z;\hbar)}{\partial z_i \partial z_j} +  (d_i - m_i d_0) \frac{\partial g(z;\hbar)}{\partial z_j} + & \\  (d_j - m_j d_0) \frac{\partial g(z;\hbar)}{\partial z_i} + (d_i - m_i d_0)(d_j - m_j d_0) g(z;\hbar)\Bigr) \/, \end{split} \] and this shows that $\rho([\omega_i, \omega_j]) = [\rho(\omega_i), \rho(\omega_j)]=0$. Therefore there exists a 
well-defined Lie algebra homomorphism $\rho: U(\mfb) \to \C[[e^z]][\partial/\partial z][\hbar, \hbar^{-1}]$. Define \begin{equation}\label{E:defhigherH} \mathcal{H}_i := \hbar^{\deg \Omega_i} \rho (\Omega_i), \quad 1 \le i \le n  \/. \end{equation}
In particular, $\mathcal{H}_1 = \hbar^2 \rho(\Omega) = 4 \mathcal{H}$ where $\mathcal{H}$ is the operator defined in (\ref{def:qTodaH}). 

\begin{cor}\label{cor:commutingqH} Let $1 \le i \le n$. Then the operator $\mathcal{H}_i$ has the following properties:
\begin{itemize} \item $\mathcal{H}_i$ belongs to $\C[e^z][\hbar \partial /\partial z][\hbar]$;

\item 
$[\mathcal{H}_i, \mathcal{H}]=0$;

\item $\mathcal{H}_i$ has degree $\deg \mathcal{H}_i \le \deg q^{\theta^\vee}$.

\end{itemize}

In particular, the specialization $\mathcal{H}_i(q; \sigma_k \star_\af; 0) = 0$ in $\quantum^*_\af(G/B)$.\end{cor}

\begin{proof} The fact that the operators are in the subring $\C[e^z][\hbar \partial /\partial z][\hbar]$ follows because by construction the integrals of motion $\Omega_i$ are in the subspace $U(\mfb)_{ev}$.~Using the PBW basis for $U(\mfb)_{ev}$, notice that any element in $\rho(U(\mfb)_{ev})$ has degree $0$, where (recall) $\deg \hbar =1$, $\deg \partial/\partial z_i = 0$ and $\deg e^{z_i} = 2$. Thus $\deg \mathcal{H}_i = \deg \Omega_i = \deg f_i$. Now we use again the table with $\deg f_i$ from above to notice that $\max \{ \deg f_i \} \le \deg q^{\theta^\vee}$. The commutation is immediate from the fact that $\rho$ preserves Lie brackets. Finally, the relation in $\quantum^*_\af(G/B)$ is now immediate from Theorem \ref{thm:higherrels}.\end{proof}

\subsection{The dual, classical, periodic Toda lattice. Relations in $\quantum^*_\af(G/B)$} \label{sec:clToda}
Corollary \ref{cor:commutingqH} shows that the quantum dual periodic Toda lattice is 
integrable. In fact, Theorem \ref{thm:Todaint} can also be used to prove the
integrability at the classical level of that system. To describe this  Hamiltonian system, let $\h_\bR$ be the real span $\Span_{\bR} \Delta^\vee$  and $\h_\bR \times \h_\bR^*$ its cotangent 
bundle. 
The Hamiltonian function is $H:\hc_\bR \times\hc_\bR^*\to \bR$ defined by
 \begin{equation}\label{E:defTodaH} H(r,s)=| r |^2 +e^{-2\langle s, \theta^\vee\rangle }+\sum_{i=1}^ne^{2\langle s, \alpha_i^\vee\rangle},\quad r\in \hc_\bR, s\in \hc_\bR^*\/.\end{equation}
%

The norm involved in the equation above  is the one induced by the Killing form.
Conserved quantities of this system are real functions on $\hc_\bR \times \hc_\bR^*$ that
commute with $H$ with respect to the Poisson bracket. To construct such functions, 
consider again the PBW basis $\{ \omega_1^{a_1} \cdots  \omega_n^{a_n} X_{\alpha_1}^{2 b_1}\cdots  X_{\alpha_n}^{2 b_n} X_{-\theta}^{2 b_0} \}$ of $U(\mfb)_{ev}$, where $a_i, b_j \in \Z_{\ge 0}$. With respect to this basis we can write each $\Omega_i$ ($1 \le i \le n$) as \[ \Omega_i = H_i (\omega_1,\ldots, \omega_n, X_{\alpha_1}^2, \ldots, X_{\alpha_n}^2, X_{-\theta}^2) + H_i' (\omega_1,\ldots, \omega_n, X_{\alpha_1}^2, \ldots, X_{\alpha_n}^2, X_{-\theta}^2) \] where $H_i$ is homogeneous of degree $\deg \Omega_i$ and $\deg H_i' < \deg H_i$. 
After the change of variables 
\begin{equation}\label{changex}X_{\alpha_i} = e^{\langle s, \alpha_i^\vee \rangle}; \quad X_{-\theta} = e^{- \langle s, \theta^\vee \rangle }\end{equation} the homogeneous polynomials $H_i$ are the aforementioned conserved quantities. 

To obtain relations in $\QH^*_\af(G/B)$ we need to consider the substitutions \begin{equation}\label{E:cov} \omega_i = x_i; \quad X_{\alpha_i}^2 = - (\alpha_i^\vee | \alpha_i^\vee) q_i; \quad X_{-\theta}^2 = - (\theta^\vee | \theta^\vee) q_0 \/. \end{equation} We abuse notation and we keep the notation $H_i=H_i(q;x)$ for the resulting polynomial in $\C[q; x]$, where $q=(q_0, \ldots, q_n)$ and $x=(x_1, \ldots ,x_n)$. By Theorem \ref{thm:defafqcoh} there is a surjective $\C[q]$-algebra homomorphism $\Phi: \C[q;x] \to \quantum^*_\af(G/B)$ obtained by sending $x_i$ to $\sigma_i$.
The main result of this section identifies the kernel of this homomorphism. It is the affine analog of B.~Kim's theorem
\cite[Theorem~I]{Kim}. 

\begin{thm}\label{thm:relations} The kernel of $\Phi$ is the ideal generated by the conserved quantities $H_i(q;x)$ ($1 \le i \le n$) of the dual periodic Toda lattice. Equivalently, there is a $\C[q]$-algebra isomorphism  \[ \C[q;x]/\langle H_1(q;x), \ldots , H_n(q;x) \rangle \to \quantum^*_\af(G/B) \] sending $x_i$ to $\sigma_i$. 
\end{thm}

\begin{proof} From the definitions of $\mathcal{H}_i$ and $H_i$ one can easily see that $\Phi(H_i(q;x))$ equals the specialization $\mathcal{H}_i(q; \sigma_k \star_\af; 0)$. This and Corollary \ref{cor:commutingqH} imply that each such $H_i$ is in the kernel of $\Phi$. 
The rest of the theorem follows from a standard Nakayama-type argument, going back to Siebert-Tian \cite{st}. We briefly remind its main idea. First, notice that the specializations $H_i(0; x)$ at $q_k=0$ equal the polynomials $f_i$ which are the homogeneous generators of the ideal of relations in the Borel presentation of the graded $\C$-algebra $\coh^*(G/B)= \quantum^*_\af(G/B)/ \langle q_0, \ldots, q_n \rangle$. Second, the images of $H_i(q;x)$ give (homogeneous) relations in $\quantum^*_\af(G/B)$. These two facts, together with the fact that $\quantum^*_\af(G/B)$ is graded, imply that the polynomials $H_i(q;x)$ generate the ideal of relations. We refer to \cite[Proposition 11]{fulton.pandharipande:notes} for full details about this argument.  
\end{proof}


\section{Examples}

In this section we illustrate Theorem \ref{thm:relations} by two examples, and we also calculate the multiplication in $\quantum^*_\af(\SL_3(\C)/B)$. 


\begin{example}\label{forma}Let $G= \SL_n(\C)$, thus $G/B = \Fl(n)$, the variety which parametrizes complete flags $F_1 \subset \ldots \subset F_n:= \C^n$. In this case, $\hc_\bR$ { can be identified with the hyperplane in
 $\bR^n$ consisting of all $r=(r_1,\ldots, r_n)$ with $r_1+\cdots  +r_n=0$.  In this identification the coroots are given by} $$\alpha_i^\vee = e_i- e_{i+1}, 1 \le i \le n-1, \quad \theta^\vee = \alpha_1^\vee + \cdots  + \alpha_{n-1}^\vee=e_1-e_n,$$  
 where $e_1,\ldots, e_n$ is the canonical basis of $\bR^n$.  
  Let  $s=(s_1,\ldots, s_n)$ be the coordinates on $\h^*_\bR$ induced from $\bR^n$. 
The  Hamiltonian function
of the periodic Toda lattice is in this case
 $$H(r,s) = \sum_{i=1}^nr_i^2 + \sum_{i=1}^n e^{2(s_i-s_{i+1})}, \ {\rm where} \  s_{n+1}:=s_1.$$ 
To describe the integrals of motion, we follow \cite[pp.~292-294]{Ad-vM} (see also \cite[pp.~81-82]{Au}).
Consider the matrix
$$A:= \left (
\begin{matrix}
-\sqrt{2}r_1 &  e^{2(s_1-s_2)}  &  0  & 0&  \ldots &0& \frac{1}{z} \\
1 &  -\sqrt{2}r_2 &  e^{2(s_2-s_3)}  & 0 & \ldots & 0&0 \\
0 & 1 & -\sqrt{2}r_3 & e^{2(s_3-s_4)} &  \ldots & 0 & 0\\
\vdots  &  \ddots  &  \ddots &\ddots  & \ddots &\ddots &\vdots \\
0  &  0  &  0  & 0  & \ldots &-\sqrt{2}r_{n-1} & e^{2(s_{n-1}-s_n)}  \\
e^{2(s_n-s_1)} z  &  0  &  0  & 0 &  \ldots &1 & -\sqrt{2}r_n \\
\end{matrix}
\right ),$$
where $z$ is a complex parameter.
Its characteristic polynomial is
$$\det (A + \lambda I_n) = \sum_{k=1}^{n-1}H_{k}\lambda^{n-k-1} + \lambda^n+ (-1)^{n-1}\left({z} +\frac{1}{z}\right),$$
where $H_k$ are independent of $z$ (note that $\lambda^{n-1}$ does not occur in the expansion above,
since its coefficient  is $-\sqrt{2}(r_1+ \ldots + r_n) =0$). 
For example,  the coefficient of $\lambda^{n-2}$ is
$$H_1= 2\sum_{1\le i < j \le n} r_i r_j -\sum_{i=1}^{n}e^{2(s_{i}-s_{i+1})}=
-\sum_{i=1}^{n}r_i^2 -\sum_{i=1}^n e^{2(s_{i}-s_{i+1})},$$
which equals $-H$. 
The functions $H_1, \ldots, H_{n-1}$ are the desired integrals of motion.
Note that each $H_k$ is a polynomial in $r_i$ and $e^{2(s_i-s_{i+1})}$, 
homogeneous of degree $2k$ relative to $\deg r_i=\deg e^{s_i-s_{i+1}} =1$. 
To obtain the relations in the quantum cohomology ring from Theorem \ref{thm:relations} we take into account the formal changes
of variables (\ref{changex}) and (\ref{E:cov}). In this case they give
$$x_i = \omega_i; \quad e^{2(s_i - s_{i+1})} = -2q_i; \quad e^{2(s_n-s_1)} = -2q_0.$$
{ With our identifications, $\langle \omega_i, r\rangle  = r_1 + \cdots + r_i$, thus the coordinates $r=(r_1,\ldots, r_n)\in \h_\bR$ are given by:}
 $$r_1=\langle \omega_1,r\rangle, r_2=\langle \omega_2-\omega_1,r\rangle, \ldots, r_{n-1}= \langle\omega_{n-1}-\omega_{n-2},r\rangle,
 r_{n}= -\langle \omega_{n-1},r\rangle \/. $$

By Theorem \ref{thm:relations}, $\quantum_\af^*({\rm Fl}(n))$ is isomorphic to
$\bC[q_0, \ldots, q_{n-1}; x_1, \ldots, x_{n-1}]$ modulo the ideal generated by $H_1(q;x), \ldots, H_{n-1}(q;x)$, which arise from the matrix
$$A(q;x):= \left (
\begin{matrix}
x_1 &  q_1  &  0  & 0&  \ldots &0& -\frac{1}{z} \\
-1 &  x_2-x_1 &  q_2  & 0 & \ldots & 0&0 \\
0 & -1 & x_3-x_2 & q_3 &  \ldots & 0 & 0\\
\vdots  &  \ddots  &  \ddots &\ddots  & \ddots &\ddots &\vdots \\
0  &  0  &  0  & 0  & \ldots &x_{n-1}-x_{n-2} & q_{n-1}  \\
q_0 z  &  0  &  0  & 0 &  \ldots &-1 & -x_{n-1} \\
\end{matrix}
\right )$$
via
$$\det \left(A(q;x) + \lambda I_n\right) = \sum_{k=1}^{n-1}H_{k}(q;x)\lambda^{n-k-1} + \lambda^n+  (-1)^{n-1}q_0 \cdots  q_n z -  \frac{1}{z} \/.$$
{ Geometrically, $x_i$ is the first Chern class of the dual of the $i$-th tautological line bundle on $\Fl(n)$ which has fibre over $F_1 \subset \ldots \subset F_n$ equal to $(F_i/ F_{i-1})^*$.} This result is in the same spirit as the main result of \cite{go}. 


\end{example}

{ \begin{example}\label{typb2} Consider now $G= \textrm{Spin}_5(\C)$ which is a group of type $B_2$. The flag variety parametrizes flags $F_1 \subset F_2 \subset \C^5$ which are isotropic with respect to a nondegenerate symmetric bilinear form. We identify both $\hc_\bR$ and $\hc_\bR^*$ with $\bR^2$. A simple root system along with the highest root is
$\alpha_1=e_1, \alpha_2=e_2-e_1, \theta=e_1+e_2$, and their duals are
$\alpha_1^\vee = 2e_1, \alpha_2^\vee = e_2-e_1, \theta^\vee = e_1+e_2=\alpha_1^\vee+\alpha_2^\vee$. The Weyl group invariants are $f_1 = y_1^2 + y_2^2$ and $f_2 = y_1^2 \cdot y_2^2$ where $y_1, y_2$ are the coordinates in $\h^*$. In terms of fundamental weights we have $y_1 = 2 \omega_1 - \omega_2$ and $y_2 = \omega_2$. 

The quantum Hamiltonian from (\ref{def:qTodaH}) is \[ \mathcal{H}_2 = 4 \frac{\hbar \partial}{\partial z_1} \frac{\hbar \partial}{\partial z_1} - 4 \frac{\hbar \partial}{\partial z_1} \frac{\hbar \partial}{\partial z_2} + 2 \frac{\hbar \partial}{\partial z_2} \frac{\hbar \partial}{\partial z_2} - 2 e^{z_0} - 4 e^{z_1} - 2 e^{z_2} \/ \] and it corresponds to the degree $2$ relation in $\quantum_\af^*(G/B)$ \[ R_1:= 4\sigma_1 \star_\af \sigma_1 - 4 \sigma_1 \star_\af \sigma_2 + 2 \sigma_2 \star_\af \sigma_2 - 2 q_0 - 4 q_1 -  2q_2 \/. \] Using the method of $ax+b$ algebras from \S \ref{s:qToda} one obtains the quantum Hamiltonian of degree $4$ \[\begin{split} \mathcal{H}_4= 4 \frac{\hbar^2 \partial^2}{\partial z_1^2} \frac{\hbar^2 \partial^2}{\partial z_2^2} - 4 \frac{\hbar \partial}{\partial z_1} \frac{\hbar^3 \partial^3}{\partial z_2^3} - 4 e^{z_0} \frac{\hbar \partial}{\partial z_1} \frac{\hbar \partial}{\partial z_2}+ 4 e^{z_2} \frac{\hbar \partial}{\partial z_1} \frac{\hbar \partial}{\partial z_2}+ \frac{\hbar^4 \partial^4}{\partial z_2^4} + 2 e^{z_0} \frac{\hbar^2 \partial^2}{\partial z_2^2} - 4 e^{z_1}  \frac{\hbar^2 \partial^2}{\partial z_2^2} \\ - 2 e^{z_2} \frac{\hbar^2 \partial^2}{\partial z_2^2}+  e^{2 z_0} - 2 e^{z_0 + z_2} + e ^{2 z_2} - 4 e^{z_0} \frac{\hbar \partial}{\partial z_1} - 4 e^{z_2} \frac{\hbar \partial}{\partial z_1}  + 4 e^{z_2} \frac{\hbar \partial}{\partial z_2} - 4 e^{z_0} - 4 e^{z_2} \/. \end{split} \] After taking the homogeneous part of degree $4$ and making substitutions $q_i = e^{z_i}$ and $\sigma_i = \frac{\hbar \partial}{\partial z_i}$ one obtains the second relation in the quantum cohomology ring \[ \begin{split} R_2= 4 \sigma_1^2 \star_\af \sigma_2^2 - 4 \sigma_1 \star_\af \sigma_2^3 - 4q_0 \sigma_1 \star_\af \sigma_2 + 4 q_2  \sigma_1 \star_\af \sigma_2 + \sigma_2^4 + 2 q_0 \sigma_2^2 - 4 q_1 \sigma_2^2 \\ - 2 q_2 \sigma_2^2 + q_0^2 - 2 q_0 q_2 + q_2^2 \/,  \end{split} \] where by $\sigma_i^k$ we mean multiplication with respect to $\star_\af$. In other words, $\quantum^*_\af(G/B)$ is the ring $\C[q_0,q_1, q_2; x_1,x_2]/\langle H_1, H_2 \rangle$ where \[ H_1 = 4x_1^2 -  4 x_1 x_2 + 2 x_2^2 -  2 q_0 -  4 q_1 -  2q_2 \]\[ H_2 = 4 x_1^2 x_2^2 -  4 x_1 x_2^3 -  4 q_0 x_1 x_2 + 4 q_2 x_1 x_2 + x_2^4 + 2 q_0 x_2^2 -  4 q_1 x_2^2 -  2 q_2 x_2^2 + q_0^2 - 2 q_0 q_2 + q_2^2 \/, \] and $x_i$ is sent to the Schubert class $\sigma_i$, for $i = 1,2$. 
\end{example}
}

\subsection{The multiplication table in $\quantum^*_\af(\Fl(3))$}\label{ss:table}
 Using the Chevalley formula (\ref{eq:Lambdafin}), and setting $w_0 = s_1 s_2 s_1$, one calculates that:
 
\begin{tabular}{|c|c|c|}
\hline
$\sigma_w$ & $\sigma_1 \star_\af \sigma_w$ & $\sigma_2 \star_\af \sigma_w$\\
\hline
$\sigma_1$ & $\sigma_{s_2 s_1} + (q_0 + q_1)$ & $\sigma_{s_1 s_2} + \sigma_{s_2 s_1} + q_0$\\
$\sigma_2$ &  $\sigma_{s_1 s_2} + \sigma_{s_2 s_1} + q_0$ & $\sigma_{s_1 s_2} + (q_0 + q_2)$\\
$\sigma_{s_1 s_2}$ & $\sigma_{w_0} + q_0 \sigma_2 - q_0 \sigma_1$ & $(q_2 - q_0) \sigma_1 + q_0 \sigma_2$\\
$\sigma_{s_2 s_1}$ & $q_0 \sigma_1 + (q_1 - q_0) \sigma_2$ & $\sigma_{w_0} + q_0\sigma_1 - q_0 \sigma_2$\\
$\sigma_{w_0}$ & $q_0 \sigma_{s_2 s_1} + (q_0+ q_1)\sigma_{s_1 s_2}+q_2(q_1 - q_0)$ & $q_0 \sigma_{s_1 s_2} + (q_0+ q_2)\sigma_{s_2 s_1}+q_1(q_2 - q_0)$\\
\hline
\end{tabular} 
 
Then one finds the ``affine quantum Schubert polynomials":
\[ \begin{split} \sigma_{s_1 s_2} = \sigma_2 \star_\af \sigma_2 - (q_0 + q_2);  \quad \sigma_{s_2 s_1} = \sigma_1  \star_\af & \sigma_1 - (q_0 + q_1); \\ \sigma_{w_0} =  \sigma_1 \star_\af \sigma_2 \star_\af \sigma_2 - q_2 \sigma_1 - q_0 \sigma_2. \end{split} \]
 
Note that there is a symmetry obtained by exchanging $1 \leftrightarrow 2$ and keeping $0$ fixed. (This follows from the corresponding symmetry of the affine Dynkin diagram of type $A_2^{(1)}$.) Therefore to finish the multiplication table it suffices to calculate $\sigma_{s_1 s_2} \star_\af \sigma_{s_1 s_2}, \sigma_{s_1 s_2} \star_\af \sigma_{s_2 s_1}, \sigma_{s_1 s_2} \star_\af \sigma_{w_0}$ and $\sigma_{w_0} \star_\af \sigma_{w_0}$. We obtain: \[ \sigma_{w_0} \star_\af \sigma_{w_0} = q_1 (q_2 - q_0) \sigma_{s_2 s_1} + q_2(q_1 - q_0) \sigma_{s_1 s_2} + 3 q_0 q_1 q_2 \]

\begin{tabular}{|c|c|}
\hline
$\sigma_w$ & $\sigma_{s_1 s_2} \star_\af \sigma_w$ \\
\hline
$\sigma_{s_1 s_2}$ & $(q_2   -   q_0) \sigma_{s_2 s_1} - q_0 \sigma_{s_1 s_2} + 2 q_0 q_2$ \\
$\sigma_{s_2 s_1}$ & $q_0 \sigma_{s_1 s_2} + q_0 \sigma_{s_2 s_1}  + q_1 q_2 +q_0 (q_1 + q_2)$\\
$\sigma_{w_0}$ & $2 q_0 q_2 \sigma_1 + (q_1 q_2 - q_0 q_1 - q_0 q_2) \sigma_2$\\
\hline
\end{tabular}

\end{document}